%% file: main-v5-10.tex
\documentclass[12pt,oneside]{article}
\input{preamble-v5-1}

\begin{document}
\raggedright

\begin{center}
{\Large \noindent
Homotopy equivalence of Grassmannians and MacPhersonians \\ 
in rank 3}

\bigskip

{\large \noindent
Michael~Gene~Dobbins
}


\begin{minipage}{0.85\textwidth}
\raggedright \footnotesize \singlespacing
\noindent
Department of Mathematical Sciences, Binghamton University (SUNY), Binghamton, \newline New York, USA. 
\texttt{mdobbins@binghamton.edu} 
\end{minipage}

\end{center}

\bigskip

\begin{abstract}
We confirm a long standing conjecture in the case of rank 3 that MacPhersonians are homotopy equivalent to Grassmannians. 
\end{abstract}

\input{intro-v5-10}

\input{tools-v5-9}

\input{param-v5-9}

\input{interp-v5-9}

\input{squish-v5-9}

\input{hood-v5-9}

\input{open-v5-7}

\input{crush-v5-9}

\input{mox-v5-9}

\input{fol-v5-9}

\input{zone-v5-9}

\input{ebbz-v5-7}

\input{step-v5-8}

\bibliographystyle{plain}
\bibliography{macp}

\end{document}

%% file: preamble-v5-1.tex
\usepackage{setspace}
\usepackage[margin=1in]{geometry}
\usepackage[colorlinks=true,citecolor=blue]{hyperref}
\usepackage[mathscr]{euscript}

\usepackage{amsmath}%
\usepackage{amsfonts}%
\usepackage{amsthm}
\usepackage{amssymb}%
\usepackage{graphicx}
\usepackage{units}
\usepackage{verbatim}
\usepackage{enumerate}
\usepackage{tikz}
\usepackage{mathtools}
\usepackage{needspace}
\usepackage[inline]{enumitem}
\usepackage{placeins}

\usepackage{parskip}

\begingroup
    \makeatletter
    \@for\theoremstyle:=definition,remark,plain\do{%
        \expandafter\g@addto@macro\csname th@\theoremstyle\endcsname{%
            \addtolength\thm@preskip\parskip
            }%
        }
\endgroup

\newtheorem{theorem}{Theorem}[section]
\newtheorem{lemma}[theorem]{Lemma}
\newtheorem{corollary}[theorem]{Corollary}

\newtheorem{claim}[theorem]{Claim}

\theoremstyle{remark}
\newtheorem{definition}[theorem]{Definition}
\newtheorem{remark}[theorem]{Remark}
\newtheorem{justification}[theorem]{Justification}

\usepackage{xspace}
\newcommand{\newcmd}[3]{\newcommand{#1}[#2]{#3\xspace}}
\newcmd{\df}{1}{\textbf{\textit{\color{cyan!10!black} #1}}}

\newcommand{\im}{\mathfrak{i}}

\newcommand{\mc}{\mathcal}

\newcommand{\mb}{\mathbb}
\newcommand{\mr}{\mathrm}

\newcommand{\eps}{\varepsilon}
\renewcommand{\phi}{\varphi}

\DeclareMathOperator{\id}{id}
\DeclareMathOperator{\rest}{rest}
\DeclareMathOperator{\dist}{dist}
\DeclareMathOperator{\cl}{cl}
\DeclareMathOperator{\pre}{pre}
\newcommand{\pto}{\rightharpoonup}

\DeclareMathOperator{\crr}{cr}
\DeclareMathOperator{\oc}{oc} 

\DeclareMathOperator{\orth}{O}
\DeclareMathOperator{\sorth}{SO}
\DeclareMathOperator{\sphere}{\mathbf{S}}
\DeclareMathOperator{\disk}{\mathbf{D}}

\DeclareMathOperator{\mcp}{MacP}
\DeclareMathOperator{\omcp}{\widetilde{MacP}}
\DeclareMathOperator{\g}{G}
\DeclareMathOperator{\og}{\widetilde{G}}
\DeclareMathOperator{\stief}{V}
\DeclareMathOperator{\psv}{PsV}
\DeclareMathOperator{\psg}{PsG}
\DeclareMathOperator{\psog}{Ps\widetilde{G}}

\DeclareMathOperator{\sign}{sign}
\DeclareMathOperator{\ot}{ot}
\DeclareMathOperator{\om}{om}
\DeclareMathOperator{\cell}{cell}
\DeclareMathOperator{\cov}{cov}
\DeclareMathOperator{\csph}{csph}
\DeclareMathOperator{\Cov}{\widehat{cov}}
\DeclareMathOperator{\Csph}{\widehat{csph}}
\DeclareMathOperator{\wt}{wt} 
\DeclareMathOperator{\chiro}{chi}

\DeclareMathOperator{\upa}{upa}
\DeclareMathOperator{\maxcov}{maxcov}
\DeclareMathOperator{\supp}{supp}

\DeclareMathOperator{\van}{van}
\DeclareMathOperator{\vanrad}{vanrad}
\DeclareMathOperator{\inrad}{inrad}
\DeclareMathOperator{\minbig}{minbig}
\DeclareMathOperator{\maxlit}{maxlit}
\DeclareMathOperator{\hood}{hood}

\DeclareMathOperator{\cp}{C_p}
\DeclareMathOperator{\graph}{graph}

\DeclareMathOperator{\ext}{ext}
\DeclareMathOperator{\Ext}{Ext}

\DeclareMathOperator{\tailsup}{tailsup}
\DeclareMathOperator{\tailinf}{tailinf}

\DeclareMathOperator{\param}{param}
\DeclareMathOperator{\area}{area}

\DeclareMathOperator{\crush}{crush}
\DeclareMathOperator{\ebb}{ebb}
\DeclareMathOperator{\lbr}{lbr}

\DeclareMathOperator{\pref}{pref}
\DeclareMathOperator{\facet}{facet}
\DeclareMathOperator{\tight}{tight}
\DeclareMathOperator{\zone}{zone}
\DeclareMathOperator{\inner}{inner}
\DeclareMathOperator{\wider}{wider}
\DeclareMathOperator{\acc}{acc}
\DeclareMathOperator{\step}{step}

\DeclareMathOperator{\interp}{interp}

\DeclareMathOperator{\mox}{mox}
\DeclareMathOperator{\fol}{fol}

\DeclareMathOperator{\el}{el}
\DeclareMathOperator{\len}{len}

\DeclareMathOperator{\pri}{pre} 
\DeclareMathOperator{\DE}{ext}  
\DeclareMathOperator{\lsp}{fol}  

\DeclareMathOperator{\pstief}{PsV} 
\DeclareMathOperator{\sep}{sep} 
\DeclareMathOperator{\projplane}{\mathbf{P}^2}

%% file: intro-v5-10.tex
\section{Introduction}


\subsection{Motivation and History}

Vector configurations, affine points sets, and linear subspaces of a real vector space are closely related and ubiquitous in math, but they are uncountable objects, 
so we often need to use some finite representation, which sign covector sets provide. 
The sign covector set of a subspace of $\mb{R}^n$ records all possible sign patterns that the coordinates of a vector in that subspace can have. 
Unfortunately, given such data, it is computationally intractable to determine if there exists a subspace with the given data as its sign covector set \cite{shor1991stretchability}.  
Oriented matroids are defined by natural combinatorial axioms that are necessary but not sufficient for such sign data to arise from a linear subspace \cite{ziegler2024oriented,bjorner2000oriented}.  

Many geometric questions can be reformulated as a statement about the sign covector set of an appropriately defined linear subspace, 
and as such, we can ask analogous questions for oriented matroids. 
However, since oriented matroids are not always realizable, 
what is true for linear subspaces might not always hold for oriented matroids. 
A long standing open question is whether oriented matroids and linear subspaces have the same global topology. 
By the global topology of linear subspaces, 
we mean the homotopy type of the $(k,n)$ real Grassmannian, 
which is the space of all $k$-dimensional subspaces of $\mb{R}^n$. 
Oriented matroids have a natural partial order, and the $(k,n)$ geometric MacPhersonian is the geometric realization of the order complex of rank $k$ oriented matroids on index set $[n]$. 
More precisely, the question is whether each MacPhersonian is homotopy equivalent to the corresponding Grassmannian.  
Nicolai Mnëv and Günter Ziegler stated this as a folklore conjecture that originated from the work of Robert MacPherson \cite{mnev1993combinatorial}. 
Here we show that this conjecture holds in rank 3, Theorem \ref{theoremMacPG}.

Interest in the conjecture grew from the work of Gelfand and MacPherson 
on a combinatorial formula for the rational Pontrjagin cohomology classes of a smooth manifold.   
This is in the spirit of the Euler characteristic as a formula for the Euler class, but is much more technically involved.  
The formula for Pontrjagin classes uses combinatorial differential manifolds, 
which are combinatorial analogs of smooth manifolds where the tangent fibers are replaced by oriented matroids. 
\cite{gelfand1992combinatorial,gaifullin2010configuration,abawonse2022gelfand}.  

A further source of interest in the conjecture comes from matroid bundles. 
Matroid bundles are the combinatorial analog of vector bundles 
and generalize combinatorial differential manifolds much like 
vector bundles generalize tangent bundles. 
Matroid bundles are defined by an order preserving map from a poset, the base space,  
to the MacPhersonian. 
This corresponds to the classifying map from the base space of a vector bundle to the Grassmannian. 
Moreover, this correspondence gives a functor from the category of vector bundles to the category of matroid bundles \cite{anderson1999matroid}.  
Our hope is to use matroid bundles to develop data structures and algorithms for working with vector bundles. 
For this to be effective, we would want the classifying space of matroid bundles, the MacPhersonian, 
to have the same homotopy type as that of vector bundles, the Grassmannian.

Eric Babson showed that the MacPhersonian conjecture holds in rank 2 \cite{babson1993combinatorial}, 
and Olakunle Abawonse gave another proof that showed these spaces are actually homeomorphic in rank 2 \cite{abawonse2023triangulations}. 
An erroneous proof of the full conjecture was published and later retracted \cite{biss2003homotopy,biss2009erratum}. 
Analogous versions of this conjecture have been shown to fail for broader combinatorial analogs of Grassmannians \cite{mnev1993combinatorial,liu2017counterexample}.

The rank 3 case is particularly significant since Mnëv's universality theorem has been a chief 
reasons to doubt the MacPhersonian conjecture. 
There is a natural map from the Grassmannian to the MacPhersonian defined by sending each subspace to the oriented matroid defined by its sign covector set.   
An earlier conjecture of Gerhard Ringel posited that the fibers of this map, which are the realization spaces of oriented matroids, are connected \cite{ringel1956teilungen}.
Mnëv showed not only that this fails, but the realization space can have the homotopy type of any semialgebraic set, 
and moreover, this holds even in rank 3 \cite{mnev1985varieties}.  
Theorem \ref{theoremMacPG} indicates that Mnëv's universality theorem no longer serves as evidence against the MacPhersonian conjecture.

\subsection{Acknowledgments}

The author would like to thank 
Laura Anderson, Seunghun Lee, and Andreas Holmsen 
for helpful discussions and insights.
A significant part of this work was done while the author was visiting KAIST 
with support from KAIST Institute of Science-X (KAI-X).

\subsection{Main theorem}

Here we state the main theorem and give a minimal description of the objects involved. 
More detailed definitions will be given in Subsection \ref{subsectionOMCov}.


The \df{real Grassmannian} $\g_{k,n}$ for $k\leq n$ is the space of all $k$-dimensional linear subspaces of $\mb{R}^n$. 
Up to homeomorphism, 
this can be equivalently defined as the quotient space of all $n$-element spanning vector configurations in $\mb{R}^k$ 
modulo the general linear group.  
The \df{oriented real Grassmannian} $\g_{k,n}$ is the space of all oriented $k$-dimensional linear subspaces of $\mb{R}^n$, and is homeomorphic to the quotient space of all $n$-element spanning vector configurations in $\mb{R}^k$ 
modulo the special linear group.  
In both cases the correspondence between the two definitions is given by associating the columns of a full rank $k\times n$ matrix with its row space.

The \df{order type} $\ot(V)$ of a vector configuration $V = \{v_i\in \mb{R}^k : i\in I\}$ on index set $I$ is 
the map given by the sign of the determinant of each $k$-tuple 
\[
\ot(V):I^k\to\{0,+,-\}, \quad 
\ot(V,i_1,\dots,i_k) = \sign(\det(v_{i_1},\dots,v_{i_k})). 
\]
That is, $\ot(V)$ gives the orientation of each basis and is 0 for nonbases. 
A \df{rank} $k$ \df{chirotope} on ground set $I$ is a map $\chi : I^k \to \{0,+,-\}$ that satisfies the chirotope axioms \cite[Section~3.5]{bjorner2000oriented}, 
which are necessary but not sufficient for $\chi$ to be the order type of a spanning vector configuration.  
Chirotopes are widely regarded as the combinatorial analog of a vector configuration in an oriented vector space. 

We define a partial order $(\leq_\mr{v})$ on $\{0,+,-\}$ where $0<_\mr{v}(+)$, and $0<_\mr{v}(-)$, and $(+)$ and $(-)$ are incomparable. 
The \df{weak order} on chirotopes is the partial order where $\chi_0 \leq_\mr{w} \chi_1$ 
when $\chi_0(i_1,\dots,i_k) \leq_\mr{v} \chi_1(i_1,\dots,i_k)$ for all inputs $i_j$. 
This is the combinatorial analog of ordering vector configurations by degeneracy in the sense that 
$\chi_0 \leq_\mr{w} \chi_1$ when each basis of $\chi_1$ is either a basis in $\chi_0$ with the same orientation 
or is a dependent set in $\chi_0$. 
The \df{oriented MacPhersonian} $\omcp_{3,n}$ is the poset of all rank $k$ chirotopes on $[n]=\{1,\dots,n\}$ with the weak order, 
and the \df{geometric oriented MacPhersonian} $\|\omcp_{3,n}\|$ is the geometric realization of the order complex of $\omcp_{3,n}$. 

Oriented matroids are the discrete analogs of vector configurations in a vector space that lacks a specified orientation, and can be defined in several equivalent ways.  Here we define an oriented matroid to be a primitive object $\mc{M}$ with an associated pair of chirotopes $\chiro(\mc{M}) = \{\chi,-\chi\}$ of opposite sign, 
and we let $\mc{M} = \mc{N}$ when $\chiro(\mc{M}) = \chiro(\mc{N})$. 
The weak order on chirotopes induces a partial order on oriented matroids, and we define the (unoriented) MacPhersonian $\mcp_{k,n}$ 
and geometric (unoriented) MacPhersonian $\|\mcp_{k,n}\|$ analogously.

\begin{theorem}
\label{theoremMacPG}
The rank 3 geometric MacPhersonian 
$\|\mcp_{3,n}\|$ is homotopy equivalent to the corresponding real Grassmannian $\g_{3,n}$ for $n\leq\infty$. 
Also, the rank 3 oriented geometric MacPhersonian  
$\|\omcp_{3,n}\|$ is homotopy equivalent to the corresponding oriented real Grassmannian $\og_{3,n}$ 
for $n\leq\infty$.
\end{theorem}

\subsection{Outline}

In the predecessor to this paper, the author introduced spaces of weighted pseudocircle arrangements $\psg_{3,n}$ and showed that these spaces are homotopy equivalent to the corresponding Grassmannians.  
We prove Theorem \ref{theoremMacPG} by showing that $\psg_{3,n}$ is also homotopy equivalent to the corresponding MacPhersonian. 
We focus mainly on the unoriented case, as the oriented case follows by the same arguments with minor adjustments. 

To obtain the desired homotopy equivalence, we construct an open cover of $\psg_{3,n}$ in Section \ref{sectionHood} with nerve complex isomorphic to the order complex of the MacPhersonian  
and apply the nerve theorem \cite[Corollary 4G.3]{hatcher2002algebraic}. 
This cover consists of an open neighborhood $\frac{\hood}{\orth_3}(\mc{M}) \subset \psg_{3,n}$ 
of the subspace of topological reprepresentations of $\mc{M}$ 
for each oriented matroid $\mc{M}\in\mcp_{3,n}$. 
To use the nerve theorem, we need the following three ingredients.  
First, we show in Subsection \ref{subsectionHoodDisjoint} that if $\mc{M}$ and $\mc{N}$ are incomparable, then the associated neighborhoods are disjoint.  
Second, we show in Section \ref{sectionOpen} that $\frac{\hood}{\orth_3}(\mc{M})$ is open. 
Third, we show in Section \ref{sectionCrush} for each finite chain $\mc{C} \subset \mcp_{3,n}$, that the intersection 
of the associated neighborhoods 
$\frac{\hood}{\orth_3}(\mc{C})=\bigcap\{\frac{\hood}{\orth_3}(\mc{M}):\mc{M}\in\mc{C}\}$ 
is contractible by constructing a deformation retraction 
$\crush(\mc{C},\Omega)$ from $\frac{\hood}{\orth_3}(\mc{C})$ to a point. 
The construction of $\crush(\mc{C},\Omega)$ involves subdividing the sphere into zones in Section \ref{sectionZone} so that a topological representation in $\frac{\hood}{\orth_3}(\mc{C})$ has a simpler form in each zone, and then constructing appropriate deformations on the zones in Section \ref{sectionEbbZ}.  

In Section \ref{sectionTools}, we present tools that will be needed for the rest of the paper. 
This includes a metric on partial functions in Subsection \ref{subsectionPartialMapMetric}, and a canonical parametrization for paths in Subsection \ref{subsectionParam}, 
and a description of how a conformal parametrization (i.e., Riemann mapping) of a Jordan domain behaves 
as the region converges to a path in Subsection \ref{subsectionSquish}. 
In Section \ref{sectionMox}, we define a canonical parametrization of a region with a system of paths that makes the paths x-monotone, which is used to construct the zones in Section \ref{sectionZone} as well as the deformations in Section \ref{sectionEbbZ}.  
In the rest of the introduction, we give basic definitions. 

The paper is largely structured in a top-down fashion. 
For instance, 
we state three lemmas in Section \ref{sectionHood} needed for the proof of Theorem \ref{theoremMacPG}, 
and then prove the theorem. 
Those lemmas are then proven later in the paper, and we continue like this, 
stating a new lemma and then using it to prove a previous lemma, until we stop needing new lemmas.

\subsection{Oriented matroids and covectors}
\label{subsectionOMCov}

Here we define additional combinatorial data associated to oriented matroids. 
%
A \df{signed subset} of a set $I$ is a map $\sigma \in \{0,+,-\}^I$. 
We call $\sigma$ a signed set when $I$ is understood.  
We let $\sigma^0 = \{i:\sigma(i)=0\}$ and we define $\sigma^+$ and $\sigma^-$ analogously. 
We may denote the all zero singed set by $0$. 

Given a rank $k$ oriented matroid $\mc{M}$ on index set $I$ with chirotope $\chi \in \chiro(\mc{M})$, 
then $\sigma = \chi(i_1,\dots,i_{k-1})$ is the signed set 
where $\sigma(x) = \chi(i_1,\dots,i_{k-1},x)$.  
The \df{cocircuit set} of $\mc{M}$ is the collection of signed sets 
\[
\mc{C}^* = \{\chi(i_1,\dots,i_{k-1}) : \chi\in\chiro(\mc{M}),i_i\in I\}.
\] 
The poset $\{0,+,-\}^I$ ordered by $(\leq_\mr{v})$ together with a top element $\top$ 
forms a lattice, and the \df{topped covector sphere} $\Csph(\mc{M})$ of $\mc{M}$ is the 
join-semilattice generated by $\mc{C}^*$.  We also associate to $\mc{M}$ the sets 
\begin{align*}
\csph(\mc{M}) 
&= \Csph(\mc{M})\setminus\{\top\} 
&& \text{covector sphere}, \\ 
\cov(\mc{M}) 
&= \csph(\mc{M})\cup\{0\} 
&& \text{covector set}, \\ 
\Cov(\mc{M}) 
&= \cov(\mc{M})\cup\{\top\} 
&& \text{topped covector set}. 
\end{align*} 
In the case where $\mc{M}$ has rank 3, then 
$\csph(\mc{M})$ is a graded poset with 3 ranks of elements.  
Here we will refer to the elements in the top rank as \df{facet covectors} or \df{topes}, and the elements in the next rank down as \df{edge covectors}, and the elements in the bottom rank as \df{vertex covectors} or \df{cocircuits}.

We may equivalently define an oriented matroid as a primitive object with an associated covector set $\cov(\mc{M})$ that satisfies the vector axioms of oriented matroids \cite[Section 3.7]{bjorner2000oriented}, 
and let $\mc{M}=\mc{N}$ when $\cov(\mc{M})=\cov(\mc{N})$. 
By \df{primitive object}, we mean an object that is not composed of other objects.  For example, numbers are generally treated as primitive objects, although 3 is sometimes iconoclastically defined for convenience in some contexts as the set $3=\{0,1,2\}$, which is not a primitive object. 
Here we define oriented matroids as primitive objects rather than as a pair of oppositely signed chirotopes or as its covector set to avoid conflict with other ways to define oriented matroids, such as by the set of vectors of $\mc{M}$.
This also provides some notational benefits similar to that provided by encapsulation in object oriented programming.


An element $i\in I$ is a \df{loop} of $\mc{M}$ 
when $\sigma(i)=0$ for all $\sigma \in \cov(\mc{M})$, 
and the \df{support} $\supp(\mc{M})$ of $\mc{M}$ is the set of nonloops. 
A set $B = \{i_1,\dots,i_k\} \subset I$ is a \df{basis} for a chirotope $\chi$ when 
$\chi(i_1,\dots,i_k) \neq 0$, 
and $J \subset I$ is an \df{independent} set when $J$ is a subset of a basis, 
and likewise for the oriented matroid $\mc{M}$ with $\chi \in \chiro(\mc{M})$.

\subsection{Pseudocircle arrangements}
\label{subsectionPseudocircle}

Jim Lawrence showed that $\|\csph(\mc{M})\|$ is homeomorphic to a $k$-sphere for every rank $k$ oriented matroid $\mc{M}$ \cite{folkman1978oriented}. 
This provides a topological representation for oriented matroids by pseudosphere arrangements, 
which are pseudocircle arrangements in rank 3.  
A \df{pseudocircle arrangement} on index set $I$ is a collection $A = \{S_i : i \in I\}$ of simple closed curves $S_i=S_i(A)$,
which we call pseudocircles, 
such that each pair $S_i,S_j$ either coincide or intersect in exactly 2 points, in which case any other 
pseudocircle $S_k$ either contains $S_i\cap S_j$ or separates the points $S_i\cap S_j$.  
Additionally, we define an \df{orientation} on each pseudocircle by specifying a postive and negative connected component of $\sphere^2 \setminus S_i$, which we call \df{pseudohemispheres} and denote by $S_i^+$ and $S_i^-$. 
We say $A$ is a \df{weighted pseudocircle arrangement} when each pseudocircle $S_i$ has an associated positive weight, which we denote by $\wt_i=\wt_i(A)$. 
We also allow the trivial pseudocircle where $S_i = \sphere^2$, and $S_i^+ = S_i^- = \emptyset$, and $\wt_i=0$, which we denote by $S_i = 0$. 
For a signed set $\sigma$ we let $S_i^\sigma = S_i^{\sigma(i)}$. 

The \df{pseudolinear Stiefel manifold} $\psv_{3,I}$ is the set of all spanning weighted pseudocircle arrangments indexed by $I$, 
and $\psv_{3,n} = \psv_{3,[n]}$. 
We endow $\psv_{3,I}$ with a metric defined by the maximum Fréchet distance between corresponding pseudocircles scaled by their associated weights.  That is, 
\[
\dist(A,\widetilde A) 
= \max_{i\in I} \inf_{\phi,\widetilde\phi} \sup_{x\in\sphere^1} \|\wt_i\phi(x) - \widetilde\wt_i\widetilde\phi(x)\|
\]
where 
$\phi : \sphere^1 \to S_i$ is a homeomorphisms directed so that $S_i^+$ is on the left 
and analogously for $\widetilde\phi$, and $\widetilde\wt_i = \wt_i(\widetilde A)$. 
For more about Fréchet distance and additional notation,  
see Subsection \ref{subsectionDef}.

We define the left and right actions $(*)$ of $\hom(\sphere^2)$ on $\psv_{3,I}$. 
These actions do not change weights. 
The \df{left action} is given by 
$S_i(\phi*A) = \phi(S_i(A))$ 
and 
$S_i^+(\phi*A) = \phi(S_i^+(A))$ 
and similarly for $S_i^-$. 
For the \df{right action}, 
let $\theta_i(A):\sphere^2 \to \{0,+,-\}$ 
where $\theta_i(A,x) = (+)$ for $x\in S_i^+$ 
and similarly for $S_i^-$ and $S_i$. 
Then, 
$A*\phi$ is the pseudocircle arrangement where 
$\theta_i(A*\phi) = \theta_i(A)\circ\phi$. 
We also define $\phi*A$ for a continuous map $\phi : \sphere^2 \to \sphere^2$ by  
\[
S_i^+(\phi*A) = \{x: \phi^{-1}(x)\subseteq S_i^+(A) \},
\]
with $S_i^-$ defined analogously, 
and $S_i$ is the compliment of $S_i^+\cup S_i^-$; 
however, $\phi*A$ is not always guaranteed to be a pseudocircle arrangement is this case.

\begin{remark}
$\phi^{-1}*A=A*\phi$ since 
\[
S_i^+(A*\phi) 
= \{x : [\theta_i(A)\circ\phi](x) = (+)\}
= \{\phi^{-1}(y) : \theta_i(A,y) = (+)\}
= \phi^{-1}*S^+.
\]
\end{remark}
The \df{pseudolinear Grassmannian manifold} and \df{oriented pseudolinear Grassmannian manifold} are 
the quotient spaces 
\[
\psg_{3,n} = \psv_{3,n}/\orth_3 \quad 
\text{and} \quad 
\psog_{3,n} = \psv_{3,n}/\sorth_3.
\]

To define the $\psv$-\df{realization space} $\psv(\mc{M})$ of an oriented matroid,    
let 
\begin{align*}
\sign(A,x,i) 
&= \begin{cases} 
0 & x\in S_i \\
+ & x\in S_i^+ \\
- & x\in S_i^-, 
\end{cases} \\
\csph(A) 
&= \{\sign(A,x) : x \in \sphere^2\}, \\
\psv(\mc{M}) 
&= \{A\in\psv_{3,n} : \csph(A) = \csph(\mc{M})\}, \\
\psg(\mc{M}) 
&= \psv(\mc{M})/\orth_3. 
\end{align*}
Let $\om(A)$ be the oriented matroid with covector sphere $\csph(A)$  
and associate to $A$ the corresponding oriented matroid data such as $\cov(A)$, 
let $A \geq_\mr{w} \mc{M}$ denote $\om(A) \geq_\mr{w} \mc{M}$, and 
\[
\cell(A,\sigma) 
= \{x \in \sphere^2 : \sign(A,x) = \sigma\}. 
\]

To define the \df{pseudolinear realization space} $\psv(\chi)$ of a chirotope,    
let $\ot(A) : I^k \to \{0,+,-\}$ by 
$\ot(A,i_1,i_2,i_3) = 0$ unless 
$\cov(S_{i_1},S_{i_2},S_{i_3}) = \{0,+,-\}^3$
in which case $\ot(A,i_1,i_2,i_3) = (+)$ if $S_{i_1},S_{i_1},S_{i_1}$ appear in counter-clockwise order 
around the cell $S_{i_1}^+\cap S_{i_2}^+\cap S_{i_3}^+$;
otherwise $\ot(A,i_1,i_2,i_3) = (-)$.  
Let 
\begin{align*}
\psv(\chi) 
&= \{A\in\psv_{3,n} : \ot(A) = \chi\}, \\ 
\psog(\chi) 
&= \psv(\chi) /\sorth_3. 
\end{align*}
In the predeccessor to this paper, the author showed the following.

\begin{theorem}[{\cite[Theorem 2.7.2]{dobbins2021grassmannians}}]
\label{theoremCrushPsVM}
Given a rank 3 oriented matroid $\mc{M}\in \psv_{3,n}$ and a pseudolinear realization $\Omega \in \psv(\mc{M})$, 
there is a strong $\orth_3$-equivariant deformation retraction $\crush(\Omega)$ from the pseudolinear realization space $\psv(\mc{M})$ to the $\orth_3$-orbit of $\Omega$.
\end{theorem}

\begin{theorem}[{\cite[Theorem 2.7.1]{dobbins2021grassmannians}}]
\label{theoremPsV}
There is a strong equivariant deformation retraction from the pseudolinear Stiefel manifold $\psv_{3,n}$ to the corresponding real Stiefel manifold $\stief_{3,n}$. 
\end{theorem}

We say $A$ is \df{symmetric} when $\sign(A,-x)=-\sign(A,x)$, or equivalently, when each pseudocircle of $A$ is an antipodally symmetric curve. 

\begin{remark}
The author showed that the space of homeomorphisms of the projective plane strongly deformation retracts to the $\sorth_3$ using curvature flow \cite{dobbins2021strong,dobbins2023continuous}.   
Consequently, Theorems \ref{theoremCrushPsVM} and \ref{theoremPsV} also hold for the subspace of symmetric pseudocircle arrangments.
This follows by the same argument as in \cite{dobbins2021grassmannians}, except in one place. 
The proofs of Theorems \ref{theoremCrushPsVM} and \ref{theoremPsV} use a theorem of Kneser that the group of homeomorphisms of the 2-sphere deformation retracts to the orthogonal group, which we can replace by the analogous deformation for the projective plane. Similarly, the author has shown this for the subspace of nullity preserving homeomorphisms, 
and consequently we have analogs of Theorems \ref{theoremCrushPsVM} and \ref{theoremPsV} for the subspace of measure 0 pseudocircles, 
and we can additionally require symmetry or not.
Hence, each of these spaces is homotopic to the corresponding Grassmannian and MacPhersonian, 
and can serve as a classifying space for vector bundles. 
\end{remark}

\subsection{Basic definitions and notation}
\label{subsectionDef}

Many of the definitions and notation here are fairly standard, but some are not, 
and less common or nonstandard definitions are mostly recalled later when they are used.  
This is intended as a convenient place for the reader to refer back to later.  

Let $[n]$ or 
$[n]_\mb{N}$ denote the set $\{1,\dots,n\}$.  
Let $\disk$ denote the unit disk in $\mb{R}^2$, 
and $\sphere^d$ denote the unit $d$-sphere in $\mb{R}^{d+1}$.    
Let $\hom(\sphere^d)$ denote the space of homeomorphisms of the sphere, 
and let $\hom^+(\sphere^d)$ denote the space of orientation preserving homeomorphisms. 
Let $X\sqcup Y$ denote the disjoint union of $X$ and $Y$. 
For $X \subset U$, 
let $X^\mr{c} = U\setminus X$ denote the compliment of $X$,  
let $\cl(X)$ denote the closure of $X$ as a subset of a topological space $U$, 
and $X^\circ = X \setminus \cl(X^\mr{c})$ denote the interior, 
and let $\partial X = \cl(X)\cap\cl(X^\mr{c})$ denote the boundary of $X$ where the ambient space $U$ is understood from context. 
We may sometimes omit the head in set-builder notation when the meaning is unambiguous, 
so that $\{\Phi(x)\} = \{x:\Phi(x)\}$. 
We denote open and closed intervals in $\mb{R}$ by $(a,b)_\mb{R} = \{a<x<b\}$
and $[a,b]_\mb{R} = \{a\leq x\leq b\}$.

A \df{path} or \df{$1$-cell} is a homeomorphic embedding of a closed interval.  
Note that a single point is not a path, but we will often consider an object $P$ that is a path or a point, which we call a \df{possibly degenerate path} or a $({\leq}1)$-cell, and we say $P$ is degenerate when $P$ is a point. 
A path is \df{directed} when the embedding map is fixed up to an increasing reparameterization.
Equivalently, we specify one endpoint of $P$ as the \df{source} $s$, 
or specify the other endpoint as the \df{terminal} $t$,  
or by a total ordering $(<_P)$ of $P$ given by $x <_P y$ when $x$ separates $y$ from $s$.  
The \df{reverse} of $P$ is the path obtained by swapping the source and terminal. 
A \df{closed curve} is a homeomorphic embedding of a circle, 
and a \df{curve} can be either a path or a closed curve.
A \df{directed closed curve} has a cyclic ordering preserved by the embedding.

A \df{metric disk} is the set of points on $\mb{R}^2$ or $\sphere^2$ that are within some positive distance $r>0$ of a center point $p$, provided this is not all of $\sphere^2$. 
A \df{topological disk} or \df{$2$-cell} is a homeomorphic embedding of a metric disk. 
A \df{$({\leq}2)$}-cell is a $k$-cell for $k\in\{0,1,2\}$. 
Disks are closed unless specified as open. 
A \df{Jordan domain} is an open topological disk $D$, such that $\partial D$ is a closed curve.

Given subsets $S,T\subset X$ of a metric space $(X,\dist)$ and $\delta\geq 0$, 
let 
\[
S\oplus\delta = \{x \in X : \exists s \in S : \dist(x,s) \leq \delta\}
\]
be the set of points that are within distance $\delta$ of some point of $S$. 
The \df{Hausdorff distance} $\dist_\mr{H}(T,S)$ between 
subsets $S,T\subset X$ 
is the infimum among $\delta$ such that $S \subseteq T\oplus\delta$ and $T\subseteq S\oplus\delta$. 

The \df{undirected Fréchet distance} 
between $({\leq}1)$-cells $P_i$ in a metric space is 
\[
\dist_\mr{F}(P_1,P_0) = \inf_{\phi}\sup_{x\in P_0}\dist(\phi(x),x)
\]
among homeomorphisms $\phi:P_0\to P_1$ if $P_0$ and $P_1$ are not degenerate, 
or $\dist_\mr{F}(P_1,P_0) = \dist_\mr{H}(P_1,P_0)$ 
in the case where $P_0$ or $P_1$ is a point. 
The \df{Fréchet distance} 
between directed paths $P_i$ is 
as above, but where $\phi$ is order preserving. 
Fréchet distance and undirected Fréchet distance are defined analogously for closed curves. 

Let $f: X \pto Y$ denote a partial function from $X$ to $Y$, 
and $f(x)=\bot$ indicate that $f(x)$ is not defined. 
The \df{preimage} $\pre(f) = \{f(x)\neq \bot\}$ is the set of $x\in X$ where $f(x)$ is defined.  
We say $f(x)$ \df{varies continuously as $x$ varies} 
when $f$ is sequentially continuous. 
This will be convenient when we do not explicitly name $f$ as a function. 
Let $\rest(S,f)$ denote the \df{restriction} of $f$ to $S \subset X$. 
The \df{restriction} 
$\rest(J,\mc{M})$ 
to a subset $J\subset I$ 
is the oriented matroid with covectors 
\[
\cov(\rest(J,\mc{M})) 
=  
\{\rest(J,\sigma) : \sigma \in \cov(\mc{M})\}  
\quad \text{where} \quad 
\rest(J,\sigma,i) 
= \begin{cases}
\sigma(i) & i\in J \\
\bot & i \not\in J.
\end{cases} 
\]
%
Similarly for a pseudocircle arrangement $A$, the restriction to $J$ is 
\[
\rest(J,A) = (T_1,\dots,T_n) 
\quad \text{where} \quad 
T_i = \begin{cases}
S_i & i\in J \\
\bot & i \not\in J.
\end{cases}
\]
Given an arrangement $A \in \psv_{3,J}$, 
the space of \df{extensions} of $A$ to index set $I$ is 
\[
\Ext(I,A) = \{X \in \psv_{3,I} : \rest(J,X) = A\}. 
\]
Let $\Ext(A) = \Ext([n],A)$. 

A \df{chain} of a poset $P$ is a totally ordered subset of $P$. 
The \df{order complex} $\oc(P)$ of $P$ is the set of all finite chains of $P$ ordered by containment. 
For $c \in P$, let 
\[
P_{\leq c} = \{x \in P : x \leq c\}
\]
and analogously for $\geq,<,>$. 
Given an abstract simplicial complex $\Delta$ on a ground set $E$, 
the geometric realization of $\Delta$ is the geometric simplicial complex 
\[
\|\Delta\| = \left\{ x \in \mb{R}^E : \supp(x) \in \Delta, \sum_{e\in E} x(e) = 1\right\}. 
\]
For a poset $P$, we let $\|P\| = \|\oc(P)\|$ denote the geometric realization of the order complex of $P$.

A \df{conformal} map is a continuous function that preserves both angle and orientation, 
and an \df{isogonal} map preserves angles, but not necessarily orientation. 
An \df{internally} conformal map is a map that is conformal on the interior of its domain, 
and the same goes for internally isogonal.
A \df{hyperbolic geodesic} in the open unit disk $\disk^\circ$ 
is the intersection of a circle with $\disk^\circ$ that meets the boundary $\sphere^1$ at right angles.  A path $g$ in a simply connected open domain $C$ is a hyperbolic geodesic when $g = h(\widetilde g)$ for a conformal embedding $h : \disk^\circ \to C$ and a hyperbolic geodesic $\widetilde g \subset \disk^\circ$.  We may include endpoints or not as convenient. 
We denote the cross-ratio by  
$\crr(w,x;y,z) = \frac{(w-y)(x-z)}{(w-z)(x-y)}$.

%% file: tools-v5-9.tex
\section{Tools}
\label{sectionTools}

\subsection{The partial map metric}
\label{subsectionPartialMapMetric}

Here we extend the sup metric to partial maps. 
Given metric spaces $X,Y$, 
let $\cp(X,Y)$ be the space of continuous partial maps $f:X \pto Y$ 
where $\pre(f)$ is closed. 
We define the \df{partial map metric}
$\dist_{\cp}(g,f)$ on $\cp(X,Y)$ to be the Hausdorff distance between the respective graphs of $f$ and $g$ with respect to the max metric on the product space $X\times Y$. 
We simply write $\dist(g,f)$ for $\dist_{\cp}(g,f)$ when there is no ambiguity about the metric.
Note that this is an extended metric since Hausdorff distance is an extended metric on closed subsets of a metric space, and $\graph(f)$ is closed.  
Equivalently, $\dist(g,f)$ is the infimum among $\delta$ such that 
if $x \in \pre(f)$ then $f(x)\in g(x\oplus\delta)\oplus\delta$ 
and 
if $x \in \pre(g)$ then $g(x)\in f(x\oplus\delta)\oplus\delta$.


\begin{lemma}
\label{lemmaCPDistSup}
If $X$ is compact, then the partial map distance on $\mr{C}(X,Y)$ is topologically equivalent 
to the sup metric. 
\end{lemma}

\begin{proof}
Suppose $\dist(g(x),f(x)) < \eps$ for all $x$. 
Then,  
\[
g(x) \in f(x)\oplus\eps \subseteq f(x\oplus\eps)\oplus\eps  
\]
and $f(x) \in g(x\oplus\eps)\oplus\eps$ for all $x$, 
so $\dist(g,f)\leq\eps$. 
Hence, 
the sup metric ball of radius $\eps$ about $f$ 
is contained in 
the partial map metric ball of radius $\eps$ about $f$. 

For the other direction, consider $\eps>0$. 
Then, $f \in \mr{C}(X,Y)$ is uniformly continuous by the Heine-Cantor Theorem, 
so we can let $\delta>0$ such that 
$f(x\oplus\delta)\subseteq f(x)\oplus\eps/2$ 
for all $x\in X$. 
Let $\eps_1 = \min(\eps/2,\delta)$ and consider $g\in \mr{C}(X,Y)$ such that 
$\dist(g,f)\leq\eps_1$. 
Then, 
\[
g(x) 
\in f(x\oplus\eps_1)\oplus\eps_1 
\subseteq f(x\oplus\delta)\oplus\eps/2 
\subseteq f(x)\oplus\eps.
\]
Hence, the partial map metric ball of radius $\eps_1$ about $f$ 
is contained in the sup metric ball of radius $\eps$ about $f$. 
\end{proof}

\begin{lemma}
\label{lemmaCPDistHDist}
$\dist(f,g) \geq \dist_\mr{H}(\pre(f),\pre(g))$. 
\end{lemma}

\begin{proof}
Let $\dist_\mr{H}(\pre(f),\pre(g))>r_1>r_2$, 
and let us assume there is $x\in\pre(f)\setminus(\pre(g)\oplus r_1)$; 
otherwise swap $f$ and $g$. 
Then, $x\oplus r_2$ is disjoint from $\pre(g)$, 
so $g(x\oplus r_2)\emptyset$, 
so $f(x) \not\in g(x\oplus r_2)\oplus r_2$, 
so $\dist_{\cp}(f,g) \geq r_2$. 
Since this holds for all $r_2 < \dist_\mr{H}(\pre(f),\pre(g))$, 
we have 
$\dist_{\cp}(f,g) \geq \dist_\mr{H}(\pre(f),\pre(g))$.
\end{proof}

Let us see a description of convergence.

\begin{lemma}
\label{lemmaCPDistLimits}
Let $X$ be compact. 
Then, a sequence of maps $f_k\in\cp(X,Y)$ converges to a map in $\cp(X,Y)$ in the partial map metric 
if and only if 
$\pre(f_k)$ converges to a closed set in Hausdorff distance
and $f_k(x_k)$ converges for every convergent sequence $x_k \to x_\infty$ with $x_k\in \pre(f_k)$.  
\end{lemma}

\begin{proof}
For the `if' direction, 
suppose 
$\pri(f_k) \to D$ compact in Hausdorff distance 
and $y_k = f_k(x_k) \to y_\infty$ for each convergent sequence $x_k \to x_\infty$.

Let us first define the map $f_\infty$ that will be the limit of the $f_k$. 
Let $D$ be the preimage of $f_\infty$, 
and let $f_\infty(x_\infty) = y_\infty$ 
where $y_\infty = \lim_{k\to\infty} y_k$ 
and $y_k = f_k(x_k)$ 
for a choice of $x_k \to x_\infty$.  
We know such a choice exists since 
$\pri(f_k) \to D$, so there is a sequence $x_k \in \pre(f_k)$ such that $\dist(x_k,x_\infty) \to 0$.


We claim $f_\infty$ is well-defined. 
Suppose there were some other choice of convergent sequence $\widetilde x_k \to x_\infty$ 
with $\widetilde x_k \in \pre(f_k)$ 
such that $f_k(\widetilde x_k) \to \widetilde y_\infty \neq y_\infty$.  Then, the sequence $x_1,\widetilde x_2,x_3,\widetilde x_4,\dots$ would converge to $x_\infty$, but the sequence of images $f_1(x_1),f_2(\widetilde x_2),\dots$ would not converge, contradicting our hypothesis.
Hence, $f_\infty$ is well-defined.   

Suppose $f_\infty$ were not continuous. 
Then there would be some convergent sequence $w_k \to x_\infty$ with $w_k \in D$ 
such that $f_\infty(w_k)$ is bounded apart from $f_\infty(x_\infty)$. 
That is, $\dist(f_\infty(w_k),f_\infty(x_\infty))>\delta>0$ for all $k$ sufficiently large. 
Since $\pre(f_k) \to D$ in Hausdorff distance, 
we can find a sequence $x_k \in \pri(f_k)$ such that $\dist(x_k,w_k) \to 0$. 
Hence, $x_k \to x_\infty$, so $f_k(x_k) \to f_\infty(x_\infty)$ by definition of $f_\infty$ since $f_\infty$ is well-defined. 
Also, for each $k$ there is a sequence $x_{k,j} \to w_k$ as $j \to \infty$ where $x_{k,j} \in \pri(f_j)$ since $\pre(f_j) \to D$, 
so $f_j(x_{k,j}) \to f_\infty(w_k)$, so 
$\dist(f_j(x_{k,j}),f_\infty(w_k))<\delta/2$ 
and $\dist(x_{k,j},w_k) < 1/k$ 
for $j \geq J_k$ sufficiently large. 

Consider a new sequence $\widetilde x_k = x_{k,j_k}$ where $j_k = \max(k,J_k)$. 
Then, $\dist(f_{j_k}(\widetilde x_k),f_\infty(w_k))<\delta/2$. 
Also, $\dist(\widetilde x_k, w_k) \to 0$ and $\dist(w_k,x_\infty) \to 0$, 
so $\widetilde x_k \to x_\infty$, 
so $f_{j_k}(\widetilde x_k) \to f_\infty(x_\infty)$, 
so $\dist(f_{j_k}(\widetilde x_k),f_\infty(x_\infty)) <\delta/2$ for $k$ sufficiently large, 
so $\dist(f_\infty(w_k),f_\infty(x_\infty))<\delta$, which is a contradiction. 
Thus, $f_\infty$ is continuous. 

For the `if' direction, 
it remains to show that $f_k \to f_\infty$ in the partial map metric. 

Suppose $\graph(f_k) \not\subseteq \graph(f_\infty)\oplus\delta$ for some $\delta>0$ 
infinitely often, and let us restrict to this subsequence. 
Then, there would be some sequence $(x_k,f_k(x_k))$ 
that is distance $\delta>0$ away from $\graph(f_\infty)$.
Let us also restrict to a subsequence where $x_k$ converges to a point $x_\infty$ 
since $X$ is compact, 
and $x_\infty \in \pre(f_\infty)$ and $f_k(x_k) \to f_\infty(x_\infty)$ 
by definition of $f_\infty$, which contradicts our choice of sequence $(x_k,f_k(x_k))$ bounded away from $\graph(f_\infty)$. 

Suppose $\graph(f_\infty) \not\subseteq \graph(f_k)\oplus\delta$ for some $\delta>0$ 
infinitely often, and again restrict to this subsequence. 
Then, there would be some sequence $(w_k,f_\infty(w_k))$ 
that is distance $\delta>0$ away from $\graph(f_k)$.
Since $\pre(f_k) \to \pre(f_\infty)$, we can find $x_k \in \pre(f_k)$ 
such that $\dist(x_k,w_k) \to 0$, 
and since $X$ is compact, we may restrict to a subsequence 
where $x_k \to x_\infty$, 
so $w_k \to w_\infty$, 
and $f_k(x_k) \to f_\infty(x_\infty)$ by definition of $f_\infty$, 
and $f_\infty(w_k) \to f_\infty(x_k)$ since $f_\infty$ is continuous, 
so $\dist(f_k(x_k),f_\infty(w_k)) \to 0$, 
which contradicts contradicts our choice of sequence $(w_k,f_\infty(w_k))$. 
Thus, $f_k \to f_\infty$.

For the `only if' direction, 
suppose $f_k \to f_\infty$. 
Then, $\pre(f_k)\to\pre(f_\infty)$ in Hausdorff distance by Lemma \ref{lemmaCPDistHDist}. 
Consider a convergent sequence $x_k \to x_\infty$ with $x_k \in\pre(f_k)$ and $\eps>0$.  
Since $X$ is compact, $f_\infty$ is uniformly continuous by the Heine-Cantor Theorem, 
so for some $\delta>0$, we have $f_\infty(x\oplus\delta)\subset f_\infty(x)\oplus\eps/2$.  
Let $r = \min(\delta,\eps)/2$. 
Then, 
for $k$ sufficiently large we have $x_k \in x_\infty\oplus\delta/2$ since $x_k \to x_\infty$, 
so $x_k\oplus r \subseteq x_\infty\oplus\delta$, 
and since $f_k \to f_\infty$, we have 
\[
f_k(x_k)
\in f_\infty(x_k\oplus r)\oplus r 
\subset f_\infty(x_\infty\oplus \delta)\oplus \eps/2  
\subset f_\infty(x_\infty)\oplus \eps.  
\]
Thus, $f_k(x_k) \to f_\infty(x_\infty)$. 
\end{proof}

Recall that $f(x)$ \df{varies continuously as $x$ varies} 
when the function $f$ is sequentially continuous. 
This will be convenient when we do not explicitly name $f$ as a function. 

We show that 
operation that we like to use on functions are continuous. 
Moreover, inversion is an isometry. 




\begin{lemma}
\label{lemmaCPDistContinuity}
Let $f\in\cp(X,Y)$ and $g\in\cp(W,X)$ for metric spaces $W,X,Y$, 
and let $W,X$ be compact.
\begin{description}
\item[Inversion]
If $f,\widetilde f \in\cp(X,Y)$ are homeomorphic embeddings 
then $\dist(f,\widetilde f)=\dist(f^{-1},\widetilde f^{-1})$. 
Hence, $f^{-1}$ varies continuously as $f$ varies.
\item[Composition]
$f\circ g$ varies continuously 
as $f$ and $g$ vary 
provided that  $g(W)\cap \pre(f)$ varies continuously, 
and in particular if $g(W)\subseteq \pre(f)$. 
\item[Restriction]
$\rest(D,f)$ varies continuously 
as $f$ and $D$ vary 
provided that $D\cap\pre(f)$ varies continuously in Hausdorff distance, 
and in particular if $D\subset\pre(f)$.  
\item[Gluing] 
If $X = C_1\cup C_2$ is a closed cover, $C_i = \pre(f_i)$, and $f(x) = f_i(x)$ for $x \in C_i$
then $f$ varies continuously as the $f_i$ and $C_i$ vary. 
\item[Partial application]
On a product space $X=X_1\times X_2$, 
the map $f(x_1) \in \cp(X_2,Y)$ varies continuously as $f$ and $x_1$ vary 
provided that $\pre(f(x_1))$ varies continuously,
and in particular if 
$\pre(f) = S_1\times S_2$ factors as a product. 
\end{description}
\end{lemma}

\begin{proof}
Consider convergent sequences $f_k \to f_\infty$ and $g_k \to g_\infty$ as in the hypotheses.

For inversion, 
$\graph(f) = \graph(f^{-1})$ are the same subset of $X\times Y$. 

For continuity of composition, 
consider a convergent sequence $w_k \to w_\infty$ with $w_k \in \pre(f_k)$.  
$g_k(w_k) \to g_\infty(w_\infty)$ by Lemma \ref{lemmaCPDistLimits}, 
so $[f_k\circ g_k](w_k) \to [f_\infty\circ g_\infty](w_\infty)$ converges appropriately, 
so $f_k\circ g_k$ converges appropriately by Lemma \ref{lemmaCPDistLimits}.  
In the case where $g(W)\subseteq \pre(f)$, 
observe that $\pre(f\circ g) = g^{-1}(\pre(f)) = \pre(g)$ 
varies continuously in Hausdorff distance since $g$ varies continuously.

For restriction, 
consider 
$D_k \to D_\infty$ and $x_k\to x_\infty$ as in the hypotheses. 
Then, 
$f_k(x_k)$ converges appropriately for every convergent sequence $x_k \in D_k$, 
so $\rest(D_k,f_k)$ converges appropriately by Lemma \ref{lemmaCPDistLimits}.  

For gluing, 
consider a convergent sequence $x_k\to x_\infty$. 
If $x_k \in C_{1,k}$, then $x_\infty \in C_{1,\infty}$ since $C_{1,k} \to C_{1,\infty}$ in Hausdorff distance, so $f_k(x_k)$ converges appropriately since $f_1$ varies continuously, and similarly if $x_k \in C_{2,k}$. 
Hence, $f_k(x_k)$ converges appropriately since the subsequences with $x_k\in C_{i,k}$ converge appropriately.

For partial application, 
consider convergent sequences $a_k \to a_\infty$ and $b_k \to b_\infty$ in $X_1$ and $X_2$. 
Then, $f_k(a_k,b_k)$ converges appropriately, 
and since this holds for each convergent sequence $b_k \in X_2$, 
we have that $f_k(a_k)\to f_\infty(a_\infty)$ by Lemma \ref{lemmaCPDistLimits}.  
In the case where $\pre(f) = S_1\times S_2$ factors,
since $f_k$ converges appropriately, $\pre(f_k) = S_{1,k}\times S_{2,k}$ must converge appropriately, 
so both $S_{i,k}$ converge appropriately, 
so $\pre(f_k(a_k)) = S_{2,k}$ converges appropriately. 
\end{proof}

For us restriction to a path will be of particular importance. 

\begin{lemma}
\label{lemmaPMFrechetDistance}
If $P \subseteq \pre(f)$ is a path and $f$ is a topological embedding of $\pre(f)$, 
then $f(P)$ varies continuously in Fréchet distance as 
$f$ varies in the partial map metric and $P$ varies in Fréchet distance. 
\end{lemma}

Note that Lemma \ref{lemmaPMFrechetDistance} does not always hold if $f$ is not always an embedding as it varies. 

\begin{proof}
Consider convergent sequences $P_k \to P_\infty$ and $f_k \to f_\infty$. 
Then, there are parameterizations $\gamma_k : [0,1]_\mb{R} \to P_k$ such that 
$\sup \{\dist(\gamma_k(x),\gamma_\infty(x)) : x\in[0,1]\} \to 0$ as $k \to \infty$ by definition of Fréchet distance, 
which means $\gamma_k \to \gamma_\infty$ in the sup metric, 
so $f_k\gamma_k \to f_\infty\gamma_\infty$ in the partial map topology by Lemma \ref{lemmaCPDistContinuity}, 
and therefore in the sup metric as well, 
and 
\[\dist_\mr{F}(f_k(P_k),f_\infty(P_\infty)) \leq \dist_{\sup}(f_k\gamma_k,f_\infty\gamma_\infty)\]
since both $f_k\gamma_k$ and $f_\infty\gamma_\infty$ are homeomorphisms, 
so $f_k(P_k) \to f_\infty(P_\infty)$ in Fréchet distance. 
\end{proof}


In the case of total functions, 
we have the following.

\begin{lemma}
\label{lemmaSupDistPartialApp}
Given $f:X \to Y$ with $X = A\times B$ and $B$ compact,   
if $f(x) \in Y$ varies continuously as $x \in X$ varies,  
then $f(a) \in Y^{B}$ varies continuously in the sup metric as $a \in A$ varies. 
Furthermore, the converse holds provided that $X$ is compact.
\end{lemma}

\begin{remark}
\label{remarkUniformConvergence}
By letting $A = \{\nicefrac1k,0:k\in\mb{N}\}$ and $f_k=f(\nicefrac1k):B \to Y$ with $B$ compact, we have 
$f_k \to f_\infty$ uniformly if and only if $f_k(b_k) \to f_\infty(b_\infty)$ 
for every convergent sequence $b_k \to b_\infty$ 
if and only if $f$ is continuous 
as a special case of Lemma \ref{lemmaSupDistPartialApp}. 
\end{remark}

\begin{proof}
Suppose that $f(x_k) \to f(x_\infty)$ for every convergent sequence $x_k=(a_k,b_k) \to x_\infty\in X$. 
Then, $f(a_k)$ converges appropriately in the partial map metric by Lemma \ref{lemmaCPDistLimits}, 
so $f(a_k)$ converges appropriately in the sup metric by Lemma \ref{lemmaCPDistSup}. 

Alternatively, suppose that $X$ is compact and 
that $f(a_k) \to f(a_\infty)$ in the sup metric 
for every convergent sequence $a_k \to a_\infty \in A$. 
Then, $f(a_k)$ converges appropriately in the partial map metric by Lemma \ref{lemmaCPDistSup}, 
so $f(a_k,b_k)$ converges appropriately  
for every convergent sequence $b_k \to b_\infty$ by Lemma \ref{lemmaCPDistContinuity}. 
Hence, $f(x_k)$ converges appropriately for every convergent sequence $x_k$.
\end{proof}

\subsection{Extension to $\mb{R}$ and tail extrema}

Here we treat convergence for the canonical extension of a function on an interval to $\mb{R}$. 
We also introduce tail extrema, which will be useful for finding a continuously varying monotone upper or lower bound for a given function.

Given a function $f : [a,b]_\mb{R} \to \mb{R}$, 
let $\ext_\mb{R}(f) : \mb{R} \to \mb{R}$ by 
\[
\ext_\mb{R}(f,x) = \begin{cases}
f(a) & x < a \\ 
f(x) & x \in [a,b] \\ 
f(b) & x > b,
\end{cases}
\]
and for an interval $I\subset\mb{R}$, let $\ext_I(f)=\rest(I,\ext_\mb{R}(f))$. 

\begin{lemma}
\label{lemmaExtendedInverse}
$\ext_\mb{R}(f)$ varies continuously in the sup metric as $f$ varies in the partial map metric.
Moreover, 
if $f:\mb{R} \pto \mb{R}$ is a strictly monotonic partial function defined on a closed interval, 
then $\ext_\mb{R}(f^{-1})$ varies continuously. 
\end{lemma}

\begin{proof}
Consider a sequence $f_k \to f_\infty$ that converges in the partial map metric.
Let $[a_k,b_k]=\pre(f_k)$. 
Then, $\pre(f_k) \to \pre(f_\infty)$, so for all $k$ sufficiently large, $\pre(f_k)$ is contained in a closed interval such as $I=\pre(f_\infty)\oplus 1$, which is compact. 
Consider a convergent sequence $x_k \to x_\infty$ in $I$. 
In the case where $x_k \in \pre(f_k)$, then $f_k(x_k) \to f_\infty(x_\infty)$ by Lemma \ref{lemmaCPDistLimits}.
Consider the case where $x_k < a_k$, then $f_k(x_k)=f_k(a_k)\to f_\infty(a_\infty)$ by Lemma \ref{lemmaCPDistLimits} since $a_k \to a_\infty$, and $x_\infty \leq a_\infty$ so $f_\infty(x_\infty)=f_\infty(a_\infty)$. 
This holds similarly in the case where $x_k>b_k$, 
so $f_k(x_k) \to f_\infty(x_\infty)$ for every convergent sequence in $I$, 
so $\ext_I(f_k) \to \ext_I(f_\infty)$ in the partial map metric by Lemma \ref{lemmaCPDistLimits}, 
so $\ext_I(f_k) \to \ext_I(f_\infty)$ in the sup metric by Lemma \ref{lemmaCPDistSup}, 
and $\ext_\mb{R}(f_k)$ on the rest of $\mb{R}$ agrees with $\ext_I(f_k)$ at the endpoints of $I$, 
so $\ext_\mb{R}(f_k) \to \ext_\mb{R}(f_\infty)$ in the sup metric by Lemma \ref{lemmaCPDistSup}.
The second part of the lemma then follows since inversion is isometric by Lemma \ref{lemmaCPDistContinuity}. 
\end{proof}

Later we will use $\ext_\mb{R}$ to find an input $s$ where some map $b_\mr{test} : [0,1] \to \mb{R}$ becomes sufficiently large or small, but $b_\mr{test}$ will not always be guaranteed to be monotonic.  Therefore, we will use $\tailsup$ or $\tailinf$ to replace a function $b_\mr{test}$ with a monotonic function $b$ and use $\ext_\mb{R}(b^{-1})$. 
See Definition \ref{defCrush} for example. 

Let $\tailsup,\tailinf:\mb{R}^{[0,1]} \times [0,1] \to \mb{R}$ by 
\begin{align*}
\tailsup(f,u) &= \sup\{f(x): x\in[u,1]\} \\ 
\tailinf(f,u) &= \inf\{f(x): x\in[u,1]\} \\ 
\end{align*}

\begin{lemma}
\label{lemmaTailsupContinuous}
$\tailsup(f)$ and $\tailinf(f)$ vary continuously as $f$ varies in the sup metric.  
\end{lemma}

\begin{proof}
Suppose $f_k \to f_\infty$ in the sup metric and let $\eps > 0$.
Then, 
$|f_k(x)-f_\infty(x)|<\eps$ for $k$ sufficiently large.  
Since $[u,1]$ is compact, the supremum of $\{f_k(x): x\in[u,1]\}$ is attained at some $x_{k,u} \in [u,1]$, 
so $\tailsup(f_k,u) = f_k(x_{k,u}) \leq f_\infty(x_{k,u}) +\eps \leq \tailsup(f_\infty,u) +\eps$, 
and similarly $\tailsup(f_\infty,u) \leq \tailsup(f_k,u) +\eps$. 
Hence, $|\tailsup(f_k,u)-\tailsup(f_\infty,u)|<\eps$, 
which means that $\tailsup(f_k) \to \tailsup(f_\infty)$ in the sup metric.
\end{proof}

%% file: param-v5-9.tex
\subsection{Parameterization}
\label{subsectionParam}

Here we state some useful properties of conformal maps and we construct a canonical parameterization of paths in the sphere 
that will satisfy the following.

\begin{lemma}
\label{lemmaPathParam}
Given a directed $({\leq}1)$-cell $P\subset\sphere^2$, then 
$\param(P):[0,1] \to P$ satisfies the following.
\begin{enumerate}
\item
\label{itemPathParamHomeo} 
If $P$ is a 1-cell, then $\param(P)$ is a homeomorphism. 
\item
\label{itemPathParamContinuous}
$\param(P)$ varies continuously in the sup-metric 
as $P$ varies in Fréchet distance. 
Likewise, for $\param^{-1}(P) = [\param(P)]^{-1}$ in the partial map metric. 
\item 
\label{itemPathParamEquivariant}
$\param$ is $\orth_3$-equivariant.
\item 
\label{itemPathParamReverse}
$\param(P^-,t) = \param(P,1-t)$ where $P^-$ is the reverse of $P$. 
\end{enumerate}
\end{lemma}

\begin{definition}[parameterization of a path]
\label{defPathParam}
Given a path $P \subset \sphere^2$ directed from $s$ to $t$, 
and a point $y\in P$. 
Let $\param^{-1}(P) : P \to [0,1]_\mb{R}$ by 
$x = \param^{-1}(P,y) = \area(C)/4\pi$ 
where $C$ is the closed region bounded by the hyperbolic geodesic curve $g = g(P,y)$ from $y$ to $y$ through $\sphere^2 \setminus P$ that contains $s$ provided that $y \not\in \{s,t\}$. 
Otherwise, $g(P,y)=y$ for an endpoint, and $\param^{-1}(P,s) = 0$, and $\param^{-1}(P,t)=1$. 
Let $\param(P,x) = y$ provided $P$ is not degenerate, and let $\param(P,x) = P$ for a degenerate path.   
\end{definition}

Extensions of the Riemann mapping theorem by Constantin Carathéodory and Tibor Radó will be vital.  

\begin{theorem}[Carathéodory \cite{caratheodory1913gegenseitige}]
\label{theoremCaratheodory}
A conformal map from the open unit disk to a Jordan domain extends to a homeomorphism of their closures. 
\end{theorem}

\begin{theorem}[Radó \cite{rado1923representation}; see also {\cite[Theorem~2.11]{pommerenke2013boundary}}]
\label{theoremRado}
Given a sequence of conformal maps $h_k$ from the open unit disk to a Jordan domain with $h_k(0)=0$ and $\mr{d}h_k(z)/\mr{d}z > 0$, if there is a parameterization of $\partial h_k(\disk)$ that converges uniformly, then $h_k$ converges uniformly.  
\end{theorem}

\begin{remark}
We can restate Radó's theorem as saying that $h$ varies continuously in the sup metric as $\partial h$ varies in Fréchet distance. 
Also, we can replace the conditions $h_k(0)=0$ and $\mr{d}h_k(z)/\mr{d}z > 0$ 
with the condition that there are 3 points on the unit circle where the image of the continuous extension of $h_k$ converges. 
\end{remark}

We say a path or a simple closed curve $S$ is \df{nonOsgood} when the 2-dimensional Hausdorff measure of $S$ is 0 \cite{osgood1903jordan}. 
We show that the area bounded by a nonOsgood curve is continuous. 
Note that the area of the region bounded by a curve is not continuous in general, as can be seen in the case of a sequence of nonOsgood curves converging to an Osgood curve. 

\begin{lemma}\label{lemmaAreaContinuous}
Let $S \subset \sphere^2$ be a nonOsgood directed closed curve and $C$ be the region to the left of $S$, 
then $\area(C)$ varies continuously as $S$ varies in Fréchet distance. 
\end{lemma}

\begin{proof}
Consider a convergent sequence of such curves $S_k \to S_\infty$ bounding respective regions $C_k$. 
We claim that if $\dist(S_k,S_\infty) < \eps$ for $\eps$ sufficiently small, then the symmetric difference $\Delta_k = (C_k \setminus C_\infty)\cup(C_\infty \setminus C_k)$ is contained in $S_\infty \oplus \eps$.
We may assume $\eps$ is small enough that there exists points $p_0,p_1$ such that 
$(p_0 \oplus \eps) \cap C_\infty = \emptyset$ and
$(p_1 \oplus \eps) \subset C_\infty$.
Since $\dist(S_k,S_\infty) < \eps$ there is a map $f : S_\infty \to S_k$ such that $\|f(x)-x\| < \eps$, so $x$ is always closer to $f(x)$ than to $p_0$ or $p_1$.
Therefore, a homotopy by spherical linear interpolation from $S_k$ to $S_\infty$ stays within $\sphere^2\setminus\{p_0,p_1\}$, 
which implies that $p_0 \not\in C_k$ and $p_1 \in C_k$.
Since this holds for each such pair $p_0,p_1$, 
each point of $\sphere^2 \setminus (S_\infty\oplus\eps)$ stays on the same side of $S_k$ throughout a homotopy to $S_\infty$, and as such is not a point of $\Delta_k$. 
Thus, $\Delta_k \subseteq S_\infty\oplus\eps$.

Consider $\eps_1 > 0$.
Since $S_\infty$ has area 0, there is an open cover $U$ of $S_\infty$ that has area at most $\eps_1$.
Since $\sphere^2\setminus U$ and $S_\infty$ are compact and disjoint, they are bounded apart by some $\eps_2 > 0$ depending on $\eps_1$. 
Since $S_k \to S_\infty$, we have for all $k$ sufficiently large that $\dist(S_k,S_\infty) < \eps_2$, 
so by the claim above, $\Delta_k \subseteq S_\infty \oplus \eps_2$, 
so $\Delta_k \subseteq U$, so $\Delta_k$ has area at most $\eps_1$.
Thus, $\area(\Delta_k) \to 0$ as $k\to \infty$, which implies that $\area(C_k) \to \area(C_\infty)$.
\end{proof}

\begin{claim}
\label{claimParamPathG}
$g$ in Definition \ref{defPathParam} varies continuously in Fréchet distance as $P$ and $y$ vary.
Moreover, if $y_1 <_P y_2$, then $C_1 \subset C_2^\circ$ where $C_i$ are the corresponding bounded regions.   
\end{claim}

\begin{proof}
Let us first assume that $P$ is a path from $0$ to $\infty$ that does not meet 1. 
Let $f(z) = \left(\frac{z+1}{z-1}\right)^2$.  
Then $f$ is 2-to-1 everywhere except at $f(1)=\infty$ and $f(-1)=0$.
Therefore, the preimage by $f$ of each point of $P$ is a pair of points except at the endpoints of $P$, and by analytic continuation of the square root along $P$, the preiamge of $P$ is a pair of paths that share common endpoints, but are otherwise disjoint.  Hence, $f^{-1}(P)$ is a closed curve, and since $P$ avoids 1, $f^{-1}(P)$ avoids $f^{-1}(1) = \{0,\infty\}$.  
Let $C \subset\mb{C}$ be the closure of the region bounded by $f^{-1}(P)$ that contains 0.  

Let $P$ vary in Fréchet distance and let $y\in P$ vary. 
Then, $f^{-1}(P)$ varies continuously in Fréchet distance with respect to the spherical metric on $\overline{\mb{C}}$ since $f$ is uniformly continuous. 
Also, there is a unique internally conformal homeomorphism 
$h:\disk \to C$ by Carathéodory's theorem (\ref{theoremCaratheodory}), 
and $h$ varies continuously in the sup-metric by Radó's theorem (\ref{theoremRado}), 
so $h^{-1}$ varies continuously in the partial map metric by Lemma \ref{lemmaCPDistContinuity}.  
In the case where $y$ is not an endpoint, then $h^{-1}f^{-1}(y)$ consists of a pair of points on the unit circle that vary continuously, and the hyperbolic geodesic $\widetilde g(P,y) \subset \disk$ between the points $h^{-1}f^{-1}(y)$ varies continuously, so $g = fh\widetilde g(P,y)$ is the hyperbolic geodesic from $y$ to itself through $\mb{C}\setminus P$, 
and $g$ varies continuously in Fréchet distance by Lemma \ref{lemmaPMFrechetDistance}. 
In the case where $y$ converges to an endpoint, then $h^{-1}f^{-1}(y)$ is a single point, and $\widetilde g$ converges to $h^{-1}f^{-1}(y)$, so $g$ converges to the same endpoint as $y$. 
In either case $g$ varies continuously. 

Consider the case where $P$ is an arbitrary directed path, 
and let $P$ and $y\in P$ and $q\not\in P$ vary. 
Let $a(z) = a_{s,t,q}(z) = \crr(z,q; s,t)$ be the conformal automorphism of $\sphere^2$ 
that sends the source $s$ to 0, the terminal $t$ to $\infty$, and $q$ to 1. 
Then, $a = a_{s,t,q}$ varies continuously in the sup-metric as $q$ and the endpoints of $P$ vary, 
so $a(P)$ varies continuously in Fréchet distance by Lemma \ref{lemmaPMFrechetDistance}, 
so $ga(P,y)$ vary continuously. 
Also, $a^{-1}$ varies continuously by Lemma \ref{lemmaCPDistContinuity}, 
so $a^{-1}ga(P,y)$ varies continuously and is the hyperbolic geodesic from $y$ to itself through $\sphere^2 \setminus P$.

For the second part, suppose $y_1 <_P y_2$. 
If $y_1 =s$, then $C_2^\circ$ contains $s$ by definition, so let us consider the case where $y_1 \neq s$.  
Then, $h^{-1}f^{-1}(y_1)$ is a pair of points that separate $h^{-1}f^{-1}(y_2)$ from $h^{-1}f^{-1}(s)$ in $\sphere^1$ by Carathéodory's theorem, 
so $\widetilde g(P,y_1)$ separates $\widetilde g(P,y_2)$ from $h^{-1}f^{-1}(s)$ in $\disk$, 
so $g(P,y_1)$ separates $g(P,y_2)$ from $s$ in $\sphere^2$, 
so $C_1 \subset C_2^\circ$. 
\end{proof}

\begin{proof}[Proof of Lemma \ref{lemmaPathParam}]
Let us start with continuity of the inverse. 
Consider convergent sequences $P_k \to P_\infty$ and $y_k \to y_\infty$ as in Definition \ref{defPathParam}.  

Consider first the case where $y_\infty$ is not an endpoint of $P_\infty$, 
and let $C_k$ be the region bounded by the hyperbolic geodesic $g_k = g(P_k,y_k)$. 
Then, 
$g_k \to g_\infty$ by Claim \ref{claimParamPathG}, 
and $g_k$ is analytic except at a single point, 
so in particular $g_k$ is nonOsgood,   
so $\param^{-1}(P_k,y_k) = \area(C_k) \to \area(C_\infty)$
by Lemma \ref{lemmaAreaContinuous}. 

In the case where $y_\infty$ is the source of $P_\infty$, 
if $y_k$ is also the source of $P_k$ for $k<\infty$, 
then $\param^{-1}(P_k,y_k) = 0 = \param^{-1}(P_\infty,y_\infty)$, 
otherwise $g_k$ converges to the point $g_\infty = y_\infty$ by Claim \ref{claimParamPathG}, so $C_k$ converges to $y_\infty$, 
so $\param^{-1}(P_k,y_k) \to 0$.  
In either case $\param^{-1}(P_k,y_k) \to \param^{-1}(P_\infty,y_\infty)$.

Hence, $\param^{-1}(P)$ varies continuously in the partial map metric as $P$ varies in Fréchet distance by Lemma \ref{lemmaCPDistLimits}, 
and $\param^{-1}(P)$ is strictly monotone with respect to the total order on $P$ 
by the second part of Claim \ref{claimParamPathG}, 
so $\param(P)$ varies continuously in the the partial map metric by Lemma \ref{lemmaCPDistContinuity}, 
and $\param(P)$ varies continuously in the the sup-metric by Lemma \ref{lemmaCPDistSup}.  
Thus, parts \ref{itemPathParamHomeo} and \ref{itemPathParamContinuous} hold. 

Part \ref{itemPathParamEquivariant} follows from the observation that orthogonal transformations preserve area and are isogonal, so hyperbolic geodesics are sent to hyperbolic geodesics. 
Part \ref{itemPathParamReverse} follows from the observation that the area of the complement of $C$ in Definition \ref{defPathParam} is $4\pi -\area(C)$ and contains $t$. 
\end{proof}

%% file: interp-v5-9.tex
\subsection{Extension to a disk and interpolation}

Here we present useful properties and implications of the Douady-Earle extension, 
which extends a homeomorphism of the circle to a homeomorphism of the disk. 
More generally, we use the Douady-Earle extension to extend a homeomorphism of the boundary of a 2-cell to the rest of the 2-cell,
which we then use with Lemma \ref{lemmaPathParam} to define an interpolation map between regular cell decompositions.

\begin{lemma}
\label{lemmaInterp}
Given an isomorphism $\lambda : \mc{C}_0\to \mc{C}_1$ 
between regular cell decompositions $\mc{C}_k$ of $X_k \subset \sphere^2$, 
then 
\[ \interp(\lambda): X_0 \to X_1 \]
is a homeomorphism satisfying the following.  
\begin{enumerate}
\item 
\label{itemInterpLFace}
Each face $F \in \mc{C}_0$ has 
$\interp(\lambda;F) = \lambda(F)$.
\item
\label{itemInterpLContinuous}
$\interp(\lambda)$ varies continuously in the partial map metric 
(and in the sup-metric if $X_0$ is fixed)   
as $\mc{C}_0$ and $\mc{C}_1$ vary in Fréchet distance on 1-cells  
and Hausdorff distance on 2-cells. 
\item 
\label{itemInterpLEquivariant}
$\interp$ is $\orth_3$-equivariant on both sides.  
That is,
$\interp(Q_1\circ \lambda\circ Q_0) = Q_1\circ \interp(\lambda)\circ Q_0$.
for $Q_i \in \orth_3$. 
\item 
\label{itemInterpLID}
$\interp(\id) = \id$.
\end{enumerate}
\end{lemma}

This is essentially the same as the map interp in \cite[Subsection 3.2]{dobbins2021grassmannians}, except there the cell decomposition consisted of the cells of a pseudocircle arrangement.
Here we give a more elegant proof using Douady-Earle extension.


\begin{theorem}[Douady and Earle, 1986 \cite{douady1986conformally}]
\label{theoremDE}
Given a homeomorphism $\phi:\sphere^1 \to \sphere^1$, 
then $\DE(\phi) : \disk \to \disk$ is a homeomorphism that satisfies the following. 
\begin{enumerate}
\item 
\label{itemDEBoundary}
$\DE(\phi)$ restricts to $\phi$ on the boundary, i.e.,
$\rest(\sphere^1,\DE(\phi)) = \phi$. 
\item 
\label{itemDEId}
$\DE(\id) = \id$. 
\item 
\label{itemDEContinuous}
$\DE(\phi)$ varies continuously as $\phi$ varies in the sup-metric. 
\item 
\label{itemDEConformal}
$\DE$ is isogonally equivariant on the left and right, i.e., 
\[\DE(g\phi h) = g\DE(\phi)h\]
for isogonal automorphisms $g,h$ of the disk.%
\footnote{Douady and Earle defined conformal maps to be angle preserving, which we call isogonal.}
\end{enumerate}
\end{theorem}


\begin{definition}[Interpolation]
We build up $\iota = \interp(\lambda)$ one dimension at a time, starting with vertices. 
On each vertex $v \in \mc{C}_0$, let $\iota(v)=\lambda(v)$. 
On each 1-cell $P \in \mc{C}_0$, 
choose a direction on $P$, direct $\lambda(P)$ as induced by $\lambda$, and 
let $\iota = \param(\lambda(P))\param^{-1}(P)$. 
Finally, on each 2-cell $C \in \mc{C}_0$, 
choose conformal maps $h_0 : \disk \to C$ and $h_1 : \disk \to \lambda(C)$,
and let $\iota = h_1\ext(h_1^{-1}\iota h_0)h_0^{-1}$. 
We will show that this is well-defined in the proof of Lemma \ref{lemmaInterp}. 
\end{definition}

\begin{proof}[Proof of Lemma \ref{lemmaInterp}]
To see that $\interp(\lambda)$ is well defined on a 1-cell $P$, 
let $\widetilde \iota$ be define with the opposite choice of direction on $P$. 
Then for $x\in P$ we have 
\[
\widetilde \iota(x) 
= \param(\lambda(P)^-)\param^{-1}(P^-) 
= \param(\lambda(P),1-(1-\param^{-1}(P,x))) 
= \iota(x)
\]
by Lemma \ref{lemmaPathParam} part \ref{itemPathParamReverse}.
For a 2-cell $C$, 
let $\widetilde \iota$ be defined by some other choice of conformal maps $g_i$. 
Then, $g_0^{-1}h_0$ and $h_1^{-1}g_1$ are conformal automorphisms of the unit disk, so 
\begin{align*}
\widetilde \iota 
&=g_1\ext(g_1^{-1}\iota g_0)g_0^{-1} \\ 
&=g_1g_1^{-1}h_1h_1^{-1}g_1\ext(g_1^{-1}\iota g_0)g_0^{-1}h_0h_0^{-1}g_0g_0^{-1} \\
&=h_1\ext(h_1^{-1}g_1g_1^{-1}\iota g_0g_0^{-1}h_0)h_0^{-1} \\
&=\iota 
\end{align*} 
by Theorem \ref{theoremDE} part \ref{itemDEConformal}. 
The restriction of $\iota$ to the 1-skeleton is a well defined homeomorphism since $\param(P)$ is a homeomorphism by Lemma \ref{lemmaPathParam}, and the source of $P$ is $\param(P,0)$, 
which $\iota$ sends to the source $\param(\lambda(P),0)$ of $\lambda(P)$, 
and similarly for the terminal endpoint. 
On the boundary of a 2-cell $C$, we have 
\[
h_1\ext(h_1^{-1}\iota h_0)h_0^{-1}
= h_1h_1^{-1}\iota h_0h_0^{-1}
=\iota
\]
by Theorem \ref{theoremDE} part \ref{itemDEBoundary}. 
so the definition of $\iota$ on the boundary of $C$ agrees with that on the 1-skeleton. 
Thus, $\iota$ is well-defined homeomorphism by the gluing lemma. 

Each face $F$ has $\iota(F)=\lambda(F)$ directly from the definition, 
so part \ref{itemInterpLFace} holds. 

By Lemma \ref{lemmaPathParam} part \ref{itemPathParamContinuous}, 
$\param^{-1}(P)$ and $\param(P)$ vary continuously in the partial map metric as 1-cells of $\mc{C}_0$ and $\mc{C}_1$ vary in Fréchet distance, 
so $\iota$ varies continuously on the 1-skeleton by Lemma \ref{lemmaCPDistContinuity}. 
We may choose $h_i$ that vary continuously as $\lambda$ varies by Radó's theorem, 
so $h_i^{-1}$ vary continuously in the partial map metric by Lemma \ref{lemmaCPDistContinuity}, 
so $h_1h_0^{-1}$ varies continuously in the sup-metric by Lemma \ref{lemmaCPDistSup}, 
so $\ext(h_1h_0^{-1})$ varies continuously in the sup-metric by Theorem \ref{theoremDE} part \ref{itemDEContinuous}, 
so $\iota$ varies continuously by Lemma \ref{lemmaCPDistContinuity}, 
which means part \ref{itemInterpLContinuous} holds. 

Since orthogonal maps are isogonal, $\interp$ is equivariant by 
Lemma \ref{lemmaPathParam} part \ref{itemPathParamEquivariant} 
and Theorem \ref{theoremDE} part \ref{itemDEConformal}, 
so part \ref{itemInterpLEquivariant} holds. 

Finally, if $\lambda=\id$, then $\param(\lambda(P))\param^{-1}(P)=\id$, 
so $\iota=\id$ on the 1-skeleton, 
and we can choose $h_0=h_1$ in each 2-cell, 
so $\iota=\id$ by Theorem \ref{theoremDE} part \ref{itemDEId}, 
so part \ref{itemInterpLID} holds. 
\end{proof}

%% file: squish-v5-9.tex
\subsection{Squishing conformal maps}
\label{subsectionSquish}

Here we study the behavior of conformal parameterizations of a Jordan domain that converges to a path. 
Specifically, we show that if the parameterizations converge anywhere then the image converges to a single point along the path.

\begin{lemma} 
\label{lemmaSquishConformalPath}
${}$
\begin{itemize}
\item 
Let $A_k,B_k$ be a pair of paths from $p_{k}$ to $q_{k}$ for $k\in\{1,\dots,\infty\}$ 
that each converge in Fréchet distance as $k \to \infty$ 
to a path $A_\infty = B_\infty = C_\infty$ 
such that $A_k\cup B_k$ is a closed curve bounding a closed 2-cell $C_k$ for $k < \infty$.
\item 
Let $x_k,y_k \in C_k$ respectively converge to distinct points $x_\infty, y_\infty \in C_\infty \setminus\{p_{\infty},q_{\infty}\}$. 
\item 
Let $h_k : \disk \to C_k$ for $k < 1$ be internally conformal maps such that $h_k(1) = p_{k}$, $h_k(-1) = q_{k}$ 
and $|h_k^{-1}(x_{k})|$ is bounded away from $\{1,-1\}$, 
i.e., 
$|h_k^{-1}(x_{k})-1| >\eps$ and $|h_k^{-1}(x_{k})+1| >\eps$ for some $\eps>0$. 
\end{itemize}
If $y_\infty$ is on the arc of $C_\infty$ from $p_{\infty}$ to $x_\infty$, 
then $h_k^{-1}(y_k) \to 1$ as $k \to \infty$. 
Alternatively, if $y_\infty$ is on the arc of $C_\infty$ from $q_{\infty}$ to $x_\infty$, 
then $h_k^{-1}(y_k) \to -1$.  
Hence, if $z_k \in \disk$ is bounded away from $\{1,-1\}$, 
then $h_k(z_k) \to x_\infty$. 
\end{lemma}

\begin{claim}
\label{claimLittleCut}
There is a sequence of paths $M_k$ in $C_k$ from $A_{k}$ to $B_{k}$ that converge to $x_\infty$.
\end{claim}

\begin{proof}
Let $\eps>0$ be sufficiently small that $p_{\infty},q_{\infty},x_\infty$ are all more than $3\eps$ away from each other.
Let $P$ be a path from $p_{\infty}\oplus\eps$ to $q_{\infty}\oplus\eps$ 
that is disjoint from $C_\infty$,  
and let $\eps_1>0$ be sufficiently small that $P$ is disjoint from $C_\infty\oplus\eps_1$ and $\eps_1\leq\eps$. 
Let $k$ be sufficiently large that $C_k \subset C_\infty\oplus\eps_1$,   
and 
$A_{k}$ and $B_{k}$ pass within distance $\eps_1$ of $x_\infty$.

Let $M$ be a path of diameter at most $\eps_1$ that cross $C_\infty$ at $x_\infty$ and meets $C_\infty$ at no other point. 
Then, $A_{k}$ intersects $M$ for $k$ sufficiently large since $A_{k} \to C_\infty$, 
and likewise for $B_{k}$. 
Choose an arc $N_k$ of $M$ from a point of $A_{k}$ to a point of $B_{k}$, 
and let $M_k$ be the arc along $N_k$ from the last time $N_k$ meets $A_{k}$ to the next time $N_k$ meets $\partial C_k$. 
Then, $M_k$ must be a path from $A_{k}$ to $B_{k}$ 
through the complement of $\partial C_k$. 

Suppose for the sake of contradiction that $M_k$ were not in $C_k$. 
Then, we would have a Jordan domain $D$ bounded by circular arcs of radius $\eps$ about $p_{\infty}$ and $q_{\infty}$ and an arc of $A_{k}$ and $B_{k}$.
Also, $P$ and $M_k$ would be cross-cuts of $D$ with endpoints alternating around the boundary of $D$, so $P$ and $M_k$ must intersect, which is a contradiction since $M_k$ is contained within $x_\infty\oplus\eps_1$, which is disjoint from $P$. 
Thus, $M_k$ is a path in $C_k$. 

Each subsequence of $M_k\cap A_{k}$ has a convergent subsubsequence 
since $M$ is compact, 
and a limit point of $M_k\cap A_{k}$ must be on both $M$ and $C_\infty$, 
so $M_k\cap A_{k} \to x_\infty$, 
and likewise for $M_k\cap B_{k}$, so $M_k \to x_\infty$.
\end{proof}

To complete the proof of Lemma \ref{lemmaSquishConformalPath} we will use extremal length. 

\begin{definition}[Extremal length]
Let $\mr{L^2}(\Omega)$ denote the space of square-integrable functions on $\Omega\subset\mb{C}$, 
and given $\rho \in \mr{L^2}(\Omega)$, let 
\[\area(\Omega,\rho) = \int_{\Omega}\rho(z)^2\mr{d}z. \]
Given a continuous map $\gamma : [0,1] \to \Omega$, the length of $\gamma$ with respect to $\rho$ is 
\[\len(\gamma,\rho) = \int_{\gamma}\rho(s)\mr{d}s.\]
where $\mr{d}s$ is arc-length. 
We say $\gamma$ is a \df{rectifiable curve} when its arc-length is finite. 
Let $\area(\Omega) = \area(\Omega,1)$ 
and $\len(\gamma) = \len(\gamma,1)$. 
The \df{extremal length} of a collection of rectifiable curves $\Gamma$ in $\Omega$ is 
\[
\el(\Gamma) 
= \sup_{\rho\in \mr{L^2}(\Omega)} \inf_{\gamma \in \Gamma} 
\frac{\len^2(\gamma,\rho)}{\area(\Omega,\rho)}
\]
We call $(\Gamma,\Gamma^*)$ a \df{conformal rectangle} 
when $\Gamma$ is the collection of rectifiable curves in a 2-cell $\Omega$ between a pair of disjoint arcs $E,F$ on the boundary of $\Omega$, 
and $\Gamma^*$ is the collection of rectifiable curves in $\Omega$ separating $E$ and $F$. 
\end{definition}

Extremal length has two features that are important for us. 
First, extremal length is a conformal invariant. 
Second, 
if $(\Gamma,\Gamma^*)$ is a conformal rectangle, 
then $\el(\Gamma)\el(\Gamma^*) = 1$
\cite[Chapter 4, see examples]{beliaev2019conformal}.

\begin{proof}[Proof of Lemma \ref{lemmaSquishConformalPath}]
We will just consider the case where $y_\infty$ is on the arc of $C_\infty$ from $p_{\infty}$ to $x_\infty$.  The other case is similar.

Let $\eps_1$ be sufficiently small that $p_{\infty}$, $q_{\infty}$, $x_\infty$, $y_\infty$ are all at least distance $5\eps_1$ from each other. 
Let $\widetilde y\in C_\infty$ be the first point of $C_\infty$ from $x_\infty$ to $y_\infty$ that is distance $\eps_1$ from $y_\infty$,
and let $\eps_2 > 0$ be at most $\nicefrac{1}{4}$ the distance between 
the arc of $C_\infty$ from $p_{\infty}$ to $y_\infty$ 
and the arc of $C_\infty$ from $\widetilde y$ to $q_{\infty}$. 
Let $M_{\mathrm{y},k}$ be a path from $A_k$ to $B_k$ 
that converges to $\widetilde y$ as implied by Claim \ref{claimLittleCut},
and let $k$ be sufficiently large that $M_{\mathrm{y},k}$ is within distance $\eps_2$ of $\widetilde y$; 
see Figure \ref{figureSquish}.

\begin{figure}
\centering
\includegraphics[scale=1]{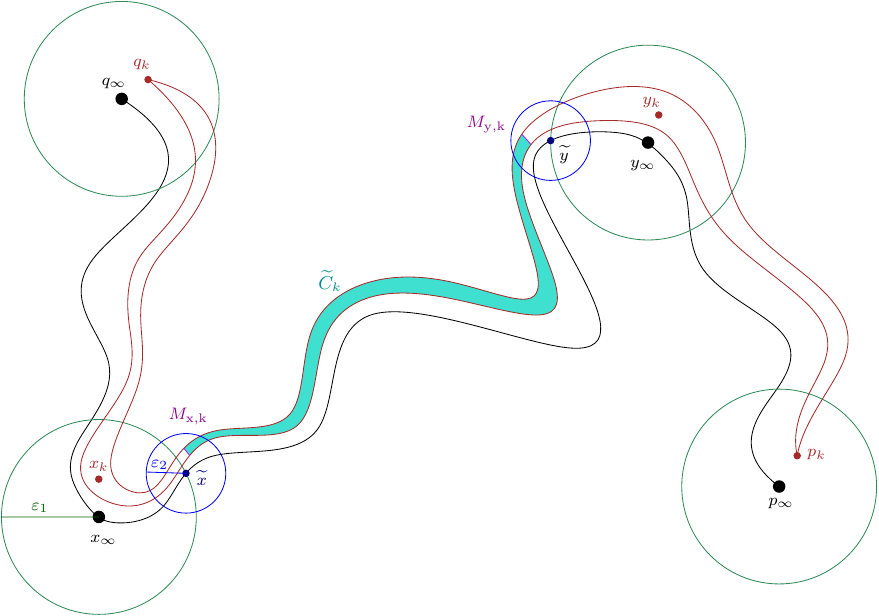}
\caption{Construction of $\widetilde C_k$ in the proof of Lemma \ref{lemmaSquishConformalPath}}
\label{figureSquish}
\end{figure}

Let $\eps_2$ also be sufficiently small and $k$ be sufficiently large for $\widetilde x$ and $M_{\mathrm{x},k}$ to be defined analogously.
That is, $\widetilde x$ is the last point of $C_\infty$ from $x_\infty$ to $y_\infty$ that is distance $\eps_1$ from $x_\infty$.
The arcs of $C_\infty$ from $q_{\infty}$ to $x_\infty$ and from $\widetilde x$ to $p_{\infty}$ are at least $4\eps_2$ apart. 
And, $M_{\mathrm{x},k}$ is a cross-cut of $C_k$ from $A_k$ to $B_k$ that converges to $\widetilde x$ and is within distance $\eps_2$ of $\widetilde x$. 

Let $\widetilde C_k$ be the region of $C_k$ that is 
between $M_{\mathrm{x},k}$ and $M_{\mathrm{y},k}$,
and let $\widetilde C_\infty$ be the arc of $C_\infty$ from $\widetilde x$ to $\widetilde y$.
Note that $\eps_2$ was chosen so that  
$x_k$ is separated from $\widetilde C_k$ by $M_{\mathrm{x},k}$ and 
$y_k$ is separated from $\widetilde C_k$ by $M_{\mathrm{y},k}$ for $k$ sufficiently large. 
Let $\Gamma_k$ be the family of all rectifiable curves from $M_{\mathrm{x},k}$ to $M_{\mathrm{y},k}$ though $\widetilde C_k$.
Let $\widetilde A_k$ and $\widetilde B_k$ respectively be the arcs of $A_k$ and $B_k$ along the boundary of $\widetilde C_k$, 
and let $\Gamma_k^*$ be the family of all rectifiable curves from $\widetilde A_k$ to $\widetilde B_k$ though $\widetilde C_k$. 
Since $(\Gamma_k,\Gamma_k^*)$ is a conformal rectangle, we have 
$\mathrm{el}(\Gamma_k)\mathrm{el}(\Gamma_k^*) = 1$ 
\cite[Chapter 4, examples]{beliaev2019conformal}.

We claim that $\el(\Gamma_k) \to \infty$ 
as $k \to \infty$. 
Consider first the case where $\widetilde C_\infty$ is a rectifiable path. 
Then, $\area(\widetilde C_k) \to 0$ since 
we can cover $\widetilde C_\infty$ with $\len(C_\infty)/\eps_3$ disks of radius $\eps_3>0$ spaced evenly along $\widetilde C_\infty$, and eventually $\widetilde C_k \subset \widetilde C_\infty\oplus\eps_3$, 
so $\area(\widetilde C_k) \leq \pi\len(C_\infty)\eps_3$. 
For each curve $\gamma \in \Gamma_k$, we have 
\[
\len(\gamma) \geq \|\widetilde y-\widetilde x\|-2\eps_2 \geq \|y_\infty -x_\infty\|-4\eps_1 \geq \eps_1, 
\]
since the paths $M_{\mathrm{x},k}$ and $M_{\mathrm{y},k}$ are respectively within $\eps_2$ of $\widetilde x$ and $\widetilde y$, so 
\begin{gather*}
\el(\Gamma_k) \geq \inf_{\gamma\in\Gamma_k} \frac{\len^2(\gamma)}{\area(\widetilde C_k)} 
\geq \frac{\eps_1^2}{\area(\widetilde C_k)}
\to \infty. 
\end{gather*}

Consider the case where $\widetilde C_\infty$ is not rectifiable. 
Then, $\area(\widetilde C_k) \leq A$ for some fixed $A$ since $\widetilde C_\infty$ is bounded. 
Also, for each $L>0$, every $\gamma_k \in \Gamma_k$ for $k$ large enough has 
$\len(\gamma_k) > L$ 
since $\widetilde C_\infty$ has a piecewise linear approximation $P$ 
of length at least $2L$ 
with vertices on $\widetilde C_\infty$, 
and eventually every $\gamma_k \in \Gamma_k$ must pass within $\delta>0$ of each vertex in order along $\widetilde C_\infty$ 
where $\delta$ is chosen so that each pair of vertices of $P$ is at least $4\delta$ apart. 
Letting $L \to \infty$, we have 
\[
\el(\Gamma_k) \geq \inf_{\gamma\in\Gamma_k} \frac{\len^2(\gamma)}{\area(\widetilde C_k)} 
\geq \frac{L^2}{A}
\to \infty. 
\]

In any case $\el(\Gamma_k) \to \infty$, so $\el(\Gamma_k^*) \to 0$, 
and so $\el(h_k^{-1}(\Gamma_k^*)) \to 0$ 
since extremal length is a conformal invariant, 
so there is a sequence of rectifiable curves $\zeta_k \in h_k^{-1}(\Gamma_k^*)$ 
from the arc $h_k^{-1}(\widetilde A_{k})$ of the upper unit semicircle 
to the arc $h_k^{-1}(\widetilde B_{k})$ of the lower unit semicircle, 
and $\len(\zeta_k) \to 0$ as $k\to \infty$. 
Also, $h^{-1}(x_k)$ is on the same side of $\zeta_k$ as $-1$ 
since we are in the case where $y_\infty$ is on the arc of $C_\infty$ from $p_{\infty}$ to $x_\infty$, 
and $h^{-1}(x_k)$ is bounded away from $-1$, 
so some ark of $\zeta_k$ must be bounded away from $-1$, 
and so $\zeta_k$ is bounded away from $-1$ since $\len(\zeta_k) \to 0$.   

Hence, $\zeta_k$ is a sequence of paths of vanishing length 
from the upper unit semicircle 
to the lower unit semicircle bounded away from $-1$, 
and as such, must converge to 1.
Thus, $h^{-1}(y_k) \to 1$ since $h^{-1}(y_k)$ is on the same side of $\zeta_k$ as 1.
\end{proof}

%% file: hood-v5-9.tex
\section{The neighborhood.}
\label{sectionHood}

\FloatBarrier

Here we will prove Theorem \ref{theoremMacPG} using the nerve theorem.  
To do so, we define an open neighborhood of the set of topological representations of a rank 3 oriented matroid such that for every incomparable pair of oriented matroids, the corresponding neighborhoods are dsjoint; see Lemma \ref{lemmaHoodDisjoint}, and for every chain of oriented matroids, the corresponding intersection of neighborhoods is contractible; see Lemma \ref{lemmaCrush}. 

Informally, 
an arrangement $A$ will be in the neighborhood corresponding to an oriented matroid $\mc{M}$ 
when the geometric features of $A$ that correspond to nondegenerate features of $\mc{M}$ 
are larger than those that correspond to degenerate features of $\mc{M}$. 
We define the neighborhood in terms of parameters $\minbig(\mc{M},A)$, which measures the minimum size of a big feature of $A$, i.e., a nondegenerate feature of $\mc{M}$, 
and $\maxlit(\mc{M},A)$, which measures the maximum size of a little feature, i.e., degenerate in $\mc{M}$.  
These parameters will use a map $\maxcov(\mc{M},A)$ to associate features of $\mc{M}$ to features of $A$. 

An arrangement $A$ is in the intersection of neighborhoods corresponding to a chain $\mc{C}\subset\mcp_{3,n}$ when features that correspond to a degeneracy are smaller the lower that degeneracy appears in the chain.  
Conceptually, 
$A$ can appear to represent a different oriented matroid at different magnifications.  
For example, two distinct pseudocircles that are very close together might appear to coincide in a low resolution image. 
In this sense, higher oriented matroids in the chain $\mc{C}$ correspond to viewing $A$ at a higher magnification, 
and at a higher magnification we can distinguish smaller features, 
so fewer features appear degenerate; see Figure \ref{figureMagnification}. 

\begin{figure}
\centering
\includegraphics[scale=1]{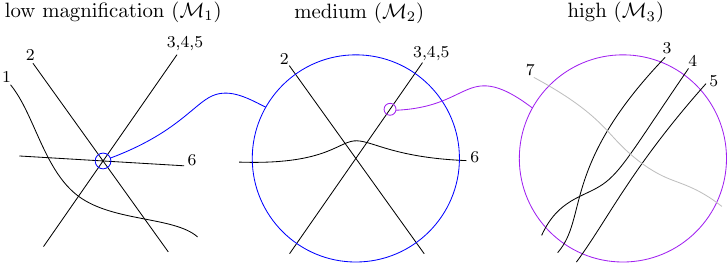}
\caption{Part of an arrangement $A \in \hood(\{\mc{M}_1,\mc{M}_2,\mc{M}_3\})$ 
at different magnifications.
Only $\mc{M}_1$ has a vertex covector $\sigma$ with $\{2,\dots,6\} \subseteq \sigma^0$. 
In both $\mc{M}_1$ and $\mc{M}_2$, the elements 3,4,5 are parallel and element 7 is a loop, but not in $\mc{M}_3$. 
Pseudocircles 2 through 6 appear to meet at a common point at low magnification, but not at medium or high. 
Pseudocircles 3, 4, and 5 coincide and pseudocircle 7 vanishes at low and medium magnification, but not at high. 
}
\label{figureMagnification}
\end{figure}

\begin{definition}[neighborhoods]
Given oriented matroids $\mc{M}_0 \leq_\mr{w} \mc{M}_1$, let 
\begin{gather*}
\maxcov(\mc{M}_0,\mc{M}_1):\cov(\mc{M}_1) \to \cov(\mc{M}_0) \text{ by } \\ 
\maxcov(\sigma_1) = \max(\cov(\mc{M}_0)_{\leq\sigma_1}).
\end{gather*}
Laura Anderson showed that the map $\maxcov(\mc{M}_0,\mc{M}_1)$ is well-defined and surjective 
for $\mc{M}_0\leq_\mr{w}\mc{M}_1$ 
\cite{anderson2001representing}

Given a rank 3 oriented matroid $\mc{M}$, 
and an arrangement $A \geq_\mr{w} \mc{M}$,  
and $x\in\sphere^2$, 
let 
\begin{gather*}
\maxcov(\mc{M},A) : \sphere^2 \to \cov(\mc{M}) \text{ by }  \\
\maxcov(\mc{M},A,x) = \maxcov(\mc{M},\om(A),\sign(A,x)) 
\end{gather*}

We may simply write $\maxcov = \maxcov(\mc{M},A)$ when $\mc{M}$ and $A$ are understood from context, 
and let $\maxcov^{-1} = \maxcov^{-1}(\mc{M},A) = [\maxcov(\mc{M},A)]^{-1}$.  
In particular, for $\sigma \in \cov(\mc{M})$, we have  
\[
\maxcov^{-1}(\sigma) = \{x \in \sphere^2: \max(\cov(\mc{M})_{\leq\sign(A,x)})=\sigma \};
\]
see Figure \ref{figureMaxcov}.
Let \[
\van(\mc{M},A) = \{p \in \sphere^2 : \exists i \in \supp(\mc{M}) : \maxcov(\mc{M},A,p,i)=0  \}.
\]

\begin{figure}
\centering
\includegraphics[scale=1]{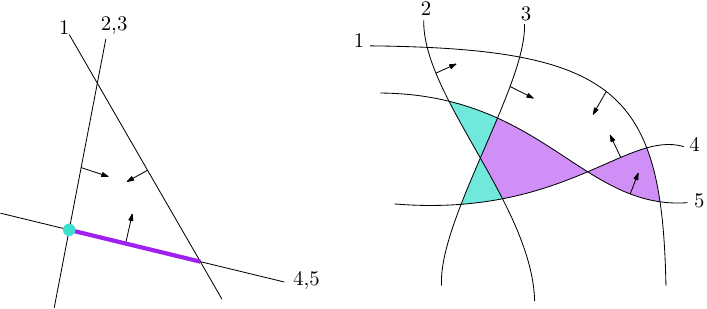}
\caption{Left:~Oriented matroid $\mc{M}$ represented by an oriented line arrangement. \newline
Right:~Part of a pseudocircle arrangement, and 
$\maxcov^{-1}(\sigma)$ for $\sigma = (+,0,0,0,0)$ in teal and for $\sigma = (+,+,+,0,0)$ in purple. 
Arrows indicate the positive side of each pseudocircle. 
}
\label{figureMaxcov}
\end{figure}

The \df{inradius} of a set $X$ is the radius of the largest open metric disk contained in $X$. 
Let $\vanrad(\mc{M},A)$ be the inradius of $\van(\mc{M},A)$, 
and let $\maxlit(\mc{M},A)$ be the maximum value among $\vanrad(\mc{M},A)$ and the weights of loops of $\mc{M}$. 

For a facet covector $\sigma$ of $\mc{M}$, 
let $\inrad(\mc{M},\sigma,A)$ be the inradius of $\maxcov^{-1}(\mc{M},A,\sigma)$. 
Let $\minbig(\mc{M},A)$ be the minimum among the value of $\inrad(\mc{M},\sigma,A)$ among facet covectors $\sigma$ of $\mc{M}$ and the weights of nonloops of $\mc{M}$.

Let 
\begin{align*}
\upa(\mc{M}) &= \{A \in \psv_{3,n} : A \geq_\mr{w} \mc{M} \}, \\
\hood(\mc{M}) &= \{A \in \upa(\mc{M}) : \maxlit(\mc{M},A) < \minbig(\mc{M},A)\}.
\end{align*}
Given a chain $\mc{C} = \{\mc{M}_1<\dots<\mc{M}_L\}$, let 
\[
\hood(\mc{C}) = \bigcap_{k=1}^L \hood(\mc{M}_k). 
\]
Let $\upa(\chi),\hood(\chi)\subset\psv_{3,n}$ be defined analogously for a chirotope $\chi\in\omcp_{3,n}$.
\end{definition}

\begin{lemma}[Open neighborhood]
\label{lemmaHoodSymOpen}
$\hood(\mc{M})$ is an $\orth_3$-symmetric open neighborhood of $\psv(\mc{M})$. 
\end{lemma}

\begin{lemma}[Disjoint neighborhoods]
\label{lemmaHoodDisjoint}
If $\mc{M}$ and $\mc{N}$ are incomparable in the weak order, 
then $\hood(\mc{M})$ and $\hood(\mc{N})$ are disjoint. 
\end{lemma}

In Definition \ref{defCrush}, we will define a map $\crush(\mc{C},\Omega)$ that satisfies the following. 

\begin{lemma}[Crush]
\label{lemmaCrush}
Given a finite chain $\mc{C} \subset \mcp_{3,n}$ and $\Omega\in\upa(\max(\mc{C}))$, 
then the map $\crush(\mc{C},\Omega)$ is a strong $\orth_3$-equivariant deformation retraction from 
$\hood(\mc{C})\cap\upa(\Omega)$ to the $\orth_3$-orbit of $\Omega$. 
Hence, $\hood(\mc{C})/\orth_3$ is contractible. 
\end{lemma}

We will prove Lemma \ref{lemmaHoodDisjoint} in the next subsection, 
and Lemma \ref{lemmaHoodSymOpen} at the end of Section \ref{sectionOpen}, 
and Lemma \ref{lemmaCrush} at the end of Subsection \ref{subsectionEbbCrush}. 

Note that each oriented matroid $\mc{M}$ corresponds to a pair of chirotopes $\chi,-\chi$.  
Moreover, $\ot(-A) = -\ot(A)$, and order type is $\sorth_3$-invariant,  
so $\hood(\mc{M})$ is the disjoint union of $\hood(\chi)$ and $\hood(-\chi)$  
since $\hood(\mc{M})$ is $\orth_3$-symmetric by Lemma \ref{lemmaHoodSymOpen},  
so by Lemmas \ref{lemmaHoodSymOpen} and \ref{lemmaHoodDisjoint} we have the following.  

\begin{corollary}
\label{corollaryChiHood}
If $\chi,\psi$ are incomparable chirotopes, then 
$\hood(\chi)$ is a $\sorth_3$-symmetric open neighborhood of $\psv(\chi)$, 
and $\hood(\chi)$ and $\hood(\psi)$ are disjoint. 
\end{corollary}

Similarly a chain of oriented matroids $\mc{C}$ 
corresponds to a pair of chains $\widetilde{\mc{C}},-\widetilde{\mc{C}}$ of chirotopes, 
and $\hood(\mc{C})$ is the disjoint union of $\hood(\widetilde{\mc{C}})$ and $\hood(-\widetilde{\mc{C}})$. 
Also, the action of $\orth_3$ on $\psv_{3,n}$ is free by \cite[Lemma 5.0.1]{dobbins2021grassmannians}, 
so the $\orth_3$-orbit of an arrangement $\Omega$ has two connected components, 
namely the $\sorth_3$-orbit of $\Omega$ and that of $-\Omega$, 
and by Lemma \ref{lemmaCrush} we have the following. 

\begin{corollary}
\label{corollaryChiCrush}
If $\widetilde{\mc{C}} \subset \omcp_{3,n}$ is a chain,
then $\hood(\widetilde{\mc{C}})/\sorth_3$ is contractible. 
\end{corollary}

The main theorem now follows  
from the above lemmas and their corollaries.

\begin{proof}[Proof of Theorem \ref{theoremMacPG}]
Let 
\[
\mc{U} = \{ \hood(\mc{M})/\orth_3 : \mc{M} \in \mcp_{3,n} \}. 
\]
Then, $\mc{U}$ 
is an open cover of $\psg_{3,n}$  
by Lemma \ref{lemmaHoodSymOpen}. 
Also, every nonempty intersection of finitely many sets among $\mc{U}$ is contractible 
by Lemmas \ref{lemmaHoodDisjoint} and \ref{lemmaCrush},
so $\psg_{3,n}$ is homotopy equivalent to the nerve of $\mc{U}$ 
by the nerve theorem \cite[Corollary 4G.3]{hatcher2002algebraic}.
Moreover, a collection of sets among $\mc{U}$ have a nonempty intersection precisely when the corresponding oriented matroids form a chain,   
so the nerve of $\mc{U}$ is isomorphic as a simplicial complex to the geometric realization of the order complex of $\mcp_{3,n}$, 
which is $\|\mcp_{3,n}\|$. 
Thus, $\|\mcp_{3,n}\|$ is homotopy equivalent to $\psg_{3,n}$, 
which is homotopy equivalent to $\g_{3,n}$ by Theorem \ref{theoremPsV}. 
Similarly, $\|\omcp_{3,n}\|$ is homotopy equivalent to $\og_{3,n}$ 
by Corollaries \ref{corollaryChiHood} and \ref{corollaryChiCrush}. 
\end{proof}

While $\mcp_{3,\infty}$ contains infinite ascending chains, 
we only need to consider finite chains, since only finite chains define a simplex of the nerve, 
or a simplex in the geometric realization of $\mcp_{3,\infty}$.

\subsection{Disjointness}
\label{subsectionHoodDisjoint}

In this subsection we prove Lemma \ref{lemmaHoodDisjoint}.

\begin{lemma}
\label{lemmaBasisBLIneq}
If there is some triple $I\in {[n]\choose 3}$ that is a basis for $\mc{M}_1$ but not for $\mc{M}_0$ 
and $A \in \upa(\mc{M}_0)\cap\upa(\mc{M}_1)$, 
then $\minbig(\mc{M}_1,A) \leq \maxlit(\mc{M}_0,A)$. 
\end{lemma}

\begin{proof}
Consider first the case where there is $i \in I$ such that 
$i$ is a loop of $\mc{M}_0$ but not $\mc{M}_1$.  
Then, by definitions of minbig and maxlit, we have 
\[
\minbig(\mc{M}_1,A) \leq \wt_i(A) \leq \maxlit(\mc{M}_0,A).
\]

Next, consider the case where no $i \in I$ is a loop of $\mc{M}_0$. 
Since $I$ is not a basis for $\mc{M}_0$, there is some 
$\sigma \in \{+,-\}^I$ that is not in the restriction of $\cov(\mc{M}_0)$ to $I$, 
but $I$ is a basis for $\mc{M}_1$, 
so every element of $\{+,-\}^I$ appears in the restriction of $\cov(\mc{M}_1)$ to $I$, 
so there is a facet covector $\tau$ of $\mc{M}_1$ such that $\rest(I,\tau) = \sigma$. 
Let $r$ be the maximum radius of a disk $D$ in $\maxcov^{-1}(\mc{M}_1,A,\tau)$. 
Then, $\rest(I,\sign(A,p))=\sigma$ on each point $p \in D$, 
so $\maxcov(\mc{M}_0,A,p)$ must vanish on some $i \in I$, 
which is a nonloop of $\mc{M}_0$ in this case, 
so $D \subset \van(\mc{M}_0,A)$. 
Thus, 
\[
\minbig(\mc{M}_1,A) \leq r \leq \maxlit(\mc{M}_0,A). \qedhere
\]
\end{proof}

\begin{proof}[Proof of Lemma \ref{lemmaHoodDisjoint}]
Suppose there were $A \in \hood(\mc{M})\cap\hood(\mc{N})$. 
Then, there would be some triple that is a basis for $\mc{M}$ but not for $\mc{N}$, 
since $\mc{M} \not< \mc{N}$ and both are below $\om(A)$ in the weak order, so 
\[
\minbig(\mc{M},A) \leq \maxlit(\mc{N},A) < \minbig(\mc{N},A)
\] 
by Lemma \ref{lemmaBasisBLIneq}. 
Likewise, there would be some other triple that is a basis for $\mc{N}$ but not $\mc{M}$, 
so $\minbig(\mc{N},A) < \minbig(\mc{M},A)$,
which is a contradiction. 
\end{proof}

%% file: open-v5-7.tex
\section{Openness and borders}
\label{sectionOpen}

Here we will prove Lemma \ref{lemmaHoodSymOpen}.  Most of the work will be to show that $\hood(\mc{M})$ is open.   
To this end, we show that $\minbig$ and $\maxlit$ are continuous, by showing that the boarders between regions $\maxcov^{-1}(\sigma)$ vary continuously. 
This is captured by the following lemma, which will also be useful for showing continuity in the proof of Lemma \ref{lemmaCrush}.

\begin{lemma}
\label{lemmaMaxcovBorder}
Let $\Sigma = \{\upsilon<\sigma\} \subset \csph(\mc{M})$ 
and consider an arrangement $A\geq_\mr{w}\mc{M}$. 
Then, 
\begin{enumerate}
\item 
\label{itemMaxcovBorderPath}
$B = \cl(\maxcov^{-1}(\sigma))\cap\cl(\maxcov^{-1}(\upsilon))$ 
is a $({\leq}1)$-cell. 
\item 
\label{itemMaxcovBorderEndpoints}
There are precisely 2 covectors $\tau_0,\tau_1 \in \csph(\mc{M})$ such that $\Sigma\cup\{\tau_i\}$ is a chain. Moreover, the endpoints of $B$ are $p_i = B\cap\cl(\maxcov^{-1}(\tau_i))$. 
\item 
\label{itemMaxcovBorderContinuous}
$B$ directed from $p_0$ to $p_1$ varies continuously in Fréchet distance 
as the arrangement $A$ varies. 
\end{enumerate}
\end{lemma}

Recall we order define a partial order $\leq_\mr{v}$ on signs $\{0,+,-\}$ by 
$0 <_\mr{v} (+)$, and $0 <_\mr{v} (-)$, and $(+),(-)$ are incomparable. 

\begin{lemma}
\label{lemmaAimSemicontinuous}
$\sign : \psv_{3,n} \times \sphere^2 \to \{0,+,-\}^n$ is lower semicontinuous for $n\leq\infty$. 
That is, if $A_t \to A_\infty$ and $x_t \to x_\infty$ are convergent sequences, 
then $\sign(A_\infty,x_\infty) \leq_\mr{v} \sign(A_t,x_t)$ for all $t$ sufficiently large. 
\end{lemma}

\begin{proof}
Let $\sigma_t = \sign(A_t,x_t)$. 
In the case where $\sigma_\infty(i)=0$, then $\sigma_\infty(i)\leq\sigma_t(i)$ regardless of the value of $\sigma_t(i)$, so let us assume by symmetry that $\sigma_\infty(i) = (+)$. 
Let $\delta$ be the distance between $x_\infty$ and $S_i(A_\infty)$. 
Then, $\wt_i(A_\infty)>0$, and $\wt_i(A_t) \to \wt_i(A_\infty)$, so $\wt_i(A_t)>0$ for $t$ sufficiently large,  
and so $S_i(A_t) \to S_i(A_\infty)$ in Fréchet distance,  
so $S_i^+(A_t)$ is within Hausdorff distance $\delta/3$ of $S_i^+(A_\infty)$ for $t$ sufficiently large. 
Also, $x_t$ is within $\delta/3$ of $x_\infty$ for $t$ sufficiently large, 
so $x_t \in S_i^+(A_t)$, so $\sigma_\infty(i)=\sigma_t(i)$ in the case where $\sigma_\infty(i)\neq 0$. 
\end{proof}

\begin{lemma}
\label{lemmaOmSemicontinuous}
$\om : \psv_{3,n} \to \mcp_{3,n}$ is lower semicontinuous for $n\leq\infty$. 
Thas is, if $A_t \to A_\infty$,
then $A_\infty \leq_\mr{w} A_t$ for $t$ sufficiently large. 
Likewise, $\ot$ is lower semicontinuous. 
\end{lemma}

\begin{proof}
Consider $\sigma_\infty\in\cov(A_\infty)$. 
Then, there is $p\in\sphere^2$ such that $\sign(A_\infty,p)=\sigma_\infty$. 
We may assume that $A_t$ has the same support as $A_\infty$; otherwise $A_t \geq_\mr{w} A_\infty$ provided that $\rest(\supp(\mc{M}),A_t) \geq_\mr{w} A_\infty$. 
Hence, there are only finitely many $\sigma_\mr{t} \in \{0,+,-\}^n$ such that $\sign(A_t,p) = \sigma_\mr{t}$ occurs infinity often as $t\to\infty$, 
and $\sigma_\infty \leq_\mr{v} \sigma_\mr{t}$ by Lemma \ref{lemmaAimSemicontinuous}, 
so $A_\infty \leq_\mr{w} A_t$ in each of these cases, 
so $A_\infty \leq_\mr{w} A_t$ for $t$ sufficiently large since there are finitely many cases. 
\end{proof}

Let us next show that the endpoints of the border between adjacent regions defined by $\maxcov$ are indeed well-defined points. 

\begin{lemma}
\label{lemmaFlagPoint}
\[
p = \bigcap_{\sigma \in \Sigma}\cl(\maxcov^{-1}(\sigma))
\]
is a single point for each maximal chain $\Sigma\subset\csph(\mc{M})$, 
and each arrangement $A \geq_\mr{w} \mc{M}$.  
\end{lemma}

\begin{proof}
Let $\Sigma = \{\upsilon<\sigma<\tau\}$ and let $\tau_\mr{top} > \sigma$ be the only facet covector other than $\tau$ that is above $\sigma$. 
Then, $\om(A)$ can be represented by x-monotone pseudolines in the projective plane $\projplane$ as the closure of $\mb{R}^2$ such that $\cell(A,\tau_\mr{top})$ is the region above the upper envelope of each curve \cite[Theorem 6.3.3]{bjorner2000oriented}.  
Let us assume that $\pi(A)$ is such a pseudoline arrangement 
where $\pi : \sphere^2 \to \projplane$ is the standard covering map. 
Moreover, 
let us assume by choice of horizon of the projective plane on the 2-sphere that $C = \pi(\cell(A,\tau)) \subset \mb{R}^2 \subset \projplane$ 
is a cell in the Euclidean plane. 
Since $\sigma = \tau\wedge\tau_\mr{top}$ is an edge covector, i.e., a subtope, 
each nonloop $i$ satisfies 
$\tau(i)\neq\tau_\mr{top}(i)$ if and only if $i\in\sigma^0$ 
\cite[Lemma 4.2.2]{bjorner2000oriented}. 
Hence, the curves above $C$ are precisely the curves $S_i$ for $i\in\sigma^0$, 
so the upper envelope of $C$ is $E = \maxcov^{-1}(\sigma)\cap\cl(\maxcov^{-1}(\tau))$, 
and 
since $\sign(A)$ is lower semicontinuous,  
$p$ is an endpoint of $E$   
\end{proof}

\begin{proof}[Proof of Lemma \ref{lemmaMaxcovBorder} parts \ref{itemMaxcovBorderPath} and \ref{itemMaxcovBorderEndpoints}]

Consider the case where $\Sigma = \{\upsilon<\sigma\}$ consists of a vertex covector and an edge covector, 
and let $p \in B$. 
Then, $p$ is the limit point of one sequence $p_{\sigma,k}$ in $\maxcov^{-1}(\sigma)$ 
and the limit point of another sequence $p_{\upsilon,k}$ in $\maxcov^{-1}(\upsilon)$, 
so $\sign(A,p_{\upsilon,k}) <_\mr{v} \sigma \leq_\mr{v} \sign(A,p_{\sigma,k})$, 
so $\sign(A,p) <_\mr{v} \sigma$ since $\sign$ is lower semicontinuous, 
so $p\in S_i=S_i(A)$ for some $i \in \upsilon^0\setminus\sigma^0$. 
Hence, $B$ is a union of vertices and edges of $A$. 

Consider a vertex $v$ of $B$, 
and let $I_v = \{i:\sign(A,v)=0\}$. 
Then, $\maxcov^{-1}(\sigma)$ coincides locally around $v$ with the intersection of the open pseudohemispheres $S_i^{\sigma(i)}$ for $i \in I_v\setminus \sigma^0$ 
and closed pseudohemispheres $\cl(S_j^{-\tau(j)})$ for $j \in (I_v\cap\sigma^0)\setminus \tau^0$ 
where $\tau > \sigma$ is a facet covector. 

In the case where there is no such facet covector $\tau$, then $\maxcov^{-1}(\sigma)$ coincides locally with an intersection of open pseudohemispheres, and the compliment coincides locally with $\maxcov^{-1}(\upsilon)$, 
so two of the corresponding pseudocircles are on the boundary between $\maxcov^{-1}(\sigma)$ and $\maxcov^{-1}(\upsilon)$, 
so $v$ has degree 2 in $B$. 
 
In the case where there is only one such $\tau$, then 
$\maxcov^{-1}(\sigma)$ coincides locally with an intersection of 
open pseudohemispheres $S_i^{\sigma(i)}$ 
and one closed pseudohemisphere $\cl(S_j^{-\tau(j)})$. 
Also, $\tau$ must be a facet covector of $A$ as well as $\mc{M}$ since $A \geq_\mr{w} \mc{M}$, 
so there is an open cell $C = \cell(A,\tau)$ and $v$ is on the boundary of this cell. 
If the compliments $\cl(S_i^{-\sigma(i)})$ of the open pseudohemispheres were to locally cover the 
compliment $S_j^{\tau(j)}$ of the closed pseudohemisphere, 
then for every point $p$ sufficiently close to $v$, if $\sign(A,p,j) = \tau(j)$ then $\sign(A,p,i) \neq \sigma(i) = \tau(i)$ for some $i \in I_v\setminus \sigma^0$ depending on $p$, 
which would mean $p \not\in C$, but that would contradict that $v$ is on the boundary of $C$. 
Hence, the closed pseudohemisphere $\cl(S_j^{-\tau(j)})$ is not covered by the closure of the open pseudohemispheres, 
so one of the pseudocircles $S_i$ is on the boundary between $\maxcov^{-1}(\sigma)$ and $\maxcov^{-1}(\upsilon)$ and another $S_j$ is on the boundary between $\maxcov^{-1}(\sigma)$ and $\maxcov^{-1}(\tau)$, 
so $v$ has degree 1 in $B$.  

In the case where there are two such $\tau$, then $\maxcov^{-1}(\sigma)$ locally intersects the boundaries of $\maxcov^{-1}(\tau)$ for each such $\tau$, 
and this locally covers the boundary of $\maxcov^{-1}(\sigma)$ by a similar argument as above, 
so $v$ is an isolated vertex of $B$.  

By Lemma \ref{lemmaFlagPoint}, 
$B$ intersects each closed region $\cl(\maxcov^{-1}(\tau))$ for $\tau>\sigma$ at a point, 
so $B$ has either 2 vertices of degree 1 or has a single isolated vertex, 
and all other vertices, if there be any, have degree 2. 
Therefore, $B$ consists of one connected component that is either an isolated vertex or a path along with a disjoint union of cycles. 

Suppose $B$ had a cycle $C$. 
Then, $C$ would divide the sphere into two regions by the Jordan curve theorem with $\maxcov^{-1}(\upsilon)$ and $\maxcov^{-1}(\sigma)$ locally on either side of $C$. 
Also,  
$\maxcov^{-1}(\upsilon)$ and $\maxcov^{-1}(\sigma)$ are path connected and no path through either set could cross $C$, 
since that would give $B$ a vertex of degree at least 3, 
so $\maxcov^{-1}(\upsilon)\setminus C$ and $\maxcov^{-1}(\sigma)$ would be  separated by $C$, but that contradicts that $B$ also has a connected component that is a path or point. 
Thus, the disjoint union of cycles is empty, and so $B$ is a path or point. 
An analogous argument holds in the cases where $\Sigma$ consists of a vertex and facet covector or of a edge and facet covector, which means part \ref{itemMaxcovBorderPath} holds, 
and since the vertices of degree 1 in the case where $B$ is a path are those on a closed pseudohemisphere $\cl(S_j^{-\tau(j)})$, part \ref{itemMaxcovBorderEndpoints} holds. 
\end{proof}

To complete the proof of Lemma \ref{lemmaMaxcovBorder}, 
it remains to show that the boarder $B$ varies continuously.  
We first show that objects appearing in the definition of $B$ vary continuously.

\begin{lemma}
\label{lemmaConvergentCrossingPoint}
If $P,Q$ are a pair of directed paths that meet at a single point, 
then $P\cap Q$ varies continuously as $P$ and $Q$ vary in Fréchet distance. 
This holds analogously for undirected paths and for closed curves in the projective plane. 
\end{lemma}

This is essentially the same as \cite[Lemma 3.1.3]{dobbins2021grassmannians}. 
We do not repeat the argument here.

\begin{lemma}
\label{lemmaConvergentPathSequence}
If $P_{i,k} \to P_\infty$ is a sequence of directed paths that converges in Fréchet distance as $k \to \infty$ for each $i\in[n]_\mb{N}$,
and $Q_k$ is a directed path that fist traverses an arc of $P_{1,k}$, and then an arc of $P_{2,k}$, and so on until traversing an arc of $P_{n,k}$, then $Q_k \to P_\infty$ in Fréchet distance. 
\end{lemma}

\begin{proof}%
Let us assume that $n=2$ since the general case will follow by induction using the same argument as in the $n=2$ case.

Let $\phi_{i,k} : [0,1]_\mb{R} \to P_{i,k}$ be parameterizations that converge to some parameterization $\phi_\infty : [0,1]_\mb{R} \to P_{\infty}$ in the sup-metric as in the hypothesis of the Lemma. 
Note what we may assume that these converge to the same $\phi_\infty$ for both $i\in\{1,2\}$; otherwise we can reparameterize appropriately.
Let $a_k \in [0,1]$ be the last time where $P_{1,k}$ traverses $Q_k$, i.e., 
\[
a_k = \sup \{t \in [0,1]: \phi_{1,k}([0,t])\subset Q_k\},
\]
and let $q_k = \phi_{1,k}(a_k)$. 
Observe that every point of $Q_k$ after $q_k$ must be in $P_{2,k}$, 
and in particular, $q_k \in P_{1,k}\cap P_{2,k} \cap Q_k$.

Fix $\eps>0$, and cover $P_\infty$ by open arcs of diameter at most $\eps$. 
Since $P_\infty$ is compact, we may choose a finite subcover.  
Let $C$ by the finite cover of $P_\infty$ by the closure of each arc, 
and let $\delta$ be the minimum separation distance between disjoint arcs of this cover. 
Note that $\delta > 0$ since the arcs are each compact. 
Let us assume that $\delta<\eps$ and $k$ is sufficiently large that $\|\phi_{i,k}-\phi_\infty\|(t) <\delta/4$.

Let $b_k \in [a_k,1]$ be the first time after $a_k$ where $P_{2,k}$ is at least $\eps$ away from $q_k$, or 1 if this never occurs, i.e., 
\[
b_k = \inf \{1,t \in [a_k,1] : \|\phi_{2,k}(t)-q_k\| \geq \eps \}.
\]
Then, $\|\phi_{2,k}-\phi_{1,k}\|(t) <\delta/2 <\eps/2$, 
so $\|\phi_{2,k}(a_k)-q_k\| <\eps/2$, so $b_k > a_k$.

Let $c_k = \phi_{2,k}^{-1}(q_k)$. 
Then, $\|\phi_\infty(c_k) -q_k\| < \delta/4$ and likewise for $a_k$, 
so $\|\phi_\infty(c_k) -\phi_\infty(a_k)\| < \delta/2$, 
so $\phi_\infty(a_k)$ and $\phi_\infty(c_k)$ cannot be in disjoint arcs of the cover $C$, 
so $\phi_\infty([c_k,a_k])$ has diameter no more than $2\eps$, 
so $\phi_{2,k}([c_k,a_k])$ has diameter less than $3\eps$, 
so $\phi_{2,k}([c_k,b_k])$ has diameter less than $5\eps$. 
Note that we do not know the order of $a_k$ and $c_k$, but if $a_k\leq c_k$, then $\phi_{2,k}([c_k,b_k]) \subseteq \phi_{2,k}([a_k,b_k])$ has diameter at most $2\eps$. 

We now define a parameterization $\psi_k$ of $Q_k$, which also depends on $\eps$.
For $t\leq a_k$, let $\psi_k(t) = \phi_{1,k}(t)$. 
For $t \geq b_k$, let $\psi_k(t) = \phi_{2,k}(t)$. 
And for $t \in [a_k,b_k]$, let $\psi_k$ be an arbitrarily chosen parameterization of the arc $\phi_{2,k}([c_k,b_k]) \subseteq P_{2,k}$ from $q_k$ to $\phi_{2,k}(b_k)$. 
Then, $\|\psi_k-\phi_\infty\|(t) < \eps/4$ for $t\not\in (a_k,b_k)$ by choice of $k$ sufficiently large. 
Also, $\|\psi_k-\phi_\infty\|(t) < 9\eps$ for $t\in [a_k,b_k]$ 
since the arc $\psi_k([a_k,b_k]) =\phi_{2,k}([c_k,b_k])$ has diameter less than $5\eps$, 
and the arc $\phi_\infty([a_k,b_k])$ has diameter less than $3\eps$,
and $\|\psi_k-\phi_\infty\|(b_k) < \delta/4<\eps$. 
Hence $\dist_\mr{F}(Q_k,P_\infty) <9\eps$. 

By letting $\eps \to 0$, we have $Q_k \to P_\infty$ in Fréchet distance.
\end{proof}

\begin{proof}%
[Proof of Lemma \ref{lemmaMaxcovBorder} part \ref{itemMaxcovBorderContinuous}]

Consider a convergent sequence $A_k \to A_\infty$. 
and let $B_k$ and $p_{i,k}$ be defined from $A_k$ as in the statement of the lemma. 

In the case where $B_k$ is a point for $k$ sufficiently large, 
there is some curve $S_j$ through $B_k$ with $j \in \upsilon^0 \setminus \sigma^0$ that crosses a curve along the boundary of $\maxcov^{-1}(\tau_0)$ at $p_{0,k} = B_k$, 
so $B_k \to B_\infty$ by Lemma \ref{lemmaConvergentCrossingPoint}.

Consider the case where $B_k$ is not a point for all $k$ sufficiently large. 
We already have that $B_k$ consists of a sequence of arcs along curves $S_{j_1}(A_k),\dots,S_{j_m}(A_k)$ from $p_{0,k}$ to $p_{1,k}$ from the proof of parts \ref{itemMaxcovBorderPath} and \ref{itemMaxcovBorderEndpoints} above.
Since it suffices to show convergence in finitely many cases, we may assume that the sequence $j_1,\dots,j_m$ is fixed for each $k<\infty$.  
The curve $B_\infty$ also consists of such a sequence of arcs, but possibly with additional degeneracies since $A_\infty \leq A_k$ by Lemma \ref{lemmaOmSemicontinuous}, so we can partition 
$j_1,\dots,j_m$ into subintervals $j_{1,1},\dots,j_{1,m_1},j_{2,1},\dots,j_{n,m_n}$ as follows. 
For $i$ even, 
each curve of the subinterval $S_{j_{i,1}},\dots,S_{j_{i,m_i}}$ converges to a common curve that traverses an arc of $B_\infty$ that is not a single point.  
Let the $i$th subinterval for $i$ odd be the intervening curves. 
Note that the subinterval could be empty in the case where $i$ is odd, or all curves could be contained in a single subinterval, which could be $i=1$ or $i=2$, but at least one subinterval is nonempty. 
Also, $B_\infty$ is a sequence of arcs along $S_{j_{2,1}},S_{j_{4,1}}\dots,S_{j_{n-1,1}}$.

The even subintervals converge appropriately by Lemma \ref{lemmaConvergentPathSequence}, 
so let us consider the case where $i$ is odd. 
To reduce the need for special cases, 
let $S_{j_0,1}$ be a curve along the boundary of $\maxcov^{-1}(\tau_0)$ through $p_{0,k}$ 
and let $S_{j_{n+1},1}$ be that of $\maxcov^{-1}(\tau_1)$ through $p_{1,k}$.  
Then, curves $S_{j_{i-1,m_{i-1}}}$ and $S_{j_{i+1,1}}$ never become parallel, so their intersection varies continuously by Lemma \ref{lemmaConvergentCrossingPoint}.   
Similarly, $S_{j_{i-1,m_{i-1}}}\cap S_{j_{i,h}}$ and $S_{j_{i,h}}\cap S_{j_{i+1,1}}$ 
vary continuously,
provided that the $i$-th subinterval is nonempty. 
Moreover, 
the order type of $[S_{j_{i-1,m_{i-1}}},S_{j_{i,h}},S_{j_{i+1,1}}](A_\infty)$ is either the same as that in $A_k$ or 0 by Lemma \ref{lemmaOmSemicontinuous}. 
In the former case,  
a curve of an odd interval would intersect $B_\infty$ in more than one point contradicting the choice of partition into even and odd intervals. 
Hence, the later case must hold, which means the curves $[S_{j_{i-1,m_{i-1}}},S_{j_{i,h}},S_{j_{i+1,1}}](A_\infty)$ 
meet at a common point, 
and  
the arc of $B_k$ along the $i$-th interval converges to that point by Lemma \ref{lemmaConvergentCrossingPoint}. 
Thus, $B_k \to B_\infty$. 
\end{proof}

Next we show that objects and parameters in the definition of $\hood(\mc{M})$ vary continuously.

\begin{lemma}
\label{lemmaMaxcovContinuous}
Given a covector $\sigma \in \csph(\mc{M})$,  
then $\maxcov^{-1}(\mc{M},A;\sigma)$ varies continuously in Hausdorff distance 
as $A \geq_\mr{w} \mc{M}$ varies. 
Hence, $\van(\mc{M},A)^\mr{c}$ varies continuously in Hausdorff distance. 
\end{lemma}

\begin{proof}
Consider a convergent sequence $A_k \to A_\infty$ in $\upa(\mc{M})$, 
and let $R_k = R_k(\sigma) = \maxcov^{-1}(\mc{M},A_k,\sigma)$.  
Since $\Cov(\mc{M})$ is thin and $\mc{M}$ has rank 3, 
the intersection graph of pairs $\{\tau,\upsilon\}$ that form a chain with $\sigma$ 
is a regular graph $G$ of degree 2, and $G$ is connected by the circuit elimination axiom, 
so $G$ is a cycle. 
Also, $\partial R_k$ is covered by 
$({\leq}1)$-cells of the form 
$B_{\upsilon,k} = \cl(R_k)\cap\cl(\maxcov^{-1}(\upsilon))$ 
with endpoints of the form 
$p_{\upsilon,\tau,k}= B_{\upsilon,k}\cap B_{\tau,k}$  
by Lemma \ref{lemmaMaxcovBorder} parts \ref{itemMaxcovBorderPath} and \ref{itemMaxcovBorderEndpoints},  
and so we have surjective maps $\phi_k : \sphere^1 \to \partial R_k$ 
such that $\phi_k \to \phi_\infty$ in the sup-metric  
by Lemma \ref{lemmaMaxcovBorder} part \ref{itemMaxcovBorderContinuous}. 
Also, $R_k$ is on the same side of $B_{\upsilon,k}$ for all $k$ sufficiently large since $A_k \to A_\infty$, so membership in $R_k$ is determined by winding number of $\phi_k$. 
If $x \in \sphere^2$ is at least distance $\eps>0$ away from $\partial R_\infty$, 
then 
the winding number of $\phi_k$ around $x$ is the same as that of $\phi_\infty$ 
for $k$ sufficiently large, 
so $x \in R_k$ if and only if $x \in R_\infty$. 
Hence, the Hausdorff distance between $R_k$ and $R_\infty$ 
is at most $\eps$, so $R_k \to R_\infty$. 
For the last part, observe that 
\[
\dist_\mr{H}(\van(\mc{M},A_k)^\mr{c},\van(\mc{M},A_\infty)^\mr{c})
\leq \max_{\sigma} \dist_\mr{H}(R_k(\sigma),R_\infty(\sigma))
\]
among facet covectors $\sigma$. 
\end{proof}

\begin{lemma}
\label{lemmaBigLittleContinuous}
The maps $\vanrad(\mc{M}),\inrad(\mc{M},\sigma),\minbig(\mc{M}),\maxlit(\mc{M}): \upa(\mc{M}) \to \mb{R}$ are each continuous. 
\end{lemma}

\begin{proof}
We start with $\vanrad(\mc{M})$, which will take most of the work.
Consider a convergent sequence $A_m \to A_\infty$, 
and suppose $r_m = \vanrad(\mc{M},A_m)$ does not converge to $r_\infty = \vanrad(\mc{M},A_\infty)$.
Since the range of $\vanrad(\mc{M})$ is compact, 
we may assume $r_m \to \widetilde r_\infty$. 

Consider the case where $\widetilde r_\infty \leq r_\infty -\eps$ with $\eps>0$. 
Then, $\van(\mc{M},A_\infty)$ contains a metric disk $D$ of radius $r_\infty$. 
Also, $\van(\mc{M},A_m)^\mr{c} \to \van(\mc{M},A_\infty)^\mr{c}$ in Hausdorff distance by Lemma \ref{lemmaMaxcovContinuous}, 
so $\van(\mc{M},A_m)^\mr{c} \subseteq \van(\mc{M},A_\infty)^\mr{c}\oplus(\eps/3)$ for all $m$ sufficiently large. 
Let $C$ be the disk concentric with $D$ of radius $r_\infty -(\eps/2)$.
Then, each point of $C$ is at least distance $(\eps/2)$ from any point in $\van(\mc{M},A_\infty)^\mr{c}$, 
whereas each point of $\van(\mc{M},A_m)^\mr{c}$ is within distance $(\eps/3)$ of some point in $\van(\mc{M},A_\infty)^\mr{c}$, so $C$ is disjoint from $\van(\mc{M},A_m)^\mr{c}$, 
so $C \subset \van(\mc{M},A_m)$, so $r_m \geq r_\infty -(\eps/2)$, 
so $\widetilde r_\infty \geq r_\infty -(\eps/2)$,
which is a contradiction. 

Consider the case where $\widetilde r_\infty \geq r_\infty +\eps$ with $\eps>0$. 
Then, $\van(\mc{M},A_m)$ contains a disk $D_m$ of radius $r_\infty +(\eps/2)$
for all $m$ sufficiently large.
Since $\sphere^2$ is compact, we may restrict to a subsequence where the center points of these disks converge to point, and this is the center point of a disk $D$ of radius $r_\infty+(\eps/3)$, which is contained in $\van(\mc{M},A_m)$ for $m$ sufficiently large. 
Again by Lemma \ref{lemmaMaxcovContinuous}, we have $\van(\mc{M},A_m)^\mr{c} \to \van(\mc{M},A_\infty)^\mr{c}$ in Hausdorff distance, 
so $\van(\mc{M},A_\infty)^\mr{c} \subseteq \van(\mc{M},A_m)^\mr{c}\oplus(\eps/6)$ for all $m$ sufficiently large. 
Let $C$ be a disk concentric with $D$ of radius $r_\infty+(\eps/9)$. 
Then, each point of $C$ is at least distance $(r_\infty+(\eps/3)) - (r_\infty+(\eps/9)) = (2\eps/9)$
from any point in $\van(\mc{M},A_m)^\mr{c}$, whereas each point of $\van(\mc{M},A_\infty)^\mr{c}$ is within distance $(\eps/9)$ of some point in $\van(\mc{M},A_m)^\mr{c}$, 
so $C$ is disjoint from $\van(\mc{M},A_\infty)^\mr{c}$, so $C\subset \van(\mc{M},A_\infty)$, 
so $r_\infty \geq r_\infty +(\eps/9)$, which is a contradiction. 
In each case we have a contradiction, so we must have $r_\infty = r_\infty$, 
and so $\vanrad(\mc{M})$ is continuous. 

Since $\partial\maxcov^{-1}(\sigma)$ varies continuously in Fréchet distance 
by Lemma \ref{lemmaMaxcovBorder}, 
the radius of the largest disk contained in the region bounded by the closed curve $\partial\maxcov^{-1}(\sigma)$ also varies continuously, 
so $\inrad(\mc{M},\sigma)$ is continuous. 
Also, weights $\wt_i$ vary continuously by definition of the metric on $\psv_{3,n}$. 
With this, $\minbig$ and $\maxlit$ are respectively a minimum and a maximum of a finite collection of continuous functions, and as such are continuous
\end{proof}




\begin{proof}[Proof of Lemma \ref{lemmaHoodSymOpen}]
Consider $A \in \psv(\mc{M})$. Then, $\van(\mc{M},A)$ is a collection of pseudocircles, so $\vanrad(\mc{M},A)=0$, and loops of $\mc{M}$ are loops of $A$, so $\maxlit(\mc{M},A)=0$. 
Also, $\maxcov^{-1}(\sigma)$ is a Jordan domain for each facet covector $\sigma$, so $\inrad(\mc{M},\sigma,A)>0$, and nonloops of $\mc{M}$ are nonloops of $A$, so $\minbig(\mc{M},A)>0$. 
Hence, $A \in \hood(\mc{M})$, so $\psv(\mc{M})\subseteq \hood(\mc{M})$. 

We have $\upa(\mc{M}) \subset \psv_{3,n}$ is open by Lemma \ref{lemmaOmSemicontinuous}, 
and $\hood(\mc{M}) \subset \upa(\mc{M})$ is defined by a strict inequality between continuous functions by Lemma \ref{lemmaBigLittleContinuous}, so $\hood(\mc{M})$ is open. 

Also, $\van(\mc{M},QA)=Q\van(\mc{M},A)$ and $\maxcov^{-1}(\mc{M},QA;\sigma) = Q\maxcov^{-1}(\mc{M},A;\sigma)$ for $Q \in \orth_3$, 
so $\vanrad(\mc{M})$ and $\inrad(\mc{M},\sigma)$ are $\orth_3$-invariant, 
and weights of pseudocircles are also invariant, so $\minbig$ and $\maxlit$ are invariant, 
so $\hood(\mc{M})$ is $\orth_3$-symmetric. 
\end{proof}













%% file: crush-v5-9.tex
\section{Crushing neighborhoods}
\label{sectionCrush}

\subsection{Ebb to Crush}
\label{subsectionEbbCrush}

Here we define 
the deformation retraction 
$\crush(\mc{C},\Omega)$ of Lemma \ref{lemmaCrush}
and prove the lemma.
For this we use the deformation retraction $\crush(\Omega)$
of Theorem \ref{theoremCrushPsVM} from $\psv(\Omega)$ to the orbit of $\Omega$, 
and another deformation $\ebb_\mr{R}(\mc{C},r)$ in Definition \ref{defEbbR}, 
which will use another deformation  $\ebb_\mr{D}(\mc{C},\eps)$ in Definition \ref{defEbbD}. 
These deformations will make use of the following parameter. 

\begin{definition}[Little-big ratio]
\label{defLBR}
\[
\lbr(\mc{M},A) = \frac{\maxlit(\mc{M},A)}{\minbig(\mc{M},A)}
\]
\end{definition}
Note that $\lbr(\mc{M})$ tests membership in both $\hood(\mc{M})$ and $\psv(\mc{M})$ 
among $A\in\upa(\mc{M})$ 
as 
\begin{align*}
\hood(\mc{M}) = \{A\in\upa(\mc{M}):\lbr(\mc{M},A)<1\}, \\ 
\psv(\mc{M}) = \{A\in\upa(\mc{M}):\lbr(\mc{M},A)=0\}. 
\end{align*}

The properties of $\ebb_\mr{R}$ we use are as follows.

\begin{lemma}[Ratio ebb]
\label{lemmaEbbR}
Let $\mc{C} = \mc{C}_{L-1}\sqcup \{\mc{M}_L\} = \{\mc{M}_1<\dots<\mc{M}_{L-1} < \mc{M}_L\}$ be a nonempty chain of oriented matroids and $r \in [0,1)_\mb{R}$.
Then, $\phi = \ebb_\mr{R}(\mc{C},r)$ is a strong equivariant deformation retraction 
from $\hood(\mc{C}_{L-1})\cap\upa(\mc{M}_L)$ to $\hood(\mc{C}_{L-1})\cap\psv(\mc{M}_L)$ satisfying the following. 
\begin{enumerate}
\item 
\label{itemEbbContinuous}
$\phi(A,t)$ varies continuously as $(r,A,t)$ vary. 
\item 
\label{itemEbbOT}
$\ot(\phi(A,t)) = \ot(A)$ is unchanging for $t<1$. 
\item 
\label{itemEbbR}
If $\lbr(\mc{M}_L,A)\leq r$, then $\phi(A,t) \in \hood(\mc{M}_L)$ for all $t$. 
\end{enumerate}
\end{lemma}

\begin{definition}[Crush]
\label{defCrush}
Here we globally define the map   
$\crush(\mc{C},\Omega) : X \times [0,1]_\mb{R} \to X$  
on $X = \hood(\mc{C})\cap\upa(\Omega)$. 
Let $A_t = \crush(\mc{C},\Omega;A,t)$ locally within this subsection. 

We define crush recursively on the length of $\mc{C}$. 
Given $A \in \hood(\mc{C})$, 
let 
\[
\crush(\emptyset,\Omega;A,t) = A_t = [\crush(\Omega)\cdot\ebb_\mr{R}(\om(\Omega),0)](A,t)
\] 
be the deformation in the case where the chain is empty, 
and for the recursive step with $\mc{C} = \{\mc{M}_1,\dots,\mc{M}_{L}\}$ nonempty, let  
\begin{align*}
\crush(\mc{C},\Omega;A,t)
= A_t 
&=\ebb_\mr{R}(\mc{C},r;A_{\mr{rec},t},s_t)
\intertext{where}
A_{\mr{rec},t}
&=\crush(\mc{C}_{L-1},\Omega;A,t), \\
\mc{C}_{L-1}
&=\{\mc{M}_1,\dots,\mc{M}_{L-1}\}, \\
r 
&=\max\{\lbr(\mc{M}_L;A_\mr{x}) : A_\mr{x} \in \{A,\Omega\}\},
\end{align*}
and where $s_t = s_t(A)$ is a stopping time defined as follows.
We want $s_t$ to be large enough to guarantee that $\ebb_\mr{R}(\mc{C},r;A_{\mr{rec},t},s_t) \in \hood(\mc{M}_L)$, and $s_t$ should depend continuously on $A$ and $t$. 
To this end, we should use some parameter that tells us when the arrangement is in $\hood(\mc{M}_L)$ and we can stop applying ebb.  The little-big ratio is such a parameter, but while ebb will eventually reduce the little-big ratio to a sufficiently small value, we would rather have a parameter that is strictly decreasing as a function of stopping time so that our stopping time will be well defined.  
Instead, we will define a strictly decreasing upper bound $b_t$ for the little-big ratio. 
Formally, let 
\begin{align*}
s_t 
&= \ext(b_t^{-1},y_t) = 
\begin{cases}
b_t^{-1}(y_t) & b_t(0) > y_t \\
0 & \text{otherwise} 
\end{cases} \\ 
\intertext{where}  
b_t(x) 
&= \tailsup(b_\mr{test}(A_{\mr{rec},t}),x) +c(1-x) \\
b_\mr{test}(A_\mr{x};x) 
&= \lbr(\mc{M}_L,\ebb_\mr{R}(\mc{C},r;A_\mr{x},x)) \\ 
y_t 
&= tb_1(0) +(1-t)b_0(0). \\ 
c 
&= (\nicefrac12)(1- \sup\{b_\mr{test}(A_\mr{x},x) : A_\mr{x} \in \{A,\Omega\}, x\in[0,1]_\mb{R} \}) 
\end{align*}
\end{definition}

We will show in Claim \ref{claimCrushBDecreasing} that the two expressions for $s_t$ are equivalent and that $s_t$ is well defined.

\begin{lemma}
\label{lemmaLBR}
Given $A\in \upa(\mc{M})$, then 
$\minbig(\mc{M},A) > 0$ and 
$\maxlit(\mc{M},A)\geq 0$ with `=' if and only if $A \in \psv(\mc{M})$.
Hence, 
$\lbr(\mc{M})$ is well defined, nonnegative, and continuous on $\upa(\mc{M})$, 
and $\lbr(\mc{M},A)=0$ if and only if $A\in\psv(\mc{M})$. 
\end{lemma}


\begin{proof}
Both $\minbig$ and $\maxlit$ are defined as extrema of weights of pseudocircles and radii of disks, so both are nonnegative. 

If $A\in \upa(\mc{M})$, then each pseudocircle of $A$ corresponding to a nonloop of $\mc{M}$ has positive weight, 
and $\maxcov^{-1}(\mc{M},A,\sigma)$ is a nonempty open region for each facet covector $\sigma$, 
and as such contains a disk of positive radius, 
so $\minbig(\mc{M},A) > 0$ in the case where $A \in \upa(\mc{M})$. 


If $A\in \psv(\mc{M})$, then each pseudocircle of $A$ corresponding to a loop of $\mc{M}$ has weight 0, 
and each facet covector of $A$ is a facet covector of $\mc{M}$, 
so $\van(\mc{M},A)$ is a finite collection of pseudocircles, 
so $\van(\mc{M},A)$ does not contain a disk of positive radius, 
so $\maxlit(\mc{M},A) = 0$ in the case where $A \in \psv(\mc{M})$.

If $A\in \upa(\mc{M})\setminus\psv(\mc{M})$, then
$\om(A) >_\mr{w} \mc{M}$, so 
some triple $I \in {[n]\choose 3}$ is a basis for $A$ but not for $\mc{M}$, 
so 
by Lemma \ref{lemmaBasisBLIneq}, 
$0<\minbig(\om(A),A)\leq\maxlit(\mc{M},A)$ 
in this case. 

For the last part, 
$\lbr(\mc{M},A)$ is well-defined 
since $\minbig(\mc{M},A) > 0$, 
and $\lbr(\mc{M},A)=0$ if and only if $A\in\psv(\mc{M})$ since the same holds for $\maxlit(\mc{M},A)$. 
Also, $\lbr(\mc{M})$ is continuous since $\minbig(\mc{M})$ and $\maxlit(\mc{M})$ are continuous by Lemma \ref{lemmaBigLittleContinuous}. 
\end{proof}



\begin{claim}
\label{claimCrushZero}
$s_0= s_1 = 0$ 
and $A_0 = A_{\mr{rec},0} = A$.
\end{claim}

\begin{proof}
We have $y_0 = b_0(0)$ and $y_1 = b_1(0)$ so $s_0= s_1 = 0$ by definition.  
Hence,  
\[
A_0  
=\ebb_\mr{R}(\mc{C},r;A_{\mr{rec},0},0)
=A_{\mr{rec},0} 
\]
since $\ebb_\mr{R}(\mc{C},r)$ is a deformation retraction by \ref{lemmaEbbR}, 
and 
\[
A_{\mr{rec},0}
= \crush(\mc{C}_{L-1},\Omega;A,0) = A 
\]
since $\crush(\mc{C}_{L-1},\Omega)$ is a deformation retraction by induction on $L$.
\end{proof}

\begin{claim}
\label{claimCrushCPositive}
$c > 0$ and $0 < y_t<1$. 
\end{claim}

\begin{proof}
We have $r \geq \lbr(\mc{M}_L,A_\mr{x})$ for $A_\mr{x} \in \{A,\Omega\}$ by definition, 
so $\ebb_\mr{R}(\mc{C},r;A_\mr{x},x) \in \hood(\mc{M}_L)$ by Lemma \ref{lemmaEbbR} part \ref{itemEbbR}, 
so $b_\mr{test}(A_\mr{x},x) < 1$ for each $x \in [0,1]$ by definition of hood, 
so $\sup\{b_\mr{test}(A_\mr{x},x) : x \in [0,1]\} < 1$ since $x$ is in a compact domain, 
so $c>0$. 

Also, $A_{\mr{rec},0} = A$ by Claim \ref{claimCrushZero}, so 
\begin{align*}
f_0(0) 
&= \sup\{b_\mr{test}(A,x) : x \in [0,1]\} +c 
\leq (\nicefrac12)(1+\sup\{b_\mr{test}(A,x) : x \in [0,1]\}) 
<1.
\end{align*} 
Similarly $b_1(0) < 1$ since $A_{\mr{rec},1} = \Omega$ by induction 
using Lemma \ref{lemmaCrush}.
Hence, $y_t < 1$. 
Also, $\lbr$ is always nonnegative, so $b_\mr{test}$ is nonnegative, so $b_t(0)>0$ since $c>0$, 
so $y_t>0$. 
\end{proof}

\begin{claim}
\label{claimCrushBDecreasing}
$b_t$ is strictly decreasing 
and $b_t(1) = 0$.   
Hence, $s_t<1$ is well defined. 
\end{claim}

\begin{proof}
$b_t(x)$ is the sum of a nonincreasing function of $x$, 
namely $\tailsup(b_\mr{test}(A_{\mr{rec},t}),x)$, 
and $c(1-x)$, which is a strictly decreasing since $c>0$ by Claim \ref{claimCrushCPositive}. 
Since $\ebb_\mr{R}(\mc{C},r)$ is a deformation retraction to a subset of $\psv(\mc{M}_L)$, 
we have 
\begin{align*}
b_t(1) 
&= \tailsup(b_\mr{test}(A_{\mr{rec},t}),1) 
= b_\mr{test}(A_{\mr{rec},t},1) 
= \lbr(\mc{M}_L,\ebb_\mr{R}(\mc{C},r;A_{\mr{rec},t},1)) 
= 0, 
\end{align*} 
by Lemma \ref{lemmaLBR}. 
Either $y_t \geq b_t(0)$, in which case $s_t=0$, 
or $0 < y_t < b_t(0)$, in which case $y_t$ is in the range of $b_t$, 
so $s_t = b_t^{-1}(y_t)$ is well defined since $b_t$ is strictly decreasing. 
Moreover, 
$\ebb_\mr{R}(\mc{C},r;A_{\mr{rec},t},x) \in \psv(\mc{M})$ only at $x=1$ 
by Lemma \ref{lemmaEbbR} part \ref{itemEbbOT}, 
so $b_t(x)=0$ only at $x=1$ by Lemma \ref{lemmaLBR}, 
so $b_t^{-1}(y)$ only attains the value 1 at $y=0$, 
but $y_t>0$ by Claim \ref{claimCrushCPositive}, 
so $s_t<1$.  
\end{proof}

\begin{claim}
\label{claimCrushSContinuous}
$s_t$ varies continuously as $A,t$ vary. 
\end{claim}

\begin{proof}
$A_{\mr{rec},t}$ varies continuously by induction, 
and $\ebb_\mr{R}(\mc{C})$ is continuous by Lemma \ref{lemmaEbbR} part \ref{itemEbbContinuous}, 
and $\lbr(\mc{M}_L)$ is continuous by Lemma \ref{lemmaLBR}, 
so $r=\max(\lbr(\mc{M}_L,\{A,\Omega\}))$ varies continuously, 
so $b_\mr{test}(A_{\mr{rec},t},x) = \lbr(\mc{M}_L,\ebb_\mr{R}(\mc{C},r,A_{\mr{rec},t},x))$ 
varies continuously as $x$ varies. 
Hence, $b_\mr{test}(A_{\mr{rec},t})$ varies continuously in the sup-metric 
by Lemma \ref{lemmaSupDistPartialApp}, 
so $\tailsup(b_\mr{test}(A_{\mr{rec},t}))$ and $c$ vary continuously 
by Lemma \ref{lemmaTailsupContinuous}, 
so $b_t$ varies continuously in the sup-metric, 
so $y_t$  varies continuously, 
and $b_t$ is strictly decreasing by Claim \ref{claimCrushBDecreasing}, 
so $s_t$ varies continuously by Lemma \ref{lemmaExtendedInverse}.
\end{proof}

\begin{proof}[Proof of Lemma \ref{lemmaCrush}]

We proceed by induction on the length of a chain. 
Let us start with the base case where $\mc{C} = \emptyset$.
Then, $\ebb_\mr{R}(\om(\Omega),0)$ is a strong equivariant deformation retraction from 
$\upa(\Omega)$ to $\psv(\Omega)$ by Lemma \ref{lemmaEbbR}. 
Also, $\crush(\Omega)$ is a strong equivariant deformation retraction from $\psv(\om(\Omega))$ to the orbit of $\Omega$ 
by Theorem \ref{theoremCrushPsVM}.
Hence, the lemma holds for the empty chain. 

Let us assume by induction that $\crush(\mc{C}_\mr{rec},\Omega)$ is such a deformation retraction 
for $\mc{C}_\mr{rec} = \{\mc{M}_1,\dots,\mc{M}_{L-1}\}$. 
Let 
$A_{t} = \crush(\mc{C},\Omega;A,t)$ for $\mc{C} = \{\mc{M}_1,\dots,\mc{M}_{L}\}$
and $A_{\mr{rec},t} = \crush(\mc{C}_\mr{rec},\Omega;A,t)$.  

Let us first show that $\crush(\mc{C},\Omega)$ is continuous. 
We have $A_{\mr{rec},t} = \crush(\mc{C}_{\mr{rec}},\Omega;A,t)$ varies continuously as a function of $(A,t)$ by inductive assumption.  
Also, $r$ is continuous by Lemma \ref{lemmaLBR}, 
and the stopping time $s_t$ is continuous by Claim \ref{claimCrushSContinuous}, 
so $\ebb(\mc{C},r;A_{\mr{rec},t},s_t)$ varies continuously by Lemma \ref{lemmaEbbR} part \ref{itemEbbContinuous}. 
Hence, $\crush(\mc{C},\Omega)$ is continuous.

Let us next show that the range of $\crush(\mc{C},\Omega)$ is in $\hood(\mc{C})\cap\upa(\Omega)$. 
We have $A_{\mr{rec},t} \in \hood(\mc{C}_\mr{rec})\cap\upa(\Omega)$ by induction, 
so $A_t\in \hood(\mc{C}_\mr{rec})\cap\upa(\mc{M}_L)$ by Lemma \ref{lemmaEbbR}. 
Also, $s_t<1$ by Claim \ref{claimCrushBDecreasing}, 
so $\ot(A_t) = \ot(A_{\mr{rec},t})$ by Lemma \ref{lemmaEbbR} part \ref{itemEbbOT}, 
so $A_t\in\upa(\Omega)$.
Also, 
\begin{align*}
\lbr(\mc{M}_L,A_t)
&= \lbr(\mc{M}_L,\ebb_\mr{R}(\mc{C},r;A_{\mr{rec},t},s_t)) 
= b_\mr{test}(A_{\mr{rec},t},s_t) 
\leq b_t(s) 
\leq y_t 
< 1
\end{align*}
by Claim \ref{claimCrushCPositive}, 
so $A_t \in \hood(\mc{M}_L)$.
Thus, $A_t \in\hood(\mc{C})\cap\upa(\Omega)$.

Let us show that $\crush(\mc{C},\Omega)$ is trivial on the orbit of $Q$. 
Consider the case where $A=\Omega*Q$. 
Then, 
$r(A)=r(\Omega)$ since little-big ratio is constant on $\orth_3$-orbits, 
and $\ebb_\mr{R}(\mc{C},r;A,x) = \ebb_\mr{R}(\mc{C},r;\Omega,x)*Q$, since $\ebb_\mr{R}(\mc{C},r)$ is equivariant by Lemma \ref{lemmaEbbR}, 
so $b_\mr{test}(A,x) = \lbr(\mc{M},\ebb_\mr{R}(\mc{C},r;A,x)) = b_\mr{test}(\Omega,x)$ for all $x\in[0,1]$ 
since $\lbr(\mc{M}_L)$ is $\mr{O}_3$-invariant, 
so $f_1(0) = f_0(0)$, so $y_t=y_0$ is unchanging. 
Also, $A_{\mr{rec},t} = \crush(\mc{C}_\mr{rec},\Omega;A,t) = A$
since we assume by induction that $\crush(\mc{C}_\mr{rec},\Omega)$ is a strong deformation retraction, 
so $b_t(0) = b_0(0)$ is unchanging,
so $b_t(0) = b_0(0) = y_0 = y_t$, so $s_t=0$, 
so $A_t = \ebb_\mr{R}(\mc{C},r;A,0) = A$ in this case. 

Finally, 
We have $A_0 = A$ by Claim \ref{claimCrushZero}. 
Also, $s_1 = 0$ by Claim \ref{claimCrushZero}, 
so $A_1 = \ebb_\mr{R}(\mc{C},r;A_{\mr{rec},1},0) = A_{\mr{rec},1}$ 
since $\ebb_\mr{R}(\mc{C},r)$ is a deformation by Lemma \ref{lemmaEbbR}, 
and $A_{\mr{rec},1}$ is in the orbit of $\Omega$ by induction, 
so $\crush(\mc{C},\Omega)$ is a strong deformation retraction to the orbit of $\Omega$. 
Also, $\crush(\mc{C},\Omega)$ is equivariant since $\ebb_\mr{R}(\mc{C},r)$ is equivariant by Lemma \ref{lemmaEbbR} and $\crush(\mc{C}_\mr{rec},\Omega)$ is equivariant by induction. 
\end{proof}

\subsection{Diff ebb to ratio ebb}

The ratio between maxlit and minbig has the advantage that it distinguishes both membership in $\hood(\mc{M})$ as well as membership in $\psv(\mc{M})$, which was useful in constructing the deformation $\crush$. 
On the other hand, the difference between maxlit and minbig has the advantage that it can be controlled by bounding how far pseudocircles can move, making it easier to construct a deformation.  
Here we switch between this ratio and difference.

\begin{lemma}[Diff ebb]
\label{lemmaEbbD}
Let $\mc{C} = \{\mc{M}_1<\dots<\mc{M}_L\}$ be a nonempty chain of oriented matroids 
and $\eps > 0$.  
Then, 
$\phi = \ebb_\mr{D}(\mc{C},\eps)$ is a strong equivariant deformation retraction 
from $\upa(\mc{M}_L)$ to $\psv(\mc{M}_L)$ satisfying the following: 
\begin{enumerate}
\item 
\label{itemEbbDContinuous}
$\phi(A,t)$ varies continuously as $\eps,A,t$ vary.
\item 
\label{itemEbbDOT}
$\ot(\phi(A,t)) = \ot(A)$ is unchanging for $t<1$. 
\item 
\label{itemEbbDMinbig}
$\minbig(\mc{M}_k,\phi(A,t)) \geq \minbig(\mc{M}_k,A) -\eps$ 
for each $k \in [L]$.
\item 
\label{itemEbbDMaxlit}
$\maxlit(\mc{M}_k,\phi(A,t)) \leq \maxlit(\mc{M}_k,A) +\eps$
for each $k \in [L]$.
\end{enumerate}
\end{lemma}

\begin{definition}[Ratio ebb]
\label{defEbbR}
\[
\ebb_\mr{R}(\mc{C},r;A,t) 
= \ebb_\mr{D}(\mc{C},\eps;A,t)
\]
where 
$M_k = \minbig(\mc{M}_k,A)$, and 
$m_k = \maxlit(\mc{M}_k,A)$, and  
$\eps =\eps(A,r)$ is the minimum among 
$(\nicefrac13)(1-r)M_L$ and   
$(\nicefrac13)(M_k - m_k)$ for $k \in [L-1]$. 
\end{definition}

\begin{proof}[Proof of Lemma \ref{lemmaEbbR}]
Let $A_t = \ebb(\mc{C},r;A,t)$. 
The parameters $\minbig(\mc{M}_k)$ and $\maxlit(\mc{M}_k)$ are continuous by Lemma \ref{lemmaBigLittleContinuous}, 
so $\eps(A,r)$ in the definition of ebb is continuous as function of $r$ and $A$, so $A_t = \ebb(\mc{C},r;A,t)$ is continuous as a function of $r$, $A$, and $t$ by Lemma \ref{lemmaEbbD} part \ref{itemEbbDContinuous}, 
which gives us part \ref{itemEbbContinuous} of Lemma \ref{lemmaEbbR}.

Each $k \in [L-1]_\mb{N}$ satisfies  
$m_k < M_k$ where $m_k = \maxlit(\mc{M}_k,A)$ and $M_k = \minbig(\mc{M}_k,A)$
since $A \in \hood(\mc{M}_k)$, 
and 
$\eps \leq (\nicefrac13)(M_k - m_k)$ 
by choice of $\eps$, 
so 
by Lemma \ref{lemmaEbbD} parts \ref{itemEbbDMinbig} and \ref{itemEbbDMaxlit}, 
we have   
\begin{align*}
\maxlit(\mc{M}_k;A_t) 
& \leq m_k +\eps 
\leq m_k +(\nicefrac13)(M_k - m_k) 
< M_k -(\nicefrac13)(M_k - m_k) 
\leq M_k -\eps \\ 
&\leq \minbig(\mc{M}_k;A_t), 
\end{align*}
so $A_t \in \hood(\mc{M}_k)$. 

Hence, $A_t \in \hood(\mc{C}\setminus\mc{M}_L)\cap\upa(\mc{M}_L)$, 
so $\ebb(\mc{C},r)$ is a strong deformation retraction from 
$\hood(\mc{C}\setminus\mc{M}_L)\cap\upa(\mc{M}_L)$ to $\hood(\mc{C}\setminus\mc{M}_L)\cap\psv(\mc{M}_L)$ 
by Lemma \ref{lemmaEbbD}.
Also, $\minbig(\mc{M}_k)$ and $\maxlit(\mc{M}_k)$ are constant on orbits, 
so $\eps$ is constant on orbits as well, 
so $\ebb(\mc{C},r)$ is equivariant.

Suppose $\lbr(\mc{M}_L,A) \leq r$. Then, 
$m_L \leq rM_L$. 
Also, $\eps < (\nicefrac12)(1-r)M_L$
by choice of $\eps$ in the definition of $\ebb_\mr{R}$, so 
\[
\lbr(\mc{M}_L,A_t) 
\leq \frac{m_L +\eps}{M_L -\eps} 
< \frac{rM_L +(\nicefrac12)(1-r)M_L}{M_L -(\nicefrac12)(1-r)M_L} 
= \frac{(\nicefrac12)(1+r)M_L}{(\nicefrac12)(1+r)M_L} = 1. 
\]
Thus, $A_t \in \hood(\mc{M}_L)$ in the case where $\lbr(\mc{M}_L,A_0) \leq r$,  
so part \ref{itemEbbR} of Lemma \ref{lemmaEbbR} holds.
\end{proof}



%% file: mox-v5-9.tex
\section{Mox}
\label{sectionMox}

In this section, we introduce a map $\mox$ in Lemma \ref{lemmaMox} which sends systems of paths to x-monotone paths, i.e., implicit functions.  
This map will then be used to define the deformation $\ebb_\mr{D}$ from Lemma \ref{lemmaEbbD}.  
Specifically, $\ebb_\mr{D}$
will deform $A \in \upa(\mc{M})$ to $\psv(\mc{M})$ 
by first defining a region $Z(\Sigma)$ of the sphere for each chain of covectors $\Sigma \subset \csph(\mc{M})$; 
see Section \ref{sectionZone}, 
and then defining a deformation separely in each of these regions; see Definition \ref{defEbbD}. 

In the case of edge covectors, the pseudocircles of $A$ corresponding to nonloops of $\sigma^0$  intersect $Z(\sigma)$ as a collection of paths, and $\ebb_\mr{D}$ will deform this collection paths to a single path.  
To do so, we use the map $\mox$ to send these paths to a collection of x-monotone paths, so that we can simply deform each path to the same horizonal segment by scaling each path vertically. 
We will also use $\mox$ to define the regions $Z(\Sigma)$.

\begin{definition}[Pseudomonotonic path system]
\label{defEastwardSystem}
A \df{cardinal region} is a compact contractible set $R \subset \sphere^2$ 
with a set 
$B = \{B_\mr{N},B_\mr{E},B_\mr{S},B_\mr{W}\}$ consisting of $({\leq}1)$-cells indexed by cardinal direction such that 
\begin{itemize}
\item 
$\bigcup B = \partial R$,
\item 
$B_i\cap B_j$ is a single point for consecutive cardinal directions $i,j$, 
\item 
$B_\mr{E}$ and $B_\mr{W}$ are disjoint, 
\item 
either $B_\mr{N}$ and $B_\mr{S}$ intersect at finitely many points without crossing, 
or $R=B_\mr{N}=B_\mr{S}$ is a path with endpoints 
$B_\mr{W}$,  
$B_\mr{E}$.   
\end{itemize} 
We call $B_\mr{N}$ the \df{northern border} of $R$ 
and $B_\mr{NE} = B_\mr{N}\cup B_\mr{E}$ the \df{northeastern border}   
and analogously for the other cardinal and intercardinal directions. 

An \df{pseudomonotonic path system} through $R$ is
a finite family of paths $P = \{P_i \subset R:i \in I\}$ from $B_\mr{W}$ to $B_\mr{E}$ that 
covers $B_\mr{N}$ and $B_\mr{S}$, 
and where each pair of paths 
either coincide, are disjoint, or 
intersect at a single point and cross at that point. 

A \df{latitude} for $P$ is a map $h : J \to \mb{R}$ 
for $J \subseteq I \subseteq [n]_\mb{N}$ 
such that the curves $P_j$ for $j\in J$ 
\begin{itemize} 
\item 
are pairwise either disjoint or the same, 
\item 
appear along $B_\mr{W}$ in the same order as $h_j$.  
\end{itemize} 
A pseudomonotonic path system may include $h$ or not. 
\end{definition}

\begin{lemma}[Mox]
\label{lemmaMox}
Let $E=(R,B,P,h)$ be pseudomonotonic path system as in Definition \ref{defEastwardSystem}. 
Then, 
$
\xi= \mox(E) : R \to  [0,1]_\mb{R} \times \mb{R} 
$
is an embedding of $R$ in $\mb{R}^2$ satisfying the following.
\begin{enumerate}
\item
\label{itemMoxContinuous}
$\xi$ varies continuously in the partial map metric as 
$R$ varies in Hausdorff distance,  
each curve of $B$ and $P$ varies in Fréchet distance, 
and 
$h_j$ varies for each $j\in J$. 
\item 
\label{itemMoxEquivariant} 
$\mox$ is $\orth_3$-equivariant, i.e., 
$\mox(Q(R,B,P),h) = \xi \circ Q^*$.
\item 
\label{itemMoxSides}
$\xi(B_\mr{W}) \subset 0\times \mb{R}$,\quad   
$\xi(B_\mr{E}) \subset 1\times \mb{R}$.\quad    
\item 
\label{itemMoxH}
$\xi(P_{j}) = [0,1]\times h_j$ for each $j\in J$. 
\item 
\label{itemMoxMonotone}
$\xi(P_i)$ is an x-monotone curve for each $i\in{I}$.
\item 
\label{itemMoxAntipodal} 
$\xi(R,\widetilde B,P,-h;p) = (x,-y)$ 
where $(x,y) = \xi(p)$ and 
\[
(\widetilde B_\mr{W},\widetilde B_\mr{E}) = (B_\mr{W},B_\mr{E}), \quad
(\widetilde B_\mr{N},\widetilde B_\mr{S}) = (B_\mr{S},B_\mr{N}).
\]
\end{enumerate}
\end{lemma}

We will construct $\mox$ in Definition \ref{defMox} 
using a pair of foliations $\fol(R,B,P)$ from Lemma \ref{lemmaLSP}, 
which will be sent to lines of slope $1$ and $-1$.   
That is, we partition $R$ into level sets of a northeast-southeast coordinate system, 
which we call \df{leaves}. 

\begin{remark}[Cardinal direction]
\label{remarkCardinalDirection} 
We consider a directed path $L$ through $R$ from $B_\mr{SW}$ to $B_\mr{NE}$ to be directed northward and eastward.  
If $L$ is directed the opposite way, then we consider $L$ to be directed southward and westward. 
The component of $R\setminus L$ that contains $B_\mr{W}\cap B_\mr{N}$ is to the west and to the north of $L$, and the other component is to the south and to the east. 
We also use analogous language for paths from $B_\mr{NW}$ to $B_\mr{SE}$.
\end{remark}

\begin{lemma}[Coordinate foliations]
\label{lemmaLSP}
Given a pseudomonotonic path system $(R,B,P)$ as in Lemma \ref{lemmaMox},  
$(\mc{L}_+,\mc{L}_-) = \lsp(R,B,P)$ 
is a pair of partitions of $R$ into $({\leq}1)$-cells 
called \df{leaves}  
satisfying the following. 
\begin{enumerate}
\item 
\label{itemLSPEndpoints}
Each $L_+ \in \mc{L}_+$ is a path 
from $B_\mr{SW} = B_\mr{S}\cup B_\mr{W}$
to $B_\mr{NE} = B_\mr{N}\cup B_\mr{E}$, 
or a point on $B_\mr{SW}\cap B_\mr{NE}$.  
Also,  
each $L_- \in \mc{L}_-$ is a path 
from $B_\mr{NW} = B_\mr{N}\cup B_\mr{W}$
to $B_\mr{SE} = B_\mr{S}\cup B_\mr{E}$,
or a point on $B_\mr{NW}\cap B_\mr{SE}$. 
\item 
\label{itemLSPEquivariant}
$\lsp$ is $\orth_3$-equivariant. 
\item 
\label{itemLSPContinuous}
The path $L_+(q) \in \mc{L}_+$ through a given point $q$ varies continuously in Fréchet distance 
as $(R,B,P,q)$ vary, 
and likewise for $\mc{L}_-$. 
\item 
\label{itemLSPProduct}
Each pair $(L_+,L_-) \in \mc{L}_+\times \mc{L}_-$ cross at most once and have no other intersection.
\item 
\label{itemLSPEastward}
Each path $P_i \in P$ intersects each $L_+ \in\mc{L}_+$ at most once, 
in which case $P_i \cap L_+$ is either an endpoint of $L_+$ 
or $P_i$ directed eastward crosses $L_+$ from northwest to southeast. 
Likewise for $L_-\in\mc{L}_-$, except $P_i$ crosses $L_-$ from southwest to northeast. 
\item 
\label{itemLSPFinite}
If $L$ is a directed path consisting of an alternating sequence of leaves of $\mc{L}_+$ and $\mc{L}_-$ directed southward, then the sequence is finite.  
\item 
\label{itemLSPAntipodal} 
$\lsp(D,\widetilde B, P) = (\mc{L}_-,\mc{L}_+)$ 
with $\widetilde B$ as in Lemma \ref{lemmaMox}.
\end{enumerate}
\end{lemma}

\begin{definition}[Mox]
\label{defMox}
We first define $\xi = \mox(R,B,P,h)$ in the case where $J=\{0\}$ is a singleton. 
On $P_0$, let 
\[\xi = \param^{-1}(P_0)\times h_0\]
with $P_0$ directed eastward 
and $\param^{-1}$ as in Definition \ref{defPathParam}.  
On the rest of $R$, let $\xi$ be the unique continuous map 
that sends each path of $\mc{L}_+$ to a segment with slope 1, 
and sends each path of $\mc{L}_-$ to a segment with slope $-1$
where $(\mc{L}_+,\mc{L}_-) = \lsp(R,B,P)$. 
In the case where $|J|>1$, we first construct $\xi$ as above, and then modify $\xi$ to send each path $P_j$ to $[0,1]\times h_j$ for $j\in J$ by linear interpolation in each vertical line, and by vertical translation outside the extrema of $h$. 
\end{definition}

\begin{claim}
\label{claimMoxWellDefined}
$\xi$ is well-defined.
\end{claim}

To show $\xi$ is well-defined, we give an equivalent construction 
starting with $P_0 \cup B_\mr{W} \cup B_\mr{E}$ for the case where $J=\{0\}$. 
%
For $q\in B_\mr{W}$ that is north of $P_0$, let 
$Z(q)$ be the path emanating from $q$ that traverses a leaf of $\mc{L}_-$
until either reaching $P_0$, or reaching $B_\mr{E}$, 
in which case $Z(q)$ then traverses a leaf of $\mc{L}_+$ 
until either reaching $P_0$, or reaching $B_\mr{W}$, 
in which case $Z(q)$ continues alternating between leaves of $\mc{L}_-$ and $\mc{L}_+$ in this manner.  
We call $Z(q)$ a \df{zigzag path}.
and we call each such leaf an \df{edge} of $Z(q)$. 

The zigzag path $Z(q)$ can only have finitely many such edges 
by Lemma \ref{lemmaLSP} part \ref{itemLSPFinite},  
and our starting point $q$ is north of $P_0$, 
so $Z(q)$ must end at some point $z(q) \in P_0$. 
If the last edge of $Z(q)$ is a path of $\mc{L}_-$, 
then let $\xi_2(q) = m-1 +\xi_1(z(q))$ 
where $\xi=\xi_1\times\xi_2$ 
and $m$ is the number of edges of $Z(q)$. 
Otherwise, the last edge of $Z(q)$ is a path of $\mc{L}_+$, 
in which case $\xi_2(q) = m -\xi_1(z(q))$. 
Let $Z(q)$ and $\xi_2$ be defined analogously on $B_\mr{E}$.

\begin{claim}
\label{claimMoxBIncreasing}
$\xi_2$ is increasing northward on $B_\mr{W}$ and likewise on $B_\mr{E}$. 
\end{claim}

\begin{proof}
Consider the case where $q_k\in B_\mr{W}$ 
with $q_2$ north of $q_1$ north of $P_0$ 
and where $Z(q_k)$ has 1 edge. 
Then $Z(q_k)$ for each $k$ consists of a single arc of a path in $\mc{L}_-$, 
and $Z(q_2)$ is entirely north of $Z(q_1)$ 
as in Remark \ref{remarkCardinalDirection}  
since $\mc{L}_-$ is a partition, 
which is also east of $Z(q_1)$, 
so $z(q_2)$ is east of $z(q_1)$, 
so $\xi_2(z(q_2)) = \xi_1(z(q_2))>\xi_1(z(q_1)) =\xi_1(z(q_1))$. 
The claim holds by a similar argument on $B_\mr{E}$ and south of $P_0$ in the case where $Z(q_k)$ has a single edge, 
and the claim holds on the rest of $B_\mr{W}$ and $B_\mr{E}$ by induction on the number of edges of $Z(q_k)$. 
\end{proof}

Now consider a point $q \in R$, and let $q_\mr{W}$ be the endpoint of the path $L_+(q) \in \mc{L}_+$ through $q$ on $P_0\cup B_\mr{W}$, and define $q_\mr{E}$ analogously. 

\begin{claim}
\label{claimMoxEWGap}
If $q_\mr{E} \in P_0$, then $|\xi_2(q_\mr{W})| \leq \xi_1(q_\mr{E})$. 
If $q_\mr{W} \in P_0$, then $|\xi_2(q_\mr{E})| \leq 1-\xi_1(q_\mr{W})$. 
Otherwise, $|\xi_2(q_\mr{E}) -\xi_2(q_\mr{W})|\leq 1$. 
\end{claim}

\begin{proof}
Let us assume that $q$ is north of $P_0$; otherwise we can replace $\xi_2$ with $-\xi_2$.

Consider the case where $q_\mr{E} \in P_0$.  
Then, $q_\mr{W}$ is either to the south of or on $L_-(q)$ 
as in Remark \ref{remarkCardinalDirection}, 
which is also to the west of $L_-(q)$.  
Hence, either $q_\mr{W} \in P_0$, in which case $\xi_2(q_\mr{W}) = 0$, 
or $q_\mr{W} \in B_\mr{W}$, 
in which case $L_-(q_\mr{W})$ is either entirely to the west of $L_-(q)$ or coincides with $L_-(q)$ 
since $\mc{L}_-$ is a partition. 
Also, the arc of $B_\mr{E}$ that is to the north of $P_0$ is entirely east of $L_-(q)$,  
so $L_-(q_\mr{W})$ meets $P_0$ at $z(q_\mr{W})$, which is either at or to the west of $q_\mr{E}$, 
so $\xi_2(q_\mr{W}) = \xi_1(z(q_\mr{W})) \leq \xi_1(z(q_\mr{E}))$. 
The claim follows by a similar argument  
in the case where $q_\mr{W} \in P_0$.

Consider the case where $q_\mr{E},q_\mr{W} \not\in P_0$.  
Then, $q_\mr{W} \in B_\mr{W}$ and $q_\mr{E} \in B_\mr{E}$. 
Suppose we had $\xi_2(q_\mr{W}) > \xi_2(q_\mr{E})+1$, 
and let $q_\mr{WE}$ be the endpoint of $L_-(q_\mr{W})$ that is not on $B_\mr{W}$. 
Then, we would have $\xi_2(q_\mr{WE}) = \xi_2(q_\mr{W})-1 > \xi_2(q_\mr{E})$, 
but $q$ is to the north of $L_-(q_\mr{W})$, 
so $q_\mr{E}$ is to the north of $q_\mr{WE}$, 
which contradicts Claim \ref{claimMoxBIncreasing}. 
Hence, we must have $\xi_2(q_\mr{W}) \leq \xi_2(q_\mr{E})+1$, 
and by a similar argument we have $\xi_2(q_\mr{E}) \leq \xi_2(q_\mr{W})+1$, 
so $|\xi_2(q_\mr{E}) -\xi_2(q_\mr{W})|\leq 1$. 
\end{proof}

\begin{proof}[Proof of Claim \ref{claimMoxWellDefined}]
By Claim \ref{claimMoxEWGap}, $\xi(q_\mr{W})$ and $\xi(q_\mr{E})$ are close enough vertically to be connected by a unimodal path $\Lambda$ with slope 1 and $-1$.  Let $\xi(q)$ be the coordinates at the highest point of $\Lambda$. 
The point $q_\mr{W}$ is the same for each choice of $q$ on a path $L_+\in\mc{L}_+$, 
so the increasing portion of $\Lambda$ is on the same line for each $q \in L_+$, 
so the image of $L_+$ by $\xi$ is a segment of slope 1,
and similarly the image of $L_- \in \mc{L}_-$ has slope $-1$, 
which is the defining feature of $\xi$ on $R\setminus P_0$. 
Thus, $\xi$ in Definition \ref{defMox} is equivalent to the construction of $\xi$ using zigzag paths.
\end{proof}

\begin{proof}[Proof of Lemma \ref{lemmaMox}]
We have that $\xi$ is equivariant and varies continuously on $S_0$ by Lemma \ref{lemmaPathParam}, 
and the levels sets $(\mc{L}_+,\mc{L}_-)$ vary equivariantly and continuously 
by Lemma \ref{lemmaLSP}, 
so $\xi$ varies equivariantly and continuously by induction on the number of edges of $Z(q_\mr{W})$ and $Z(q_\mr{E})$, 
which means that parts \ref{itemMoxContinuous} and \ref{itemMoxEquivariant} hold. 
We have 
$\mu(B_\mr{W}) \subset 0\times \mb{R}$, and 
$\mu(B_\mr{E}) \subset 1\times \mb{R}$, and     
$\mu(P_{j}) = [0,1]\times h_j$ for each $j\in J$ 
by construction, so parts \ref{itemMoxSides} and \ref{itemMoxH} hold.

Since $\xi$ sends each path of $\mc{L}_+$ and $\mc{L}_-$ to a segment of slope 1 and $-1$ respectively by construction, 
and each path of $\mu({P_i})$ only crosses these segments from east to west 
by Lemma \ref{lemmaLSP} part \ref{itemLSPEastward}, 
the path $\mu(P_i)$ is the graph of a 1-Lipschitz function, and in particular is x-monotone, 
which means that part \ref{itemMoxMonotone} holds.  

Let $\widetilde \xi$ be defined as in Definition \ref{defMox}, 
but with $B$ replaced with $\widetilde B$, and $S^+$ with $-S^{+}$, and $h$ with $-h$ 
as in part \ref{itemMoxAntipodal}. 
Replacing $S^+$ with $-S^+$ in the definition of $\xi$ on $P_0$ swaps $\xi_+$ and $\xi_-$, 
so $\xi_1 = \tfrac{\xi_++\xi_-}{2}$ is unchanged on $P_0$, 
so $\widetilde \xi = (\xi_1,-h_0) = [(x,y) \mapsto (x,-y)] \circ \xi$, 
so part \ref{itemMoxAntipodal} holds on $P_0$. 
The map $[(x,y) \mapsto (x,-y)]$ swaps segments with slope 1 and segments with slope $-1$, 
and replacing $B$ with $\widetilde B$ in $\lsp(D,B,P)$ swaps $\mc{L}_+$ and $\mc{L}_-$
by Lemma \ref{lemmaLSP} part \ref{itemLSPAntipodal}, 
so $[(x,y) \mapsto (x,-y)]\circ \widetilde \xi$ 
sends $\mc{L}_+$ to the class of segments with slope 1 
and sends $\mc{L}_-$ to the class of segments with slope $-1$, 
and agrees with $\xi$ on $P_0$, which are the defining characteristics of $\xi$, 
so part \ref{itemMoxAntipodal} holds.
\end{proof}

%% file: fol-v5-9.tex
\subsection{Coordinate foliations} 

In this subsection we define the function $\fol$ and prove Lemma \ref{lemmaLSP}. 
We will define foliations in each cell of $R$ subdivided by $P$ using conformal maps, 
and then glue the foliations of adjacent cells together. 
We first prove two claims to show the existence and uniqueness of the northeast most point of a cell, which will be used in the definition. 

Let $\mc{C}$ be the set of closed 2-cells of $R$ subdivided by $P$.

\begin{claim}
\label{claimLSPNoCycle}
$B\cup P$ has no northeastward directed cycles. 
Likewise southeastward. 
\end{claim}

\begin{proof}
A northeastward directed cycle cannot include $B_\mr{W}$ since there are no paths directed toward $B_\mr{W}$. 
Also, a northeastward directed cycle cannot include $B_\mr{E}$ since there are no paths directed away from $B_\mr{W}$. 
Hence, we are left with paths of $P$. 

Suppose we had such a cycle $\partial C$ with edges $E_1,\dots,E_m$ along paths $P_1,\dots,P_m$ in that order directed northeastward. 
We will get a contradiction from the order that the paths intersect $B_\mr{W}$. 

We may assume that $\partial C$ is a simple closed curve; otherwise traverse the cycle until the first time a point is repeated, and take the directed cycle at that point. 
Then, $\partial C$ bounds a cell $C \in \mc{C}$. 
Otherwise, if there is some path $P_j$ passing through the interior of $C$, 
we can find a new directed cycle $\partial C'$ that bounds a proper subset $C' \subset C$ 
on one side of $P_j$, and by induction on the number of cells in $C$, we get a minimal such $C$, which is a cell of $R$. 

Consider how $P_1,P_2$ cross at the vertex $v$ of the cycle. 
We may assume by symmetry that $E_2$ is north of $P_1$. 
Then, $P_2\cap B_\mr{W}$ must be south of $P_1\cap B_\mr{W}$ along $B_\mr{W}$. 
Also, $C$ is on or to the north of $P_2$ since $C$ is a cell, so $E_3$ must be north of $P_2$, 
so $P_3\cap B_\mr{W}$ must be south of $P_2\cap B_\mr{W}$ along $B_\mr{W}$, 
and continuing likewise, 
$P_{j+1}\cap B_\mr{W}$ must be south of $P_j\cap B_\mr{W}$, 
and $P_1\cap B_\mr{W}$ must be south of $P_m\cap B_\mr{W}$, 
which is south of $P_1\cap B_\mr{W}$ along $B_\mr{W}$, 
which is a contradiction. 
\end{proof}

\begin{claim}
\label{claimLSPSource}
The boundary of a cell $C \in \mc{C}$ directed northeastward has a unique source and a unique sink. 
Likewise directed southeastward. 
\end{claim}

\begin{proof}
The boundary of a cell $C$ has at least one sink and source by Claim \ref{claimLSPNoCycle}. 
Suppose there were more than one.  
Directing curves northeastward defines a partial order on $P\cup B$ 
by Claim \ref{claimLSPNoCycle},  
and we extend this to a total order $(\leq_\mr{NE})$. 
Then, the sources and sinks are the respective local minima and maxima with respect to $(\leq_\mr{NE})$, which alternate in the cyclic order around $\partial C$.  
We can find sources $\{s_1,s_2\}$ and sinks $\{t_1,t_2\}$ alternating around $\partial C$ such that  
$\{s_1,s_2\} <_\mr{NE} \{t_1,t_2\}$. 
For instance, let $t_1$ be a global maximum and $s_1,s_2$ be the sources adjacent to some other sink $t_2$. 
We can find a walk $W$ from $s_1$ to $s_2$ by first traversing a path of $P$ from $s_1$ southwest to $B_\mr{W}$, then traversing $B_\mr{W}$ to a path of $P$ that we can traverse northeast to $s_2$.  Every point of this walk is either west along a path through $s_1$ or $s_2$ or south of such a path along $B_\mr{W}$, so $W \leq_\mr{NE} \{s_1,s_2\}$.  Then, $W$ is connected, so $W$ has a path $P_\mr{s}$ from $s_1$ to $s_2$. 
Similarly, we can find a path $P_\mr{t}$ from $t_1$ to $t_2$ such that 
\[
P_\mr{s} \leq_\mr{NE} \{s_1,s_2\} <_\mr{NE} \{t_1,t_2\} \leq_\mr{NE} P_\mr{t},  
\]
so the paths $P_\mr{s}$ and $P_\mr{t}$ cannot cross each other, and cannot cross any edge of the cycle $\partial C$, 
since $\partial C$ bounds a cell, so $\partial C \cup P_\mr{s} \cup P_\mr{t}$ is a planar embedding of a 4-clique. 
We can then add a point $c \in C^\circ$ with disjoint paths from $c$ to each of the points $s_1,s_2,t_1,t_2$ in $C$ 
to get a planar embedding of a 5-clique, which is impossible. 
Thus, $\partial C$ has one source and one sink.   
\end{proof}

\begin{definition}[Coordinate foliations]
\label{defLSP}
We first define the foliations on a cell $C \in \mc{C}$ where $\mc{C}$ is the set of closed 2-cells of $R$ subdivided by $P$. 
Let $p_\mr{SW}$ be the source and $p_\mr{NE}$ be the sink of $\partial C$ directed northeastward as in Claim \ref{claimLSPSource}, and define $p_\mr{NW}$ and $p_\mr{SE}$ analogously. 
Let $C_\mr{N}$ be the path from $p_{NW}$ to $p_{NE}$ 
with $C$ to the south 
as in Remark \ref{remarkCardinalDirection}, 
and let $C_\mr{NE} = C_\mr{N}\cup C_\mr{E}$ 
with other cardinal and intercardinal directions defined analogously.

For $k \in \{+,-\}$, let 
$h_k : C \to \mb{C}$ be an internally isogonal map sending $C$ to the upper half-plane of $\mb{C}$ satisfying the following.  
\begin{align*}
h_+(C_\mathrm{SW}) &= [0,\infty]_\mb{R}, & 
h_+(C_\mathrm{NE}) &= [-\infty,0]_\mb{R}, & 
h_+(p_\mathrm{NW}) &= 0, & 
h_+(p_\mathrm{SE}) &= \infty, \\
h_-(C_\mathrm{SE}) &= [0,\infty]_\mb{R}, & 
h_-(C_\mathrm{NW}) &= [-\infty,0]_\mb{R}, & 
h_-(p_\mathrm{SW}) &= 0, & 
h_-(p_\mathrm{NE}) &= \infty. 
\end{align*}
Note that the $h_k$ are only defined up to a positive scalar multiple. 
Let $\Theta(c,r)$ be the semicircle with center $c$ and radius $r$. 
Let 
\[\mc{L}_+(C) = \{h_+^{-1}(\Theta(-r,3r)) : r\in [0,\infty]\}, \quad 
\mc{L}_-(C) = \{h_-^{-1}(\Theta(r,3r)) : r\in [0,\infty]\}.\]
Let each path of $\mc{L}_+(C)$ be directed northeast from $C_\mr{SW}$ to $C_\mr{NE}$, 
and let $\mc{L}_+$ be the set of all maximal directed paths $L$ 
formed by a union of paths among $\mc{L}_+(\mc{C}) = \bigcup\{\mc{L}_+(C) : C \in \mc{C}\}$, 
define $\mc{L}_-$ analogously, 
and let $\lsp(R,B,P)=(\mc{L}_+,\mc{L}_-)$.  
\end{definition}

\begin{remark}
\label{remarkLSPhScalar} 
$h_+$ is only defined up to a positive scalar multiple, 
but $\mc{L}_+$ does not depend on the choice of scalar multiple. 
Likewise for $h_-$ and $\mc{L}_-$.
Also, the $h_k$ are internally conformal provided that east is clockwise of north, 
which we may assume by symmetry. 
\end{remark}

\begin{remark}
\label{remarkLSPEWDisjoint}
Note that on $C_\mr{W}$ northeastward coincides with northwestward,  
so $C_\mr{W}$ is either an arc along $B_\mr{W}$ or a single point, 
which is the source directed northeastward and southeastward. 
Similarly, $C_\mr{E}$ is either an arc along $B_\mr{E}$ of a single point,  
which is the sink directed northeastward and southeastward, 
so $C_\mr{W}$ and $C_\mr{E}$ must be disjoint. 
Hence, $C_\mr{N}$ and $C_\mr{S}$ are 1-cells, 
$p_\mr{NW}\neq p_\mr{NE}$, 
and $p_\mr{SW}\neq p_\mr{SE}$. 
\end{remark}

\begin{justification}
We use preimages of $\Theta(-r,3r)$ and $\Theta(r,3r)$ 
rather than the level sets of $|h_+|$ and $|h_-|$ 
because in some cases we can have $p_\mr{NW}=p_\mr{SW}$ and $p_\mr{NE}=p_\mr{SE}$, 
in which case we have $h_+=h_-$, but we cannot allow $\mc{L}_+$ and $\mc{L}_-$ to coincide in a cell. 
\end{justification}

We will show in Claim \ref{claimLSPMaxPath} that $\mc{L}_+$ and $\mc{L}_-$ are partitions. 
To prove this, we first partition $D$ into regions that determine how paths of $\mc{L}_+$ continue northeastward. 
Let 
\[
\mc{D}_\mr{NE} = \{ B_\mr{NE},\ C{\setminus} C_\mr{NE} : C \in \mc{C} \}. 
\]

\begin{claim}
\label{claimLSPHalfOpenCells}
$\mc{D}_\mr{NE}$ is a partition of $R$, and likewise for $\mc{D}_\mr{SE}$ defined analogously. 
\end{claim}

\begin{proof}
$B_\mr{NE}$ is only on the northeast boarder of cells by definition, so $B_\mr{SE}$ does not intersect any other set of $\mc{D}_\mr{NE}$. 
If $q$ is in the interior of a cell $C$, then $q \in C\setminus C_\mr{NE}$ and is not in any other cell, so $q$ is in a unique set of $\mc{D}_\mr{NE}$. 
If $q$ is on an edge of $P \cup B$ that is not along $B_\mr{N}$ or $B_\mr{E}$, 
then $q$ is only on the southwest boarder of one cell $C$ and possibly on the northeast boarder of another, 
so $q$ is in a unique set of $\mc{D}_\mr{NE}$. 

One case remains, which is where $q$ is at a point where paths of $P \cup B$ meet. 
Consider the case where $q \in B_\mr{W}$ and $q$ is in one or more path of $P$. 
Then each edge incident to $q$ directed southeastward is directed outward
except for the edge $E$ extending northward along $B_\mr{W}$, 
so $q$ is the southeastward source $p_\mr{NW}$ for each adjacent cell $C$ except the cell $C_0$ that 
contains the edge $E$, 
and $p_\mr{NW}$ is an endpoint of $C_\mr{NE}$, so $q \not\in C \setminus C_\mr{NE}$. 
There is only one cell $C_0$ that contains $E$ since $E$ is on the boundary of $R$, 
so $q$ is in a unique set of $\mc{D}_\mr{NE}$. 

Consider the case where $q$ is at a point where paths of $P$ cross. 
Note that this also includes the case where $q \in B_\mr{S}$, since $P$ covers $B_\mr{S}$. 
Let $P_1,\dots,P_n$ be the paths ordered by the position of their endpoints on $B_\mr{W}$ directed northward. 
Then, the cyclic order of eastward directed edges around $q$ consists of edges on the paths $P_1,\dots,P_n$ directed inward to $q$ and then paths $P_1,\dots,P_n$ directed outward from $q$ 
since the paths of $P$ pairwise intersect at a single point where they cross. 
Hence, $q$ is either a source or sink for each cell incident to $q$ 
except the pair of cells incident to $P_1$ and $P_n$. 
One cell has $q$ on its northern border, 
and the other has $q$ on its southern border, 
since the paths separate $B_\mr{N}$ and $B_\mr{S}$.  
Thus, $q$ is in a unique member of $\mc{D}_\mr{NE}$. 
\end{proof}

\begin{claim}
\label{claimLSPMaxPath}
Each point $q \in R$ is in at most one maximal path $L_+(q) \in \mc{L}_+$. 
Hence, $\mc{L}_+$ is a partition of $R$ into $({\leq}1)$-cells, and likewise for $\mc{L}_-$. 
\end{claim}

Let $L_+(C,q)$ be the path of $\mc{L}_+(C)$ through $q \in C$ 
and let $L_+(q)$ be the path of $\mc{L}_+$ through $q \in R$.  

\begin{proof} 
Suppose for the sake of contradiction that $q$ is in two distinct maximal paths $L_1,L_2\in \mc{L}_+$. 
Since these paths are distinct, there must be some $p_1 \in L_1 \setminus L_2$, 
and we may assume by symmetry that $p_1$ is northeast of $q$ along $L_1$. 
Since $L_2$ is maximal, there must also be a point $p_2 \in L_2 \setminus L_1$ 
that is northeast of $q$ along $L_2$; 
otherwise we could extend $L_2$ with the arc of $L_1$ from $q$ to $p_1$. 
Let $q_0$ be the last point on the arc $A_1$ of $L_1$ from $q$ to $p_1$ that is on $L_2$. 

Then, $L_1$ and $L_2$ must continue northeastward along separate paths of the form 
$L_+(C_1,q_0)$ and $L_+(C_2,q_0)$ by construction, 
and $q_0$ cannot be the northeast endpoint of $L_+(C_1,q_0)$, 
so $q_0 \in C_1 \setminus C_{1,\mr{NE}}$, and likewise $q_0 \in C_2 \setminus C_{2,\mr{NE}}$, 
so $C_1=C_2$ by Claim \ref{claimLSPHalfOpenCells}, 
so $L_1$ and $L_2$ continue along the same path, 
which is a contradiction.
\end{proof}

\begin{proof}[Proof of Lemma \ref{lemmaLSP} parts \ref{itemLSPEndpoints} and \ref{itemLSPEquivariant}]
Consider the case where $q$ is in a path $L\in\mc{L}_+$. 
Then, the sink $t$ of $L$ directed northeastward cannot be on the southwest boarder of any cell; 
otherwise $L$ could be extended in that cell, which would contradict the maximality of $L$. 
Hence, $t \in B_\mr{NE}$, and similarly the source of $L$ is in $B_\mr{SW}$. 
In the case where $q$ is not in a path of $\mc{L}_+$, 
then $q$ cannot be in the southwest boarder of cell or in the northeast boarder, so $q\in B_\mr{SW}\cap B_\mr{NE}$.
An analogous argument applies for $\mc{L}_-$, 
so part \ref{itemLSPEndpoints} holds.

Consider a cell $C\in\mc{C}$ of $R$ subdivided by $P$, and $Q \in\orth_3$. 
Then, $QC\in Q\mc{C}$ is a cell of $QD$ subdivided by $QP$, 
and $h_+Q^{-1}$ is an internally angle preserving map from $QC$ to the upper half-plane 
that sends $Q C_\mr{SW}$ to $[0,\infty]_\mb{R}$ 
and analogously for the other defining conditions of $h_+$, 
and for each $L\in\mc{L}_+$ there is some $r\in[0,\infty]$ such that 
$h_+(L) = \Theta(-r,3r)$, 
so $h_+Q^{-1}(QL) = \Theta(-r,3r)$, 
and analogously for $\mc{L}_-$. 
Thus, $\lsp(Q(R,B,P))=(Q\mc{L}_+,Q\mc{L}_-)$, 
which means part \ref{itemLSPEquivariant} holds.
\end{proof}

\begin{proof}[Proof of Lemma \ref{lemmaLSP} part \ref{itemLSPContinuous}]
Consider a sequence of valid inputs $(D_k,B_k,P_k,q_k)$ to $\lsp$ that converge to $(D_\infty,B_\infty,P_\infty,q_\infty)$, 
and let $(\mc{L}_{+,k},\mc{L}_{-,k}) = \lsp(D_k,B_k,P_k,q_k)$ and similarly add a corresponding subscript to each object in Definition \ref{defLSP}. 

Let $L_k = L_{+,k}(q_k)$, 
and consider the case where $L_k=q_k$ infinitely often.
Then, $q_k = L_k \in B_{\mr{SW},k}\cap B_{\mr{NE},k}$,
and since $q_k$ and $B_k$ converge appropriately, 
$q_\infty \in B_{\mr{SW},\infty}\cap B_{\mr{NE},\infty}$, 
so $L_\infty=q_\infty$ is the limit point of $L_k$. 
Let us then assume that $L_k$ is not a point.

Suppose for the sake of contradiction that $L_k = L_{+,k}(q_k)$ does not converge to 
$L_\infty$ 
in Fréchet distance. 
Then, we may assume that $L_k$ crosses the paths $P_k$ in a fixed order for $k<\infty$; 
otherwise restrict to such a subsequence. 
If each edge of $L_k$ subdivided by $P_k$ were to converge to a corresponding edge or vertex of $L_\infty$ in Fréchet distance, then $L_k$ would converge to $L_\infty$ in Fréchet distance, so there must be some edge $L'_k$ of $L_k$ that does not converge to an edge or vertex of $L_\infty$.

Let us choose $L'_k$ to be separated from $q_k$ by the fewest number of paths of $P_k$ among arcs that do not converge appropriately.
Then, either $q_k \in L'_k$ or $L'_k$ shares an endpoint with an arc that does converge appropriately, so in either case some point $q'_k$ of $L'_k$ converges to a point $q'_\infty$ on $L_\infty$.  Let us assume that $q_k \in L'_k$; otherwise we will get a contradiction by the same argument with $q'_k$ instead $q_k$. 
Hence, $q_k$ is in a cell $C_k \in \mc{C}_k$ with $I_\mr{N}(C_k)$ and $I_\mr{S}(C_k)$ fixed 
and $L'_k = L_{+,k}(C_k,q_k)$ for $k<\infty$, 
and $L'_k$ neither converges to 
an edge nor to a vertex of $L_\infty$ in Fréchet distance that is on or between the corresponding paths of $P_\infty$. 

If $C_k$ converges to a point 
in Hausdorff distance, 
then $C_k$ converges to $q_\infty$, which is a vertex of $P_\infty$,  
so $L'_k \subset C_k$ converges to a vertex of $P_\infty$, which is a contradiction. 
Therefore, we shall assume that $C_k$ does not converge to a point.    

Let us restrict to a subsequence where the cyclic order of paths along $\partial C_k$ is fixed for $k<\infty$.  Since $\sphere^2$ is compact, we may restrict to a subsequence where each vertex of $\partial C_k$ converges. 
Then, 
$p_{\mr{NW},k}$ and $p_{\mr{SE},k}$ respectively converge to points 
$\widetilde p_{\mr{NW}}$ and $\widetilde  p_{\mr{SE}}$, 
and $\widetilde p_{\mr{NW}} \neq \widetilde  p_{\mr{SE}}$, 
otherwise $C_k$ would converge to a point. 
Also, 
$C_{\mr{NE},k}$ and $C_{\mr{SW},k}$ respectively converge to paths 
$\widetilde C_\mr{NE}$ and $\widetilde C_\mr{SW}$ 
from $\widetilde p_{\mr{NW}}$ to $\widetilde  p_{\mr{SE}}$ 
in Fréchet distance by Lemma \ref{lemmaConvergentPathSequence}.

Consider the case where $q_k = p_{\mr{NW},k}$. 
Then, $L'_k = q_k$, and 
$\widetilde p_{\mr{NW}} = q_\infty$ is a vertex of $L_\infty$, 
since $p_{\mr{NW},k}$ is on a path of $P_k$, so $\widetilde p_{\mr{NW}}$ is on a path of $P_\infty$.  Hence, $L'_k$ converges to a vertex of $L_\infty$.  
A similar argument shows convergence in the case where $q_k = p_{\mr{SE},k}$. 
Therefore, let us assume that $q_k \not\in\{p_{\mr{NW},k},p_{\mr{SE},k}\}$.

We split into cases depending on whether $\widetilde C_\mr{NE}$ and $\widetilde C_\mr{SW}$ 
intersect at a point other than $\widetilde p_{\mr{NW}}$ or $\widetilde p_{\mr{SE}}$.

Consider the case where there is a point $v \in\widetilde C_\mr{NE} \cap \widetilde C_\mr{SW} \setminus \{\widetilde p_{\mr{NW}},\widetilde p_{\mr{SE}}\} \neq \emptyset$. 
We claim that $\widetilde C_\mr{NE}=\widetilde C_\mr{SW}$ in this case. 
Since $C_{\mr{NE},k} \to \widetilde C_\mr{NE}$ in Fréchet distance, 
there is a sequence of points $v_{k} \in C_{\mr{NE},k}$ such that 
$v_{k} \to v$, 
so there is some path $P_{i,k} \in P_k$ along $C_{\mr{NE},k}$ such that $v_{k} \in P_{i,k}$ infinitely often, so $v \in P_{i,\infty}$. 
Similarly, there is some $P_{j,k}$ along $C_{\mr{SW},k}$ such that $v \in P_{j,\infty}$. 

Suppose $P_{i,\infty}$ and $P_{j,\infty}$ crossed at $v$. 
Let us assume by symmetry that $P_{i,\infty}$ directed eastward crosses $P_{j,\infty}$ from north to south at $v$. 
Then, $\widetilde p_\mr{SE}$ would be east of $v$ along paths of $P_\infty$, so $\widetilde p_\mr{SE}$ would be 
on or north of $P_{i,\infty}$, 
but $p_{\mr{SE},k}$ is 
on or south of $P_{i,k}$, 
so $\widetilde p_\mr{SE}$ would have to be on $P_{i,\infty}$, 
and similarly $\widetilde p_\mr{SE}$ would be on $P_{j,\infty}$, 
so $P_{i,\infty}$ and $P_{j,\infty}$ meet at $\widetilde p_\mr{SE}\neq v$,
which contradicts the condition that distinct curves of $P_\infty$ intersect in at most one point where they cross. 
Hence, $P_{i,\infty}$ and $P_{j,\infty}$ cannot cross at $v$, so we must have $P_{j,\infty} = P_{i,\infty}$.

Also, $C_{\mr{NE},k} \cup C_{\mr{SW},k}$ is on or south of $P_{i,k}$, 
so $\widetilde C_\mr{NE}\cup \widetilde C_\mr{SW}$ is on or south of $P_{i,\infty}$, 
and similarly $\widetilde C_\mr{NE}\cup \widetilde C_\mr{SW}$ is on or north of $P_{j,\infty}$, 
so $\widetilde C_\mr{NE}\cup \widetilde C_\mr{SW}$ is on $P_{i,\infty} = P_{j,\infty}$, 
and $C_{\mr{NE},k}$ is an eastward directed path from $p_{\mr{NW},k}$ to $p_{\mr{SW},k}$, 
so $\widetilde C_\mr{NE}$ is an eastward directed arc from $\widetilde p_{\mr{NW}}$ to $\widetilde p_{\mr{SE}}$ along $P_{i,\infty}$ and likewise for $\widetilde C_\mr{SW}$. 
Hence, $\widetilde C_\mr{NE}=\widetilde C_\mr{SW}$.

Let $\eta(z) = \frac{\im-z}{\im+z}$ be the conformal map sending the upper half-plane to the unit disk, so $\eta(0) = 1$ and $\eta(\infty) = -1$ and $\eta(\im)=0$. 
Then, $[\eta h_{+,k}]^{-1}$ satisfies the conditions of Lemma \ref{lemmaSquishConformalPath}.
Recall that $h_{+,k}$ is only defined up to a positive scalar factor, 
and let us choose this factor so that $|h_{+,k}(q_k)|=1$. 
Let 
$\widehat q_k$ be the point on $L'_k$ that is farthest away from $q_k$. 
Then, $\nicefrac12 \leq|h_{+,k}(\widehat q_k)|\leq 2$ by the definition of $\mc{L}_+(C)$ 
in Definition \ref{defLSP}, 
so $|h_+(\widehat q_k)|$ is bounded away from 0 and $\infty$ in the spherical metric, 
so $\eta h_+(q_k)$ and $\eta h_+(\widehat q_k)$ are both bounded away from 1 and $-1$, 
so $\widehat q_k \to q_\infty$ by Lemma \ref{lemmaSquishConformalPath}. 
Hence $L'_k \to q_\infty$ 
since $\widehat q_k$ is the farthest point from $q_k$ on $L'_k$, 
and $q_\infty$ is the vertex of $L_\infty$ at the point of intersection with $P_{i,\infty}$, 
which is a contradiction since we chose $L'_k$ to be an edge that does not converge to a vertex of $L_\infty$. 
This completes the case where $\widetilde C_\mr{NE} \cap \widetilde C_\mr{SW} \setminus \{\widetilde p_{\mr{NW}},\widetilde p_{\mr{SE}}\} \neq \emptyset$.

Consider the case where $\widetilde C_\mr{NE} \cap \widetilde C_\mr{SW} \setminus \{\widetilde p_{\mr{NW}},\widetilde p_{\mr{SE}}\} = \emptyset$. 
Then, $\partial C_k \to \partial C_\infty$ in Fréchet distance for a cell $C_\infty \in \mc{C}_\infty$, 
so with an appropriate choice of scalar factor in the definition of $h_+$, 
we have $[\eta h_{+,k}]^{-1} \to [\eta h_{+,\infty}]^{-1}$ in the sup-metric 
by Radó's theorem, 
so $h_{+,k}^{-1} \to h_{+,\infty}^{-1}$ with the spherical metric on $\overline{\mb{C}}$, 
and so $h_{+,k} \to h_{+,\infty}$ in the partial map topology by Lemma \ref{lemmaCPDistContinuity},
so $h_{+,k}(q_k) \to h_{+,\infty}(q_\infty)$, 
so the sequence of semicircles $\Theta(-r_k,3r_k)$ that pass through $h_{+,k}(q_k)$ converge appropriately in Fréchet distance, 
so $L'_k \to L_+(C_\infty,q_\infty)$ in Fréchet distance by Lemma \ref{lemmaPMFrechetDistance}, 
and $L_+(C_\infty,q_\infty)$ is an edge of $L_\infty$, 
which is a contradiction. 

Thus, $L_k = L_{+,k}(q_k)$ converges to $L_\infty$ in Fréchet distance, 
and by a similar argument $L_{-,k}(q_k)$ converges to $L_{-,\infty}(q_\infty)$.
\end{proof}

\subsection{Journey to the boarder}

To prove the rest of Lemma \ref{lemmaLSP}, 
we need to follow the paths of $\mc{L}_+$ or $\mc{L}_-$ to the boarder of $R$. 
By symmetry, it suffices to follow one of these. 

For a point $q \in C \in \mc{C}$, let $q_\mr{NE} = L_+(C,q) \cap C_\mr{NE}$. 
Let $\hat g_\mr{NE} : R \to R$ by 
$\hat g_\mr{NE}(q) = q_\mr{NE}$ as above for the cell $C$ 
such that $q \in C \setminus C_\mr{NE}$. 
Note that $\hat g_\mr{NE}$ is well-defined on $R$ by Claim \ref{claimLSPHalfOpenCells}. 
Let $g_\mr{NE} : R \to B_\mr{NE}$ by 
\[
g_\mr{NE}(q) = \begin{cases}
q & q \in B_\mr{NE} \\ 
g_\mr{NE}(\hat g_\mr{NE}(q)) & q \in B_\mr{NE}. 
\end{cases}
\]
Define the corresponding points for the other intercardinal directions analogously. 
Let $I_\mr{S}(q)$ denote the set of $i$ such that $q$ is to the south of $P_i$ 
as in Remark \ref{remarkCardinalDirection} 

\begin{remark}
\label{remarkLSPIS} 
If $q \in R \setminus B_\mr{NE}$, then 
$I_\mr{S}(\hat g_\mr{NE}(q)) \subset I_\mr{S}(q)$ 
since $q$ is in a cell $C$ that is to the south of some path through $\hat g_\mr{NE}(q)$ 
and $q$ is not on that path, 
and $q$ is not on or separated from $\hat g_\mr{NE}(q)$ any path north of $\hat g_\mr{NE}(q)$, 
since $\hat g_\mr{NE}(q)\in C$ and $q \not\in C_\mr{NE}$. 
Therefore, iterating $\hat g_\mr{NE}$ eventually stabilizes at a point on $B_\mr{NE}$.  
Hence, $g_\mr{NE}$ is well-defined. 
\end{remark}

\begin{claim}
\label{claimLSPqNE}
$L_+(C,q)\cap L_-(C,q) = q$. 
Also, $q_\mathrm{NE}$ is either northeast of $q_\mathrm{NW}$ or $q_\mathrm{SE}$ along $\partial C$, and $q_\mathrm{NE}$ is strictly northeast of one of these points unless $q \in C_\mr{NE}$. 
This holds analogously for the other intercardinal directions.
\end{claim} 

\begin{claim}
\label{claimLSPMonotone}
If $q_1$ is strictly southeast of $q_0$ 
along paths of $P$, $B$, and $\mc{L}_-$, 
then $g_\mr{NE}(q_1)$ is strictly southeast of $g_\mr{NE}(q_0)$.  
\end{claim}

\begin{proof}
Let $L_i = L_+(C,q_i)$, and $K_i = h_+(L_i)$,.

Consider the case where $q_1,q_0 \in C_\mr{SW}$ for some cell $C$. 
Then, $h_+(q_1) > h_+(q_0) > 0$, so 
the semicircle $K_0$ is contained in the half-disk bounded by $K_1$, 
so $h_+\hat g_\mr{NE}(q_1) < h_+\hat g_\mr{NE}(q_0) < 0$, 
so $\hat g_\mr{NE}(q_1)$ is strictly southeast of $\hat g_\mr{NE}(q_0)$, 
so $g_\mr{NE}(q_1)$ is strictly southeast of $g_\mr{NE}(q_0)$ 
by induction on the number of iterations for $\hat g_\mr{NE}$ reach $B_\mr{NE}$. 

Next, consider the case where $q_1,q_0\in L \in \mc{L}_-$ 
in the same region $C \setminus C_\mr{NE}$ of $\mc{D}_\mr{NE}$.  
Then, 
$q_\mr{SE} = q_{1,\mr{SE}} = q_{0,\mr{SE}}$ 
is strictly southeast of one of the points 
$\{q_{0,\mr{NE}},q_{0,\mr{SW}}\} = L_0 \cap \partial C$
by Claim \ref{claimLSPqNE}, 
so $h_+(q_\mr{SE})$ is outside the semicircle $K_0$.
Also, $q_1$ appears between $q_0$ and $q_\mr{SE}$ on $L$ 
since $q_1$ is strictly southeast of $q_0$, 
so $h_+(q_1)$ appears between between $h_+(q_0)$ and $h_+(q_\mr{SE})$ on $h_+(L)$, 
so $h_+(q_1)$ is outside the semicircle $K_0$, 
so $K_1$ is outside the semicircle $K_0$,
so $h_+\hat g_\mr{NE}(q_1) < h_+\hat g_\mr{NE}(q_0) \leq 0$, 
so $g_\mr{NE}(q_1)$ is strictly southeast of $g_\mr{NE}(q_0)$ as above. 

Paths of $P$ and $B$ are either along the southwest boarder of a cell or on $B_\mr{NE}$, so the above cases cover the cases where $q_1$ and $q_0$ are in the same region of $\mc{D}_\mr{NE}$. 
In general we have that $g_\mr{NE}(q_1)$ is strictly southeast of $g_\mr{NE}(q_0)$ by induction on the number of regions of $\mc{D}_\mr{NE}$ along the southeastward path from $q_0$ to $q_1$. 
\end{proof}

To prove Claim \ref{claimLSPqNE}, 
we first prove the analogous numerical claim, 
namely Claim \ref{claimLSPInequalities}. 
Let 
\[
x_k = h_+(q_k), \quad
y_k = h_-(q_k)
\]
for each intercardinal direction $k$. 

\begin{claim} 
\label{claimLSPInequalities}
If $x_\mathrm{SE} \neq 0$ and $x_\mathrm{NW} \neq \infty$, then one of the following four strings of inequalities holds, 
\begin{align*}
& x_\mathrm{NE} < x_\mathrm{NW} \leq x_\mathrm{SW} \leq x_\mathrm{SE} > 0, \\
& x_\mathrm{NE} \leq x_\mathrm{NW} \leq x_\mathrm{SW} < x_\mathrm{SE} > 0, \\
0 > x_\mathrm{SE} <\ & x_\mathrm{NE} \leq x_\mathrm{NW} \leq x_\mathrm{SW}, \\
0 > x_\mathrm{SE} \leq\ & x_\mathrm{NE} \leq x_\mathrm{NW} < x_\mathrm{SW}. 
\end{align*}
If $x_\mathrm{SE} = 0$, then $x_\mathrm{NE} = x_\mathrm{NW} = x_\mathrm{SW} = 0$.
If $x_\mathrm{NW} = \infty$, then $x_\mathrm{SW} = x_\mathrm{SE} = x_\mathrm{NE} = \infty$.
\end{claim}

Note that the values $x_k$ are in the one point compactification of the real line, $\overline{\mb{R}} = \mb{R} \cup \infty$. 
Also, $(\leq)$ is only a total order on $\mb{R}$.  
We use the convention that if $x \neq \infty$, then both $x < \infty$ and $\infty < x$ hold, and we write the latter inequality as $-\infty < x$, although the meaning is the same.

\begin{proof}
We start by dealing with several special cases. 
We first deal with special cases where $x_\mathrm{SE} = 0$, and then where $x_\mathrm{NW} = \infty$, and then where $q$ is one of the points $p_\mathrm{NW},p_\mathrm{NE},p_\mathrm{SE}$, or $p_\mathrm{SW}$.
We then deal with the cases where $x_\mathrm{SE} = \infty$, 
$x_\mathrm{SE} > 0$, and $x_\mathrm{SE} < 0$.

The cases where $x_\mathrm{SE} > 0$ and $x_\mathrm{SE} < 0$ 
will use calculations involving $h_+$ and $h_-$. 
Recall that the maps $h_k$ were only defined up to positive scaling.  
We would like to choose these scaling factors
so that $h_+(p_\mr{NE}) = -1$ and $h_-(p_\mr{SE}) = 1$, 
but to do so we will need to deal with another special case 
where $p_\mr{NE} = p_\mr{SE}$ to ensure that 
$h_+(p_\mr{NE})$ and $h_-(p_\mr{SE})$ are finite.

Let $h = h_+h_-^{-1}$, 
and let 
\[\Theta_+ = h_+(\mc{L}_+(C,q)),\quad \Theta_- = h_-(\mc{L}_-(C,q)).\]   
Observe that $\Theta_+$, $\Theta_-$, and $h(\Theta_-)$ are the possibly degenerate upper semicircles centered on the real line 
where $\Theta_+$ has endpoints $x_\mr{SW}$ and $x_\mr{NE}$, 
and $\Theta_-$ has endpoints $y_\mr{NW}$ and $y_\mr{SE}$, 
and $h(\Theta_-)$ has endpoints $x_\mr{NW}$ and $x_\mr{SE}$. 
Also, 
\[
x_\mr{NE} = -2x_\mr{SW}, \quad 
y_\mr{NW} = \tfrac{-1}2 y_\mr{SE}
\]
by Definition \ref{defLSP}.

{Consider the case where $x_\mathrm{SE} = 0$.}
Then, $q_\mathrm{SE} = h_+^{-1}(0) = p_\mathrm{NW} \in C_\mathrm{NW}$, 
so $q_\mathrm{SE} \in C_\mathrm{NW} \cap C_\mathrm{SE} = \{p_\mathrm{NE},p_\mathrm{SW}\}$, 
so $q_\mathrm{SE} = p_\mathrm{NW} = p_\mathrm{SW}$ by Remark \ref{remarkLSPEWDisjoint}, 
and so $\Theta_-$ has endpoints $y_\mr{SE} = h_-(q_\mathrm{SE}) = h_-(q_\mathrm{SW}) =0$ and $y_\mr{NW} = \frac{-1}{2}y_\mr{SE}=0$, 
so $\Theta_- = 0$, 
so $q = p_\mathrm{SW} = p_\mathrm{NW}$, 
so $q_\mr{NE} = q_\mr{NW} = q_\mr{SW} = p_\mr{NW}$, 
so $x_\mathrm{NE} = x_\mathrm{NW} = x_\mathrm{SW} = 0$.

{Consider the case where $x_\mathrm{NW} = \infty$.}
Then, $q_\mathrm{NW} = h_+^{-1}(\infty) = p_\mathrm{SE} \in C_\mathrm{SE}$, 
so $q_\mathrm{NW} \in C_\mathrm{NW} \cap C_\mathrm{SE} = \{p_\mathrm{NE},p_\mathrm{SW}\}$, 
so $q_\mr{NW} = p_\mr{SE} = p_\mr{NE}$ by Remark \ref{remarkLSPEWDisjoint}, 
so $y_\mr{NW} = h_-(q_\mr{NW}) = h_-(p_\mr{NE}) = \infty$, 
and so $\Theta_-$ is the possibly degenerate semicircle with endpoints 
$y_\mr{NW}=\infty$ and $y_\mr{SE}=-2 y_\mr{NW}=\infty$, 
so $\Theta_-=\infty$, 
and $h_-(q) \in \Theta_-$, 
so $q = p_\mr{NE} = p_\mr{SE}$, 
so $q_\mr{NE} = q_\mr{SE} = q_\mr{SW} = p_\mr{SE}$, 
so $x_\mathrm{SW} = x_\mathrm{SE} = x_\mathrm{NE} = \infty$.

{Next, we deal with the cases where 
$q \in \{p_\mathrm{NE},p_\mathrm{NW},p_\mathrm{SW},p_\mathrm{SE}\}$}
and $x_\mr{SE} \neq 0$, $x_\mr{NW}\neq\infty$. 

If $q = p_\mathrm{SE}$, 
then $q_\mr{NE} = q_\mr{SE} = q_\mr{SW} = p_\mr{SE}$, 
so $-\infty = x_\mathrm{NE} < x_\mathrm{NW} < x_\mathrm{SW} = x_\mathrm{SE} = \infty$. 

If $q = p_\mathrm{NW}$, 
then $q_\mr{NE} = q_\mr{NW} = q_\mr{SW} = p_\mr{NW}$, 
so $x_\mathrm{NE} = x_\mathrm{NW} = x_\mathrm{SW} = 0$. 

If $q = p_\mathrm{NE}$, 
then $q_\mr{NW} = q_\mr{NE} = q_\mr{SE}$, 
so $x_\mathrm{SE} = x_\mathrm{NE} = x_\mathrm{NW} < 0 \leq x_\mathrm{SW}$. 

If $q = p_\mathrm{SW}$, 
then $q_\mr{NW} = q_\mr{SW} = q_\mr{SE}$, 
so $x_\mathrm{NE} \leq 0 < x_\mathrm{NW} = x_\mathrm{SW} = x_\mathrm{SE}$.

For the rest of the proof let us assume that 
$q \not\in \{p_\mathrm{NE},p_\mathrm{NW},p_\mathrm{SW},p_\mathrm{SE}\}$
and $x_\mr{SE} \neq 0$, $x_\mr{NW}\neq\infty$. 
This means that $\Theta_+$ and $h(\Theta_-)$ are nondegenerate semicircles that 
pass through $h_+(q)$, 
and therefore 
$\infty\not\in\{x_\mr{NE},x_\mr{SW}\}$ 
and the endpoints of these semicircles alternate on $\overline{\mb{R}}$. 
Hence, 
$x_\mr{SE}\neq x_\mr{NW}$ and $x_\mr{SW}\neq x_\mr{NE}$, and 
\[
\begin{array}{lccrl}
\text{either} &
x_\mr{NE} \leq x_\mr{SE} \leq x_\mr{SW} &
\text{and} &
x_\mr{NW} \not\in (x_\mr{NE},x_\mr{SW})_\mb{R} 
& \text{(clockwise)} , \\ 
\text{or} & 
x_\mr{NE} \leq x_\mr{NW} \leq x_\mr{SW} &
\text{and} &
x_\mr{SE} \not\in (x_\mr{NE},x_\mr{SW})_\mb{R}
& \text{(counter-clockwise)}.
\end{array}
\]
To distinguish these, 
in the first case we say the $x_k$ increase clockwise, 
and in the second case say the $x_k$ increase counter-clockwise.  
This is consistent with the convention that southeast is clockwise of northeast. 
To complete the proof, we show that the $x_k$ increase counter-clockwise. 

If $x_\mr{SE} = \infty$, 
then $x_\mr{SE} \not\in [x_\mr{NE},x_\mr{SW}]_\mb{R}$, 
so the $x_k$ increase counter-clockwise.

For the sake of contradiction, let us assume that the $x_k$ are finite and increase clockwise.

Consider the case where $p_\mathrm{NE}=p_\mathrm{SE}$.
Then,   
\[h_+^{-1}(\infty) = p_\mr{SE} = p_\mr{NE} = h_-^{-1}(\infty),\] 
so $h$ is a linear fractional transformation that preserves the real line and preserves the point at $\infty$, and since $h_+$ and $h_-$ are only defined up to positive scalar multiple, which we are free to choose, we may choose these so that $h$ is a translation. 
Also, $h(0) = h_+\circ h_-^{-1}(0) = h_+(p_\mr{SW}) \geq 0$, since $p_\mr{SW} \in C_\mr{SW}$, 
so $h(z) = z +h_+(p_\mr{SW})$, 
and $h_+(p_\mr{SW}) < \infty$, since $p_\mr{SW} \neq p_\mr{NE}=p_\mr{SE}$. 
Hence, 
\[ 
x_\mr{NW} 
= h(y_\mr{NW})
= h(\tfrac{-1}{2} y_\mr{SE})
= h(\tfrac{-1}{2} h^{-1}(x_\mr{SE}))
= \tfrac{-1}{2}x_\mr{SE} +\tfrac{3}{2}h_+(p_\mr{SW})
\geq \tfrac{-1}{2}x_\mr{SE}. 
\] 
Also, 
$C_\mathrm{E} = p_\mathrm{SE}$ is a single point, so $q_\mr{SE} \in C_\mr{SE} = C_\mr{S}$, so 
$x_\mr{SE} \geq h_+(p_\mr{SW}) \geq 0$, 
and we have already eliminated the case where $x_\mr{SE}=0$, 
so $x_\mr{SE} > 0$.  
Since we assume the $x_k$ increase clockwise, we have 
$x_\mr{NE} \leq x_\mr{SE} \leq x_\mr{SW}$ and 
$x_\mr{NW} \not\in (x_\mr{NE},x_\mr{SW})_\mb{R}$,
so 
\[x_\mr{NW} \geq \tfrac{-1}{2}x_\mr{SE} \geq \tfrac{-1}{2}x_\mr{SW} > -2x_\mr{SW} = x_\mr{NE},\]
so $x_\mr{NW} > x_\mr{NE}$, which implies that $x_\mr{NW} \geq x_\mr{SW}$, so 
\[x_\mr{NW} \geq x_\mr{SW} \geq x_\mr{SE} \geq h_+(p_\mr{SW}) \geq x_\mr{NW},\]
but that contradicts $x_\mr{NW} \neq x_\mr{SE}$.

For the rest of the proof let us assume that $p_\mathrm{NE}\neq p_\mathrm{SE}$. 
Then, $h_+(p_\mathrm{NE}) \in (-\infty,0)$ 
and $h_-(p_\mathrm{SE}) \in (0,\infty)$, 
and $h$ preserves cross-ratios, so 
\[
\frac{h_-(p_\mr{NW})}{h_-(p_\mr{SE})}
= \crr(0,\infty;h_-(p_\mr{NW}),h_-(p_\mr{SE}))
= \crr(h_+(p_\mr{SW}),h_+(p_\mr{NE});0,\infty)
= \frac{h_+(p_\mr{SW})}{h_+(p_\mr{NE})}. 
\]
Since $h_+$ and $h_-$ 
were only defined up to positive scalar, let us choose these scalars and $r>0$ so that  
\[
h_+(p_\mathrm{NE}) = -1, \quad h_+(p_\mathrm{SW}) = r, \quad h_-(p_\mathrm{SE}) = 1, \quad h_-(p_\mathrm{NW}) = -r.
\]
Then, $h$ is the linear fractional transformation such that 
$h(1)=h_+(p_\mathrm{SE})=\infty$, $h(-r)=h_+(p_\mathrm{NW})=0$, and $h(\infty)=h_+(p_\mathrm{NE})=-1$, so  
\[
h(z) = \frac{-z-r}{z-1}, \quad
h^{-1}(z) = \frac{z-r}{z+1}.
\]
Since $q_\mr{NW} \in C_\mr{NW}$ and  
$q_\mr{SE} \in C_\mr{SE}$, we have
\[
-1 \leq x_\mr{NW} \leq  r 
\quad \text{and} \quad
x_\mr{SE} \not\in (-1,r). 
\]

Consider the case where $x_\mathrm{SE} > 0$.
Then, 
$q_\mr{SE} \in C_\mr{S}$, so 
$x_\mr{SE} \geq r$.
Also,  
\[
h\circ\tfrac{-1}{2}h^{-1}(z) 
=\frac{(-1+2r)z+3r}{3z-r+2}
\]
is decreasing 
except at $\frac{r-2}{3}$
since 
$
\tfrac{\mr{d}}{\mr{d}z}[h\circ\tfrac{-1}{2}h^{-1}](z) 
= \frac{-2(r^2+2r+1)}{(3z-r+2)^2} < 0
$,
and 
$
\tfrac{r-2}{3}< r \leq x_\mr{SE} \leq x_\mr{SW}
$
by our assumption that the $x_k$ increase clockwise, so 
\begin{gather*}
3x_\mathrm{SW} -r+2 > 0, \\
 ((-1+2r)x_\mathrm{SW} +3r) + 2x_\mathrm{SW} (3x_\mathrm{SW} -r+2) 
= 6x_\mathrm{SW}^2 +3x_\mathrm{SW} +3r > 0, \\
x_\mr{NW}
= h\circ\tfrac{-1}{2}h^{-1}(x_\mr{SE}) 
\geq h\circ\tfrac{-1}{2}h^{-1}(x_\mr{SW}) 
= \frac{-(1+2r)x_\mr{SW}+3r}{3x_\mr{SW}-r+2}
>-2x_\mr{SW} 
= x_\mr{NE}. 
\end{gather*}
The inequality $x_\mr{NW}>x_\mr{NE}$ 
together with our assumption that the $x_k$ increase clockwise 
implies that $x_\mr{NW}\geq x_\mr{SW}$, 
so $r \leq x_\mr{SE} \leq x_\mr{SW} \leq x_\mr{NW} \leq r$, 
but that contradicts $x_\mr{NW} \neq x_\mr{SE}$.

Consider the case where $x_\mathrm{SE} < 0$.
Then, 
$q_\mr{SE} \in C_\mr{E} \setminus\{p_\mr{NE},p_\mr{SE}\}$, so 
$x_\mr{SE} \leq -1$,
and 
$
\tfrac{r-2}{3}> -1 \geq x_\mr{SE} \geq x_\mr{NE}
$
by our assumption that the $x_k$ increase clockwise, so 
\begin{gather*} 
2 ((-1+2r)x_\mathrm{NE} +3r) + x_\mathrm{NE} (3x_\mathrm{NE} -r+2) 
= 3x_\mr{NE}^2 +3rx_\mr{NE} +6r > 0, \\
x_\mathrm{NW} = h\circ \tfrac{-1}{2}h^{-1}(x_\mathrm{SE}) \leq h_3\circ \tfrac{-1}{2}h^{-1}(x_\mathrm{NE}) = \frac{(-1+2r)x_\mathrm{NE} +3r}{3x_\mathrm{NE} -r+2}  < \tfrac{-1}{2} x_\mathrm{NE} = x_\mathrm{SW}. 
\end{gather*}
The inequality $x_\mr{NW}<x_\mr{SW}$ 
together with our assumption that the $x_k$ increase clockwise 
implies that $x_\mr{NW}\leq x_\mr{NE}$, 
so $-1 \leq x_\mr{NW} \leq x_\mr{NE} \leq x_\mr{SE} \leq -1$, 
but that contradicts $x_\mr{NW} \neq x_\mr{SE}$.   

Thus, the $x_k$ increase counter-clockwise. 
\end{proof}

\begin{proof}[Proof of Claim \ref{claimLSPqNE}]
Let us use the same notation as in the proof of Claim \ref{claimLSPInequalities}. 
$\Theta_+$ and $h(\Theta_-)$ are either distinct upper semicircles centered on the real line 
or a single point, 
so $\Theta_+$ and $h(\Theta_-)$ cannot intersect at multiple points, 
so $\Theta_+\cap h(\Theta_-) = h_+(q)$, 
so $L_+(C,q)\cap L_-(C,q) = q$. 

For the second part of the proof, 
we split into cases where $q_\mr{NE}$ is in $C_\mr{N}$ or $C_\mr{E}$.

Consider the case where $q_\mr{NE} \in C_\mr{N}$.
Then, $x_\mr{NE} \leq x_\mr{NW}$ by Claim \ref{claimLSPInequalities}. 
If $x_\mr{NE} = x_\mr{NW}$, then $\Theta_+ \cap h(\Theta_-) = x_\mr{NE}$, since $x_\mr{NE}$ is an endpoint of $\Theta_+$ and $x_\mr{NW}$ is an endpoint of $h(\Theta_-)$, 
so $L_+(C,q) \cap L_-(C,q) = q_\mr{NE} = q_\mr{NW} = q$, 
so $q \in C_\mr{NE}$ and $q_\mr{NE}$ is northeast of $q_\mr{NW}$, so the claim holds. 
If $x_\mr{NE} < x_\mr{NW}$, then $h_+(q_\mr{NE}) < h_+(q_\mr{NW})$ and $h_+$ is decreasing northeastward on $C_\mr{N}$, so $q_\mr{NE}$ is strictly northeast of $q_\mr{NW}$. 
Thus, the claim holds in the case where $q_\mr{NE} \in C_\mr{N}$.

Consider the case where $q_\mr{NE} \in C_\mr{E}$.
If $q_\mr{SE} \in C_\mr{S}$, then $q_\mr{NE}$ is strictly northeast of $q_\mr{SE}$ and the claim holds, so let us assume that $q_\mr{SE} \not\in C_\mr{S}$. 
Then, $q_\mr{SE} \in C_\mr{E} \setminus p_\mr{SE}$, so $x_\mr{SE} < 0$, so $x_\mr{SE} \leq x_\mr{NE}$ by Claim \ref{claimLSPInequalities}. 
If $x_\mr{SE} = x_\mr{NE}$, then $q = q_\mr{NE} = q_\mr{SE} \in C_\mr{NE}$ and $q_\mr{NE}$ is northeast of $q_\mr{SE}$ as above, so the claim holds.
If $x_\mr{SE} < x_\mr{NE}$, then $h_+(q_\mr{SE}) < h_+(q_\mr{NE})$ and $h_+$ is increasing northeastward on $C_\mr{E}$, so $q_\mr{NE}$ is strictly northeast of $q_\mr{SE}$. 
Thus, the claim holds in the case where $q_\mr{NE} \in C_\mr{E}$.

Analogous statements for the other intercardinal directions follow similarly. 
\end{proof}


\begin{proof}[Proof of Lemma \ref{lemmaLSP} parts \ref{itemLSPProduct} and \ref{itemLSPEastward}]
Suppose there were $L_+ \in \mc{L}_+$ and $L_- \in \mc{L}_-$ 
with distinct points $q_0,q_1$ on $L_+\cap L_-$.  
Then, $g_\mr{NE}(q_0) = L_+\cap B_\mr{NE} = g_\mr{NE}(q_1)$, 
but one of the points must be southeast of the other on $L_-$, 
which would contradict Claim \ref{claimLSPMonotone}. 
Similarly, if there were distinct points $q_0,q_1$ on $L_+\cap P_i$ for $P_i\in P$, 
then we would have a contradiction in the same way. 
Suppose $q=L_+\cap P_i$ is not an endpoint of $L_+$.
Then $q \in C_\mr{SW} {\setminus} \{p_\mr{NW},p_\mr{SE}\}$ 
and $P_i$ directed eastward traverses $C_\mr{SW}$ directed from $p_\mr{NW}$ to $p_\mr{SE}$, 
which is from the northwest of $L_+$ to the southeast. 
A similar argument shows that we cannot have distinct points on $L_-\cap P_i$, and $P_i$ directed eastward crosses from southwest to northeast.
Thus, parts \ref{itemLSPProduct} and \ref{itemLSPEastward} hold. 
\end{proof}

\begin{proof}[Proof of Lemma \ref{lemmaLSP} part \ref{itemLSPFinite}]
Consider a southward directed zigzag path $Z$ that alternates between leaves of $\mc{L}_+$ and $\mc{L}_-$ as in part \ref{itemLSPFinite}.
Since the path $Z$ is directed southward, 
$I_N(q)$ for $q\in Z$ shrinks each time $Z$ crosses a path of $P$
by Remark \ref{remarkLSPIS}, 
so the path cannot return to a cell more than once. 
Hence, there must be a cell $C\in\mc{C}$ where $Z$ traverses infinitely many leaves of $\mc{L}_+$ and $\mc{L}_-$.
Since $B_\mr{W}$ and $B_\mr{E}$ are disjoint, 
the southern boarder $C_\mr{S}$ cannot be a single point, 
so $p_\mr{SW} \neq p_\mr{SE}$, 
and $p_\mr{SW} \in C_\mr{SW}$, 
so $0\leq h_+(p_\mr{SW}) < h_+(p_\mr{SE}) = \infty$. 

Let $q_0$ be the southwest most point of an edge $L_+(q_0) \in \mc{L}_+$ of $Z$. 
We claim that there are only finitely many edge of $Z$ in $C$ south of $q_0$. 
If $q_0$ is not an endpoint of $L_+(q_0)$, then $L_+(q_0)$ continues into another cell south of $C$, so $L_+(q_0)$ is the last edge of $Z$ through $C$ and the claim holds. 
Let us suppose $q_0$ is an endpoint of $L_+(q_0)$. 
If $q_0 \in C_\mr{S}$, then $q_0 \in B_\mr{S}$ is the southeast endpoint of the edge $L_-(q_0)$, 
so the leaf $L_-(q_0)$ extends northward from $q_0$, 
so $Z$ cannot continue along a leaf of $\mc{L}_-$, 
which again means that $L_+(q_0)$ is the last edge. 

Let us suppose $q_0 \not\in C_\mr{S}$, and let $q_1,q_2,\dots$ be the sequence of vertices of $Z$ continuing south from $q_0$.  
Then, $q_1 \in C_\mr{SE}$ since $q_1$ is the southeast endpoint of $L_-(q_0)$, 
and if $q_1 \in C_\mr{S}$ then $q_1$ is the southwest endpoint of $L_+(q_1)$ and $Z$ cannot continue southward along a leaf of $\mc{L}_+$ in $C$.

Let us suppose that $q_1 \not\in C_\mr{S}$. 
Then, $q_1 = \hat g_\mr{SE}(q_0) \in C_\mr{E}$, so $q_1$ is not comparable to $\hat g_\mr{SW}(q_0) \in C_\mr{SW}$ on $\partial C$ directed southeastward, 
so $q_1$ is strictly southeast of $\hat g_\mr{NE}(q_0)$ by Claim \ref{claimLSPqNE},
so
\[
h_+(q_2) = h_+(\hat g_\mr{SW}(q_1)) 
= \tfrac{-1}2 h_+(q_1) 
< \tfrac{-1}2 h_+(\hat g_\mr{NE}(q_0)) 
= h_+(q_0),
\]
which means that $q_2$ is strictly southeast of $q_0$. 

Let us suppose that $q_1,q_2,q_3,\dots \not\in C_\mr{S}$. 
Then, $q_2,q_4,q_6,\dots \in C_\mr{W}$ is monotonic on $C_\mr{W}$ by the argument above, and since $C_\mr{W}$ is compact, the sequence converges to a point $q_{2\infty}\in C_\mr{W}$. 
Also, $q_{2i+2} = \hat g_\mr{SW}\hat g_\mr{SE}(q_{2i})$, and the restriction of $\hat g_\mr{SW}\hat g_\mr{SE}$ to $S_\mr{W}$ is a Möbius transformation, and in particular is continuous, 
so $q_{2\infty} \in C_\mr{SW}$ is a fixed-point of $\hat g_\mr{SW}\hat g_\mr{SE}$,  
and also $q_1,q_3,q_5,\dots$ converges to a fixed-point $q_{2\infty+1} = \hat g_\mr{SE}(q_{2\infty})$ of $\hat g_\mr{SE}\hat g_\mr{SW}$ on $C_\mr{E}$, 
but by the same argument as above, $\hat g_\mr{SW}\hat g_\mr{SE}(q_{2\infty})$ is strictly southeast of $q_{2\infty}$, which is a contradiction. 
Thus, there is no southward zigzag path with infinitely many edges, 
which means part \ref{itemLSPFinite} holds.
\end{proof}

\begin{proof}[Proof of Lemma \ref{lemmaLSP} part \ref{itemLSPAntipodal}] 
Let $(\widetilde{\mc{L}}_+,\widetilde{\mc{L}}_-) = \lsp(R,\widetilde B,P)$ where $B_\mr{N}$ and $B_\mr{S}$ are swapped in $\widetilde B$ and denote the corresponding objects in the construction of $\lsp(R,\widetilde B,P)$ in Definition \ref{defLSP} with the diacritic. 
Then, we have the same cells, but $\widetilde C_\mr{N} = C_\mr{S}$ and $\widetilde C_\mr{S} = C_\mr{N}$, 
so $-h_-(\widetilde{p}_\mr{NW}) = -h_-(p_\mr{SW}) = 0$,  
and $-h_-(\widetilde{p}_\mr{SE}) = -h_-(p_\mr{NE}) = \infty$
and $-h_-(\widetilde{C}_\mr{NE}) = -h_-(C_\mr{SE}) = [-\infty,0]$, 
and $-h_-$ is internally isogonal, 
so $\widetilde{h}_+ = -h_-$ with appropriate choice of positive scalar as in Remark \ref{remarkLSPhScalar}, 
so $\widetilde{h}_+^{-1}(rK_+) = h_-^{-1}(-rK_+) = h_-^{-1}(rK_-)$, 
so $\widetilde{\mc{L}}_+ = \mc{L}_-$, 
and similarly $\widetilde{\mc{L}}_- = \mc{L}_+$. 
Thus, part \ref{itemLSPAntipodal} holds.   
\end{proof}

%% file: zone-v5-9.tex
\section{Zoning}
\label{sectionZone}

The construction of $\ebb_\mr{D}$ from Lemma \ref{lemmaEbbD} in Definition \ref{defEbbD}, 
will involve subdividing the sphere into a face decomposition, called a zoning, and defining the deformation separately in each face. 
While there are many ways that this could potentially be done, the decomposition that we use is the result of a 3-way compromise between 
simplicity of the definition, ease of construction, and applicability to constructing an appropriate deformation.

This decomposition will include a subregion within each facet covector region that pseudocircles will be forbidden from moving into so as to prevent the inradius of a facet covector region from shrinking too much, which will ensure that part \ref{itemEbbDMinbig} of Lemma \ref{lemmaEbbD} holds. 
This will also prevent the inradius of the vanishing region from growing too much, 
which will ensure that part \ref{itemEbbDMaxlit} holds.

\begin{definition}[Zoning]
\label{defZoning}
A \df{generic $\mc{M}$-zoning} is a face lattice decomposition $Z$ of the 2-sphere with a 2-cell $Z(\Sigma)$ for each chain of $\csph(\mc{M})$ that intersect as in Figure \ref{figureZoneEg}.  
Specifically, cells $Z(\Sigma_1)$ and $Z(\Sigma_2)$ share a common edge when 
$\Sigma_1 \subset \Sigma_2$ and $|\Sigma_2| = |\Sigma_1|+1$ or when 
$\Sigma_2 = \{\upsilon,\sigma,\tau\}$ and $\Sigma_1 = \{\tau\}$
for a maximal chain $\upsilon < \sigma < \tau$. 

\begin{figure}[h]
\centering
\includegraphics[scale=1]{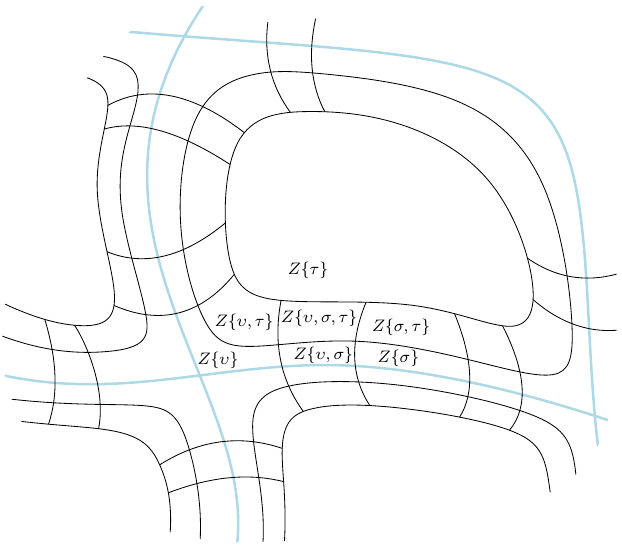}
\caption{A portion of an $\mc{M}$-zoning in black with an accomodated arrangement in light blue.  Zones $Z(\Sigma)$ for $\Sigma \subset \{\upsilon,\sigma,\tau\}$ are labeled.}
\label{figureZoneEg}
\end{figure}

Note that a generic $\mc{M}$-zoning $Z$ has a $2k$-gonal cell $Z\{\upsilon\}$ for each vertex covector 
$\upsilon$ of degree $k$, 
and has a $4k$-gonal cell $Z\{\tau\}$ for each facet covector 
$\tau$ of degree $k$, 
and has a square cell for each edge covector. 
$Z$ also has a square cell $Z\{\sigma,\tau\}$ for comparable pair $\sigma<\tau$ of covectors, which is adjacent to the cells $Z\{\sigma\}$ and $Z\{\tau\}$, 
and $Z$ has a square cell $Z(\Sigma)$ for each maximal chain $\Sigma$, which is adjacent to each cell corresponding to a pair of covectors in $\Sigma$ and the the cell corresponding to the facet covector; see Figure \ref{figureZoneEg}. 

We generally just define the 2-cells of $Z$, which we call \df{zones}, since this determines all lower dimensional cells.
We may write $Z(\sigma)$ for $Z\{\sigma\}$, 
and we  write $Z(\Sigma_1,\Sigma_2)$ for the face $Z(\Sigma_1)\cap Z(\Sigma_2)$. 

For the sake of some compactness arguments later, we define a \df{degenerate} $\mc{M}$-zoning $Z$ 
by  
\[Z(\Sigma) = \bigcap \{\cl(\maxcov^{-1}(\mc{M},A;\sigma)):\sigma\in\Sigma\}.\] 
for some arrangement $A \in \upa(\mc{M})$. 
We say $Z$ is \df{fully degenerate} when $A\in\psv(\mc{M})$, and we may identify $Z$ with $A$ when convenient. 
A $\mc{M}$-zoning could be generic or degenerate. 

We measure \df{distance} between $\mc{M}$-zonings by the maximum Fréchet distance between corresponding edges of the face decomposition. 
That is, 
\[
\dist(Z_1,Z_0) = \max_{\Sigma,\Sigma'}\dist_\mr{F}(Z_1(\Sigma,\Sigma'), Z_0(\Sigma,\Sigma'))
\]
where $\Sigma,\Sigma'$ are an adjacent pair of chains 
in the face lattice of a generic $\mc{M}$-zoning 
with $Z(\Sigma,\Sigma')$ directed according to some fixed ordering on covectors. 
Note that this is well-defined for degenerate zonings since $Z(\Sigma,\Sigma')$ is a path or point by Lemma \ref{lemmaMaxcovBorder} part \ref{itemMaxcovBorderPath}.

Let $Z(\Sigma)$ be an \df{inner $(i,+)$-zone} 
when $i\in\sigma^+$ for each $\sigma\in\Sigma$, 
and let
$Z(\Sigma)$ be a \df{wider $(i,+)$-zone} 
when $i\in\sigma^+$ for some $\sigma\in\Sigma$. 
Let $\inner(Z,i,+)$ be the union of the inner $(i,+)$-zone 
and $\wider(Z,i,+)$ be the union of the wider $(i,+)$-zone. 
That is, 
\begin{align*}
\inner(Z,i,+)
&=\bigcup\left\{{\textstyle Z(\Sigma) : i\in\bigcap\Sigma^+ }\right\}, \\
\wider(Z,i,+)
&=\bigcup\left\{{\textstyle Z(\Sigma) : i\in\bigcup\Sigma^+ }\right\}, 
\end{align*}
where 
$\Sigma^+ = \{\sigma^+:\sigma\in\Sigma\}$.
We define $\Sigma^-$ and $\Sigma^0$ 
and inner and wider $(i,-)$-zones and $(i,0)$-zones analogously. 
\end{definition}

\begin{remark}
\label{remarkInnerWiderSubdivision}
\begin{gather*}
\left\{
\inner(Z,i,0), \wider(Z,i,+), \wider(Z,i,-)
\right\}, \\ 
\left\{
\wider(Z,i,0), \inner(Z,i,+), \inner(Z,i,-)
\right\}\phantom{,} 
\end{gather*}
are subdivisions of $\sphere^2$.
\end{remark}

\begin{definition}[Accommodating an arrangement]
\label{defAccommodate}
We say a generic $\mc{M}$-zoning ${Z}$ \df{accommodates} an arrangement $A$
when for each nonloop $i$ of $\mc{M}$ in the ground set of $A$, 
\[
\inner(Z,i,0)^\circ \supset S_i, \quad 
\wider(Z,i,+) \subset S_i^+, \quad 
\wider(Z,i,-) \subset S_i^-,
\]
and $S_i\cap Z(\Sigma)$ is a path with endpoints $S_i\cap \partial Z(\Sigma)$ 
for each inner $(i,0)$-zone $Z(\Sigma)$. 
Also, the fully degenerate zoning defined by $A$ accommodates $A$ as well as the restriction to a smaller ground set. 



Let $\facet(Z)$ denote the set of zones $Z(\sigma)$ for each facet covector $\sigma$ of $\csph(\mc{M})$. 
We say $F = \facet(Z)$ \df{accommodates} $A$ when the above holds for each zone $F(\sigma)$. 
That is, $F(\sigma) \subset S_i^\sigma$ for each $i\in\supp(\sigma)$.

We say that $Z$ accommodates $A$ \df{$\eps$-tightly} when 
each facet covector $\sigma$ has the boundary of its zone $\partial Z(\sigma)$ 
within Fréchet distance $\eps$ of $\partial\maxcov^{-1}(\sigma)$, 
and 
$\wider(Z,i,0) \subset S_i\oplus\eps$ 
for each nonloop $i$.
We also say the fully degenerate zoning defined by $A$ accommodates $A$ 0-tightly. 
Let $\tight(Z,A)$ be the infimum of $\eps$ such that $Z$ accommodates $A$ $\eps$-tightly. 
\end{definition}

\begin{definition}[Preference]
\label{defPreference}
Given subsets $N_1\subset\dots\subset N_L\subseteq[n]_\mb{N}$, 
let 
\[(\lesssim) = \pref(N_1,\dots,N_L)\]  
be the total preorder on $[n]_\mb{N}$ where 
$j \lesssim i$ when each of the sets $N_k$ that contain $i$ also contain $j$. 
That is, the $N_k$ are upper sets of $(\lesssim)$.
We may also call $(\lesssim)$ a \df{preference}. 
We write $j \lnsim i$ when $j \lesssim i$, but the reverse does not hold, 
or equivariantly, stricly more of the sets $N_k$ contain $i$ than contain $j$.

The \df{greedy choice} $i_\mr{g}$ from $I\subseteq [n]$ is the minimum element in the natural order on $[n]$ among the maximal elements of $I$ with respect to $(\lesssim)$.  That is, 
\[
i_\mr{g} = \min\{i\in I : \forall j\in I,\, j\lesssim i\}. 
\]
We say $i$ is a \df{greedy element} of $(\mc{M},\lesssim)$ when $i$ is the greedy choice from its span.

We say $(\lesssim)$ \df{favors nonloops} of $\mc{M}$ 
when if $j$ is a loop of $\mc{M}$ but $i$ is not, then $j \lesssim i$.  
We say $(\lesssim)$ favors nonloops of a chain $\mc{C}$ 
when $(\lesssim)$ favors nonloops for each oriented matroid in $\mc{C}$.
A \df{preferred set} is a full rank upper set.

A \df{generic $(\mc{M},\lesssim)$-zoning} $Z$ consists of a $\mc{N}_k$-zoning $Z_k$ 
for each restriction $\mc{N}_k = \rest(N_k,\mc{M})$ 
to a preferred set $N_k$ satisfying the following.
If $N_j \subset N_k$ and $i\in N_j$, then we require 
\[
\inner(Z_j,i,+)^\circ \supset \wider(Z_k,i,+)
\quad \text{and} \quad 
\inner(Z_j,i,-)^\circ \supset \wider(Z_k,i,-). 
\]
Equivalently, 
if $Z_j(\Sigma_1) \cap Z_k(\Sigma_2) \neq \emptyset$ for $j<k$,
then 
$\sigma_j \geq_\mr{v} \rest(N_j,\sigma_k)$  
for each pair of $\sigma_j \in \Sigma_j$, $\sigma_k \in \Sigma_k$. 
We also define a \df{degenerate} $(\mc{M},\lesssim)$-zoning 
as a sequence of $\mc{N}_k$-zonings where one or more $Z_k$ is degenerate and  
\[
\inner(Z_j,i,+) \supseteq \wider(Z_k,i,+)
\quad \text{and} \quad 
\inner(Z_j,i,-) \supseteq \wider(Z_k,i,-) 
\]
for $i\in N_j \subset N_k$
and we say $Z$ is \df{fully degenerate} when each $Z_k$ is fully degenerate.

We say that $Z$
\df{accommodates} $A$ when each $Z_k$ accommodates $A$, 
and $\tight(Z,A)$ is the maximum among $\tight(Z_k,A)$.  
To measure distance between $(\mc{M},\lesssim)$-zonings, we take the maximum distance between 
the corresponding $\mc{N}_k$-zonings. 
\end{definition}

\subsection{Zone map}

\begin{lemma}[Zone map]
\label{lemmaZoneMap}
Given $A \in \upa(\mc{M})$ and $\eps\geq 0$, 
then  
$Z = \zone(\mc{M},\lesssim,\eps,A)$ 
is a $(\mc{M},\lesssim)$-zoning 
that $\eps$-tightly accommodates $A$ and depends continuously and $\orth_3$-equivariantly on $(\eps,A)$. 
Also, $Z$ is generic if $\eps>0$. 
\end{lemma}

\begin{definition}[Zone map]
\label{defZone}


To construct $\zone(\mc{M},\lesssim,\eps,A)$, 
we first construct a continuous family of $\mc{M}$-zonings $Z(A,x)$ for $x\in [\nicefrac23,1]_\mb{R}$ for which the tightness decreases to 0 as $x \to 1$.
We will construct regions  
\[Z_\omega(\Sigma) = Z(A,x,\Sigma)\cap\cl(\maxcov^{-1}(\omega))\] 
for each $\omega \in \csph(\mc{M})$ and chain $\Sigma \leq \omega$.
That is, $\sigma \leq \omega$ for each $\sigma \in \Sigma$.  

In the case where $\omega =\upsilon$ is a vertex covector of $\mc{M}$, 
let $Z_\upsilon(\upsilon) = \maxcov^{-1}(\upsilon)$.

Consider the case where $\omega=\sigma$ is an edge covector of $\csph(\mc{M})$. 
Since $\Cov(\mc{M})$ is thin, 
there are exactly two vertex covectors $\{\upsilon_1,\upsilon_2\}<\sigma$, 
and two facet covectors $\{\tau_1,\tau_2\}>\sigma$. 
Let $R = \cl(\maxcov^{-1}(\sigma))$, 
\begin{align*}
B_\mr{N} &= R\cap \cl(\maxcov^{-1}(\tau_1)), &
B_\mr{W} &= R\cap \maxcov^{-1}(\upsilon_1), \\ 
B_\mr{S} &= R\cap \cl(\maxcov^{-1}(\tau_2)), &  
B_\mr{E} &= R\cap \maxcov^{-1}(\upsilon_2).  
\end{align*}
Let $P$ be the restriction of pseudocircles of $A$ that vanish on $\sigma$ to $R$. 
Let $i_0$ be the greedy choice from $\sigma^0$,
and $h(i_0)=0$.  
Let 
$\xi = \mox(R,B,P,h)$, 
and let $\xi_1$ be the projection of $\xi$ to the first coordinate. 
Let 
\begin{align*}
Z_\sigma(\upsilon_1) &= \xi_1^{-1}[0,\tfrac{1-x}2], \\
Z_\sigma(\{\upsilon_1,\sigma\}) &= \xi_1^{-1}[\tfrac{1-x}2,1-x], \\
Z_\sigma(\sigma) &= \xi_1^{-1}[1-x,x], \\
Z_\sigma(\{\upsilon_2,\sigma\}) &= \xi_1^{-1}[x,\tfrac{x+1}2], \\
Z_\sigma(\upsilon_2) &= \xi_1^{-1}[\tfrac{x+1}2,1].
\end{align*}

Consider the case where $\omega = \tau$ is a facet covector of $\csph(\mc{M})$.
Let $(\Sigma_0,\Sigma_1,\Sigma_2)$ be the greedily chosen triple of extensions of $\tau$ to maximal chains of $\csph(\mc{M})$ that only have $\tau$ in common, 
and let 
\[
p_k = \bigcap_{\sigma\in\Sigma_k} \cl(\maxcov^{-1}(\sigma))
\]
as in Lemma \ref{lemmaFlagPoint}, 
and let 
$\phi: \disk \to C_\tau = \cl(\maxcov^{-1}(\tau))$ be the internally isogonal map 
where $\phi(e^{2\pi k\im/3}) = p_k$. 
Note that our greedy choice is by support, so for $\omega=-\tau$ we choose the chains $(-\Sigma_0,-\Sigma_1,-\Sigma_2)$. 
Let 
\begin{align*}
Z_\tau(\tau) 
&= \phi(x\disk),\\ 
C_\Sigma  
&= \left\{ru : r \in [0,1]_\mb{R}, 
u \in \phi^{-1}\left(Z_{\widetilde\omega}(\Sigma)\right), 
\widetilde\omega <\omega \right\}, \\
Z_\tau(\Sigma)  
&= \phi(C_\Sigma \setminus \tfrac{1+x}2\disk^\circ) \\
Z_\tau(\Sigma\cup\{\tau\})  
&= \phi(C_\Sigma \cap \tfrac{1+x}2\disk \setminus x\disk^\circ)
\end{align*} 
That is, $C_\Sigma$ is the cone emanating from the origin that meets the circle at the union of arcs covered by $Z_{\widetilde\omega}(\Sigma)$ among $\widetilde\omega<\tau$, 
which is where the portion of $Z(\Sigma)$ that we just defined above meets $C_\tau$, 
and $Z_\tau(\Sigma)$ and $Z_\tau(\Sigma\cup\{\tau\})$ are portions of the cone $C_\Sigma$ in annuli close the unit circle.

Let $Z(A,x)$ be the subdivision of $\sphere^2$ into zones 
\[
Z(A,x,\Sigma) = \bigcup\left\{ Z_\omega(\Sigma) : \omega \in \csph(\mc{M}) \right\}.
\]

To construct $\zone(\mc{M},\lesssim,\eps,A)$ from $Z(A,x)$, 
we start with the coarsest preference $(\lesssim_\mc{M})$ favoring nonloops of $\mc{M}$ and define zonings for finer preferences recursively on the number of preferred sets. 
That is, $i\lesssim_\mc{M} j$ unless $j$ is a loop and $i$ is a nonloop. 
Let 
\begin{align*}
\zone(\mc{M},\lesssim_\mc{M},\eps,A) 
&= Z(A,s) \quad \text{where} \\
s 
&= \ext(b^{-1},\eps) 
=\begin{cases}
b^{-1}(\eps) & b(\nicefrac23) > \eps \\
\nicefrac23 & \text{otherwise}, \\
\end{cases} \\
b(x) 
&= \tailsup(b_\mr{test};x) + (1-x), \\
b_\mr{test}(x) 
&= \tight(Z(A,x),A). 
\end{align*}

We next define zones for a general preference. 
Consider 
$(\lesssim) = \pref(N_1,\dots, N_L)$ favoring nonloops of $\mc{M}$.
Let 
\begin{align*}
Z_\mr{top} 
&= \zone(\mc{N}_L,\lesssim_{\mc{N}_L},\eps,A), \\
Z_\mr{rec} 
&= \zone(\mc{N}_{L-1},\lesssim_\mr{rec},\eps_\mr{rec},A), \\
\mc{N}_{k} 
&= \rest(N_{k},\mc{M}), \\
\lesssim_\mr{rec} 
&= \pref(N_1,\dots, N_{L-1}), 
\end{align*}
and $\eps_\mr{rec}$ is half the minimum distance between 
$\wider(Z_\mr{top},i,s)$ and $S_i^{-s}$. 
That is, $\eps_\mr{rec}$ is the minimum value satisfying 
\begin{align*}
(\wider(Z_\mr{top},i,+) \oplus 2\eps_\mr{rec}) &\subseteq \cl(S_i^+), \\
(\wider(Z_\mr{top},i,-) \oplus 2\eps_\mr{rec}) &\subseteq \cl(S_i^-).
\end{align*}
for each $i \in N_{L-1}$. 
Let $\zone(\mc{M},\lesssim,\eps,A)$ be the set of zonings
$Z_\mr{rec}$ together with $Z_\mr{top}$. 
\end{definition}

\begin{claim}
\label{claimZoneAXContinuous}
$Z(A,x)$ varies continuously as $(A,x)$ vary.
\end{claim}

\begin{proof}
We first show continuity of $Z_\sigma$ for an edge covector $\sigma$. 
The region $R=\maxcov^{-1}(\sigma)$ varies continuously in Hausdorff distance by Lemma \ref{lemmaMaxcovContinuous},  
and the borders $B$ vary continuously in Fréchet distance by
Lemma \ref{lemmaMaxcovBorder}, 
so the endpoints of each $P_i$ path of $P$ vary continuously by Lemma \ref{lemmaConvergentCrossingPoint}, 
so $P_i$ varies continuously in Fréchet distance \cite[Lemma 3.2.4]{dobbins2021grassmannians},  
so $\xi = \mox(R,B,P,h)$ varies continuously in the partial map metric by Lemma \ref{lemmaMox}, 
so the edges of the cell decomposition $Z_\sigma(\Sigma)$ vary continuously in Fréchet distance by Lemma \ref{lemmaPMFrechetDistance}. 
Next consider a facet covector $\tau$. 
The map $\phi$ varies continuously in the sup-metric by Radó's theorem, 
so the edges of $Z_\tau(\Sigma)$ vary continuously in Fréchet distance by Lemma \ref{lemmaPMFrechetDistance}.  
None of the edges of $Z(A,x)$ pass through $\maxcov^{-1}(\upsilon)$ for a vertex covector $\upsilon$, so we need not consider this case. 
Thus, edges of $Z_\omega(\Sigma)$ vary continuously in Fréchet distance for each pair $(\omega,\Sigma)$, 
so $Z(A,x)$ varies continuously. 
\end{proof}

\begin{proof}[Proof of Lemma \ref{lemmaZoneMap}]
We will induct on the number of preferred sets of $(\lesssim)$. 
We start with $(\lesssim_\mc{M})$.

We first show that $Z = \zone(\mc{M},\lesssim_\mr{M},\eps,A)$ accommodates $A$. 
We only have regions $Z_\omega(\Sigma) = Z(\Sigma)\cap\maxcov^{-1}(\omega)$ for chains $\Sigma \leq \omega$, so if there is $\sigma \in \Sigma$ with $\sigma(i)=(+)$, then $\omega(i)=(+)$, 
so $Z(\Sigma) \subset S_i^+$, and similarly for the case where $\sigma(i)=(-)$.

Consider the case where $\sigma(i)=0$ for some nonloop $i$ and all $\sigma \in \Sigma$, 
and $\omega$ is an edge covector. 
Then, the arc $P_i$ of $S_i$ is among the paths $P$, in the construction of $Z_\omega(\Sigma)$, 
so $S_i$ traverses $R$ from $B_\mr{W}$ to $B_\mr{E}$, 
and $\xi(S_i)$ is an x-monotone curve, which crosses the vertical lines at $x$ and $1-x$, 
which are the endpoints of $\xi(S_i\cap Z(\Sigma))$. 
The pseudocircle $S_i$ might traverse some portion of $B_\mr{N}$, 
but $Z_\tau(\Sigma)$ for $\tau>\sigma$ is a region of $Z(\Sigma)$ on the other side of $S_i$, 
and similarly for $B_\mr{S}$, so $S_i\cap Z(\Sigma)$ is a path that intersects $\partial Z(\Sigma)$ at its endpoints. 
The chain $\Sigma$ cannot include a facet covector, 
and this also holds by a similar argument in the case where $\Sigma$ is a vertex covector or a pair of covectors, so $Z$ accommodates $A$.

Next we show that $Z$ accommodates $A$ $\eps$-tightly.
Observe that $\tailsup(b_\mr{test})$ is non-increasing, 
so $b$ is strictly decreasing,
so $s = \ext(b^{-1},\eps)$ is well-defined. 
Also, $\tight(Z,A) = b_\mr{test}(s) < b(s) \leq \eps$, 
so $Z$ accommodates $A$ $\eps$-tightly. 
%

Also, 
$b(1) = \tailsup(b_\mr{test};1) = b_\mr{test}(1) = \tight(Z(A,1);A) = \tight(A;A) = 0$, 
so if $\eps>0$, then $b^{-1}(\eps) < 1$ since $b$ is strictly decreasing, so $s<1$.
By construction we have $Z(A,s;\Sigma)$ has nonempty interior for each $\Sigma$ if $s<1$, 
so $Z$ is not degenerate if $\eps>0$.

Next we show continuity. 
The boundary of the zone $Z(A,x;\Sigma)$ 
varies continuously in Fréchet distance 
by Claim \ref{claimZoneAXContinuous}, 
so $b_\mr{test}(x) = \tight(Z(A,x),A)$ varies continuously, 
so $b_\mr{test}$ varies continuously in the sup-metric by Lemma \ref{lemmaSupDistPartialApp}, 
so $b$ varies continuously in the sup-metric, 
so $s$ varies continuously in the sup-metric by Lemma \ref{lemmaExtendedInverse}, 
so $Z = Z(A,s(\eps))$ varies continuously.

Next, we show equivariance.
The map $\mox$ in the construction of $Z_\sigma$ for an edge covector $\sigma$ is $\orth_3$-equivariant by Lemma \ref{lemmaMox}, 
so 
\begin{align*}
Z_\sigma(QA,x,\sigma) &= [\mox(Q(R,B,P),h)]^{-1}([1-x,x]\times\mb{R}) \\
&= Q[\mox(R,B,P,h)]^{-1}([1-x,x]\times\mb{R}) \\
&=Q Z_\sigma(\sigma).  
\end{align*}
and likewise for other chains $\Sigma \ni \sigma$. 
Also, $Q\phi$ is isogonal 
for the map $\phi$ the construction of $Z_\tau$ 
for a facet covector $\tau$, 
and so $Q\phi$ satisfies the defining conditions of $\phi$ with $A$ replaced by $QA$,
and $[Q\phi]^{-1}(Z_{\widetilde\omega}(QA,x,\Sigma)) 
= \phi^{-1}(Z_{\widetilde\omega}(\Sigma))$, 
so $C_\Sigma$ remains unchanged by the $\orth_3$-action on $A$, 
so $Z_\tau(QA,x,\Sigma) = QZ_\tau(A,x,\Sigma)$.  
Thus, $\zone$ is $\orth_3$-equivariant.

We now consider a general preference $(\lesssim)$ and argue by induction. 
We have continuity and equivariance and that $Z$ accommodates $A$ by induction and the same argument above applied to $Z_\mr{top}$.  

If $\eps>0$, 
then $\wider(Z_\mr{top},i,+)$ 
is a closed subset of the open region $S_i^+$, 
so $\eps_\mr{rec} > 0$  
and $Z$ is not degenerate by induction. 

If we had $\eps_\mr{rec} \geq \eps$, then 
$
\inner(Z_\mr{top},i,+)  \subset 
\wider(Z_\mr{top},i,+) 
$
would be disjoint from $\cl(S^-\oplus\eps) \subset \cl(S^-\oplus\eps_\mr{rec})$, 
and therefore bounded apart,  
by definition of $\eps_\mr{rec}$, 
but 
$\inner(Z_\mr{top},i,+)\cup (S^-\oplus\eps)$ covers $\sphere^2$  
since $Z_\mr{top}$ is $\eps$-tight, 
which is a contradiction.  
Hence, $\eps_\mr{rec} < \eps$ and  
$Z$ accommodates $A$ $\eps$-tightly by induction.

Lastly we show that 
$Z$ is indeed a generic $(\mc{M},\lesssim)$-zoning for $\eps>0$. 
Let $Z_k = \rest(\mc{N}_k,Z_\mr{rec})$
and $i \in N_k$ 
for $k<L$. Then, 
\[
\inner(Z_k,i,+)^\circ  
\supseteq \sphere^2 \setminus (S_i^-\oplus\eps_\mr{rec}) 
\supset \wider(Z_\mr{top},i,+).
\] 
The first containment holds since $Z_k$ accommodates $A$ $\eps_\mr{rec}$-tightly by induction, 
and the second containment holds 
since no point of $\wider(Z_\mr{top},i,+)$ is within distance $\eps_\mr{rec}$ 
of a point of $S^-$ 
by definition of $\eps_\mr{rec}$. 
Similarly $\inner(Z_k,i,-)^\circ \supset \wider(Z_\mr{top},i,-)$, 
so $Z$ is a $(\mc{M},\lesssim)$-zoning. 
\end{proof}

%% file: ebbz-v5-7.tex
\section{Zone ebb}
\label{sectionEbbZ}

Here we define the deformation retraction $\ebb_\mr{D}$ of Lemma \ref{lemmaEbbD} 
using a map $\ebb_\mr{Z}$ with the following properties.

\begin{lemma}[Zone ebb wrapper]
\label{lemmaGebbWrap}
Given a preference $(\lesssim) = \pref(N_1,\dots, N_L)$ 
favoring nonloops of $\mc{M}$, 
a generic or fully degenerate 
$(\mc{M},\lesssim)$-zoning 
that accommodates $A \in\upa(\mc{M})$, 
and $t\in [0,1]_\mb{R}$,  
then $A_t = \ebb_\mr{Z}(Z,A,t) \in \upa(\mc{M})$ 
satisfies the following:  
\begin{enumerate}
\item
\label{itemGebbWrapContinuous} 
$A_t$ varies continuously as $(Z,A,t)$ vary.
\item
\label{itemGebbWrapZero} 
$A_0 = A$.
\item
\label{itemGebbWrapOne}
$A_1 \in \psv(\mc{M})$. 
\item
\label{itemGebbWrapStrong}
If $Z = A \in \psv(\mc{M})$ is degenerate, then $A_t = A$ is unchanging. 
\item
\label{itemGebbWrapEquivariant} 
$\ebb_\mr{Z}$ is equivariant, i.e., 
$\ebb_\mr{Z}(Q(Z,A),t) = Q A_t$ for $Q \in \orth_3$. 
\item 
\label{itemGebbWrapOT}
For $t<1$, order type does not change, i.e., $\ot(A_t) = \ot(A)$. 
\item 
\label{itemGebbWrapGuard}
$\facet(Z)$ accommodates $A_t$, 
\item 
\label{itemGebbWrapWeight}
weights of loops decrease 
and weights of nonloops do not change, 
i.e., $\wt_i(A_t) \leq \wt_i(A)$ 
with equality for nonloops. 
\end{enumerate}
\end{lemma}

We will see later that $\ebb_\mr{Z}$ depends on a given zoning $Z$ that accommodates $A$, but that $Z$ might not continue to accommodate $A_t$ as it evolves, so while $\ebb_\mr{Z}$ has many properties of a deformation retraction, $\ebb_\mr{Z}$ is not exactly a deformation retraction, since $A_t$ leaves the domain of $\ebb_\mr{Z}(Z)$ for a fixed $Z$.  Nevertheless, we will get a deformation retraction $\ebb_\mr{D}$ using the continuous dependence on $Z$.

\begin{definition}[Diff ebb]
\label{defEbbD}
Given a chain $\mc{C}= \{\mc{M}_1 < \dots < \mc{M}_L\}$ and $\eps>0$, let 
\begin{align*}
 \ebb_\mr{D}(\mc{C},\eps;A,t) 
&= \ebb_\mr{Z}(Z,A,t)
\intertext{where}   
Z &= \zone(\mc{M}_L,\lesssim,\eps_\mr{Z},A), \\
\eps_\mr{Z} &= \min(\eps,\maxlit(\mc{M}_L,A)), 
\end{align*}
and 
$i\lesssim j$ when $i$ is a loop of more oriented matroids of $\mc{C}$ than $j$ is.

\end{definition}

\begin{remark}
\label{remarkEbbDFavorNonloops}
$(\lesssim)$ favors nonloops of $\mc{C}$ since the sets of nonloops of the oriented matroids of a chain are nested.  
\end{remark}

We will use the following to show that maxlit and minbig do not grow and shrink too much respectively.

\begin{lemma}
\label{lemmaNestedDisks}
If $D\subseteq C$ are nested 2-cells, 
then $\delta = \dist_\mr{F}(\partial D, \partial C) \geq \dist_\mr{H}(D,C)$.
Moreover, every point $p \in C \setminus D$ is strictly within distance $\delta$ of $\partial C$. 
\end{lemma}

\begin{proof}
By definition of Fréchet distance, for each $\eps>0$, 
there is a homeomorphism $\psi: \partial D \to \partial C$ such that 
$\|\psi(x)-x\|<\delta+\eps$. 
We can construct a deformation retraction $\rho$ from $\mb{R}^2\setminus p$ to $\partial C$
since $\partial C$ is a Jordan curve.  
We will also construct a map $\xi : \widetilde D \to \mb{R}^2$ 
from a 2-cell $\widetilde D$ 
that consists of $D$ together with formal linear interpolations between each $x \in \partial D$ and its image $\psi(x) \in \partial C$. 
That is, $\widetilde D$ is the union of 
$D$ and $\partial D \times [0,1]_\mr{R}$ 
with $\partial D$ and $\partial D \times \{0\}$ identified. 
Note that $\partial \widetilde D = \partial D \times \{1\}$. 
Let $\xi : \widetilde D \to \mb{R}^2$ be the identity on $D$ 
and send $(x,t) \in \partial D \times [0,1]$ 
to $\xi(x,t) = t\psi(x)+(1-t)x$.  
Note that the definitions of $\xi$ on $D$ and $\partial D \times [0,1]_\mr{R}$ agree on $\partial D$ and $\partial D \times \{0\}$.
If $\xi$ never hit $p$, then we would get a homotopy equivalence between $\widetilde D$ and $\partial C$ 
via the map $\rho\circ\xi$, but that is impossible, since $\widetilde D$ is simply connected, but $\partial C$ is not. 
The restriction of $\xi$ to $D$ cannot hit $p$, since $p \not\in D$, 
so there must be some $(x,t)$ such that $p = \xi(x,t) = t\psi(x)+(1-t)x$ with $t>0$. 
Hence, $p$ is within distance $(1-t)(\delta+\eps)$ of $\partial C$, 
and this holds for all $\eps > 0$, so the distance from $p$ to $\partial C$is strictly less than $\delta$. 
\end{proof}

\begin{lemma}
\label{lemmaWeakTope}
If $\mc{M}_0 \leq_\mr{w} \mc{M}_1$ have the same set of loops, 
then each facet covector $\sigma$ of $\cov(\mc{M}_0)$ is also a facet covector of $\cov(\mc{M}_1)$,
and $\maxcov^{-1}(\mc{M}_0,A,\sigma) = \maxcov^{-1}(\mc{M}_1,A,\sigma)$.
\end{lemma}

\begin{proof}
Consider a facet covector $\sigma$ of $\cov(\mc{M}_0)$. 
Then, $\maxcov(\mc{M}_0,\mc{M}_1)$ is a surjective map from $\cov(\mc{M}_1)$ to $\cov(\mc{M}_0)$ \cite{anderson2001representing}, 
so there is some $\tau \in \maxcov^{-1}(\mc{M}_0,\mc{M}_1,\sigma)$, which is a covector of $\mc{M}_1$ such that $\sigma \leq_\mr{v} \tau$. 
Since $\tau$ is a covector of $\mc{M}_1$, 
the support of $\tau$ is contained in the set of nonloops $N$ of $\mc{M}_1$, 
which is also the set of nonloops of $\mc{M}_0$. 
Since $\sigma$ is a facet covector of $\cov(\mc{M}_0)$, the support of $\sigma$ is $N$, 
and so the support of $\tau$ must contain $N$, since $\sigma\leq_\mr{v} \tau$.  
Hence, $\sigma$ and $\tau$ have the same support, so $\sigma=\tau$.  

For the last part observe that since $\sigma$ is a facet covector of $\mc{M}_0$, we have 
\[
\maxcov^{-1}(\mc{M}_0,A,\sigma) = \bigcup \{\cell(A,\upsilon) : \upsilon \in \cov(A), \upsilon \geq_\mr{v} \sigma\}
\]
and likewise for $\mc{M}_1$.
\end{proof}

Let $\phi = \ebb_\mr{D}(\mc{C},\eps)$, 
and $A_t = \phi(A,t)$, 
and $S_{i,t} = S_i(A_{t})$.
Let $N_k$ be the set of nonloops of $\mc{M}_k$, and let $\mc{N}_k = \rest(\mc{M}_L,N_k)$.

\begin{claim}
\label{claimEbbDTopeShrink}
Given a facet covector $\sigma$ of $\mc{M}_k$, 
the inradius of $C_t = \maxcov^{-1}(\mc{M}_k,A_t,\sigma)$ is at most $\eps$ smaller than that of $C_0$. 
\end{claim}

\begin{proof}
Suppose not.
Then, we could find a disk $D_0 \subseteq C_0$ of radius $r>\eps$ such that no disk of radius $r-\eps$ is contained in $C_t$. 
Let $D_1$ be the disk of radius $r-\eps$ concentric with $D_0$. 
Then, there must be some point $p \in D_1 \setminus C_t$, 
so $p \in D_0$ is at least distance $\eps$ away from $\partial D_0$, 
so $p$ is at least distance $\eps$ away from $\partial C_0$.  
Also, $\sigma$ is a facet covector of $\mc{N}_k$ by Lemma \ref{lemmaWeakTope}, 
and $N_k$ is a preferred set of $(\leq_\mc{C})$ by Remark \ref{remarkEbbDFavorNonloops}, 
and $\facet(Z)$ accommodates $A_t$ by Lemma \ref{lemmaGebbWrap} part \ref{itemGebbWrapGuard}, 
so we have a zone 
\[
Z(\sigma) \subseteq \maxcov^{-1}(X,\mc{N}_k,A_t) = C_t.
\]

$Z$ accommodates $A$ $\eps_\mr{Z}$-tightly by Lemma \ref{lemmaZoneMap}, 
so $\partial Z(\sigma)$ is within Fréchet distance $\eps_\mr{Z}\leq\eps$ of $\partial C_0$, 
and $p \not\in Z(\sigma)$, so 
$p$ is strictly closer than $\eps$ to $\partial C_0$ by Lemma \ref{lemmaNestedDisks}, 
which is a contradiction.
\end{proof}

\begin{claim}
\label{claimEbbDVanGrow}
The inradius of $\van(\mc{M}_k,A_t)$ is at most $\eps$ greater than that of $\van(\mc{M}_k,A_0)$. 
\end{claim}

\begin{proof}
Suppose not. 
Then we could find a metric disk $D_1 \subset \van(\mc{M}_k,A_t)$ of radius $r > \eps$ such that no disk of radius $r-\eps$ is contained in $\van(\mc{M}_k,A_0)$. 
Let $D_0$ be the metric disk concentric with $D_1$ with radius $r-\eps$. 
Then, $D_0$ must intersect $C_0$ at a point $p$ 
for some some cell $C_t = \maxcov^{-1}(\mc{M}_k,A_t,\sigma)$ 
at time $t=0$, where $\sigma$ is a facet covector of $\mc{M}_k$. 
Then, $\sigma$ is also a facet covector of $\mc{N}_k$ by Lemma \ref{lemmaWeakTope}, 
and $N_k$ is a preferred set of $(\mc{M}_L,\lesssim)$ by Remark \ref{remarkEbbDFavorNonloops}. 
Also, $\facet(Z)$ accommodates $A_t$ by Lemma \ref{lemmaGebbWrap} part \ref{itemGebbWrapGuard}, so 
$Z(\sigma)\subset C_t$ is disjoint from $\van(\mc{M}_k,A_t)$, 
which contains $D_1$, 
so $D_1$ cannot intersect $Z(\sigma)$.
Hence, $p\in D_0$ is distance at least $\eps$ away from $Z(\sigma)$, but $p$ is contained in $C_0$, 
so $C_0$ and $Z(\sigma)$ are at least $\eps$ apart in Hausdorff distance, 
so $\partial C_0$ and $\partial Z(\sigma)$ are at least $\eps$ apart in Fréchet distance 
by Lemma \ref{lemmaNestedDisks}, which is a contradiction since $Z$ accommodates $A$ $\eps$-tightly in Definition \ref{defEbbD} by Lemma \ref{lemmaZoneMap}.
\end{proof}

\begin{proof}[Proof of Lemma \ref{lemmaEbbD}]
The map 
$\ebb_\mr{D}(\mc{C})$ is continuous 
since $Z$ varies continuously as $(\eps_\mr{Z},A)$ vary by Lemma \ref{lemmaZoneMap}, 
and $\eps_\mr{Z}$ varies continuously by Lemma \ref{lemmaBigLittleContinuous}, 
and $\ebb_\mr{Z}$ is continuous by Lemma \ref{lemmaGebbWrap} part \ref{itemGebbWrapContinuous}, 
so part \ref{itemEbbDContinuous} holds. 
Order type does not change for $t<1$ 
by part \ref{itemGebbWrapOT} of Lemma \ref{lemmaGebbWrap}, 
so part \ref{itemEbbDOT} holds. 

We have $A_0 = A$, $A_1 \in \pstief(\mc{M}_L)$, 
and $\ebb_\mr{D}(\mc{C},\eps)$ is equivariant 
by the respective parts \ref{itemGebbWrapZero}, \ref{itemGebbWrapOne}, 
and \ref{itemGebbWrapEquivariant} 
of Lemma \ref{lemmaGebbWrap}. 
In the case where $A \in \pstief(\mc{M}_L)$, 
we have $\maxlit(\mc{M}_L,A)=0$, so $\eps_\mr{Z}=0$, 
so $Z = \zone(\mc{M}_L,\lesssim,0,A) = A$ since $Z$ is 0-tight by Lemma \ref{lemmaZoneMap}, 
so $\ebb_\mr{D}(\mc{C},\eps)$ is trivial on $\pstief(\mc{M}_L)$ 
by Lemma \ref{lemmaGebbWrap} part \ref{itemGebbWrapStrong}. 
Thus, $\ebb_\mr{D}(\mc{C},\eps)$ is a strong equivariant deformation retraction to $\psv(\mc{M}_L)$.

By Claim \ref{claimEbbDTopeShrink}, 
$\inrad(\mc{M}_k,\sigma,\phi(A,t)) \geq \inrad(\mc{M}_k,\sigma,A)-\eps$ for each facet covector $\sigma$ of $\mc{M}_k$, 
and the nonloops of $\mc{M}_k$ are also nonloops of $\mc{M}_L \geq_\mr{w} \mc{M}_k$, 
so the weights of nonloops of $\mc{M}_k$ do not change by Lemma \ref{lemmaGebbWrap} part \ref{itemGebbWrapWeight}. 
Thus, $\minbig(\mc{M}_k,\phi(A,t)) \geq \minbig(\mc{M}_k,A)-\eps$, 
so part \ref{itemEbbDMinbig} holds.

By Claim \ref{claimEbbDVanGrow}, 
$\vanrad(\mc{M}_k,\phi(A,t)) \leq \vanrad(\mc{M}_k,A)+\eps$, and weights do not increase by Lemma \ref{lemmaGebbWrap} part \ref{itemGebbWrapWeight}, 
so $\maxlit(\mc{M}_k,\phi(A,t)) \leq \maxlit(\mc{M}_k,A)+\eps$, 
so part \ref{itemEbbDMaxlit} holds. 
\end{proof}

\subsection{Zone ebb accumulator}



To construct $\ebb_\mr{Z}$, we will recursively build the output for successively larger preferred sets.  For this, we define an accumulator map $\acc$ that takes as part of its input a partially constructed output and contributes the portion of the output for the next preferred set. 
In this subsection, we give the inductive argument for the accumulator and prove Lemma \ref{lemmaGebbWrap}. 
in the next subsection, 
we will construct the deformation $\step$ that the accumulator uses to construct the next portion of the output and show it has the desired properties. 

Recall that $\ext(\Omega)$ denotes the space of all extensions of an arrangement $\Omega$ on a ground set $N\subset [n]$ to an arrangement in $\psv_{3,n}$.

\newcounter{counterGuardEbb}

\begin{lemma}[Zone ebb accumulator]
\label{lemmaGebbAcc}
Let $N_L$ be the set of nonloops of $\mc{M}$, $N_0=\emptyset$, 
\[
(\lesssim) = \pref(N_1,\dots, N_L), 
\quad 
k \in \{0,\dots,L\}, 
\quad 
\mc{N}_k = \rest(N_k,\mc{M}),
\]
$Z$ be either a generic or fully degenerate $(\mc{M},\lesssim)$-zoning, 
\[
Z \text{ accommodate } A \in \upa(\mc{M}), 
\quad 
Z_k = \rest(\mc{N}_k,Z) \text{ accommodate } \Omega \in \psv(\mc{N}_k),  
\]
\begin{equation}
\label{equationSAOmega}
S_i(A)\cap Z_k(\Sigma) = S_i(\Omega)\cap Z_k(\Sigma) 
\end{equation}
for each chain $\Sigma$ 
where 
$i=i_\Sigma$ is the greedy choice from $\bigcap\Sigma^0$.
%
Then, 
\[\psi_t = \acc(\Omega,Z,A,t) : \sphere^2 \to \sphere^2\]
satisfies the following. 
\begin{enumerate}
\item 
\label{itemGebbAccHomeo}
$\psi_t \in \hom^{+}(\sphere^2)$ for $t<1$, 
\item 
\label{itemGebbAccContinuous}
$\psi_t$ varies continuously in the sup-metric as $(\Omega,Z,A,t)$ vary, 
\item 
\label{itemGebbAccZero}
$\psi_0 = \id$. 
\item 
\label{itemGebbAccOne}
$\psi_t*A \to \psi_1*A \in \psv(\mc{M})\cap\ext(\Omega)$ as $t\to 1$, 
\item 
\label{itemGebbAccStrong}
if $Z$ is degenerate, then $\psi_t = \id$, 
\item 
\label{itemGebbAccEquivariant}
$\acc$ is $\orth_3$-equivariant, i.e., 
$\acc(Q(\Omega,Z,A),t) = Q\psi_t$, 
\item 
\label{itemGebbAccFacet}
$\psi_t$ is the identity on $Z_L(\sigma)$ for each facet covector $\sigma$ of $\csph(\mc{M})$, 
\item 
\label{itemGebbAccOmega}
$\psi_t(S_i(\Omega)) = S_i(\Omega)$ for $i \in N_{k}$,
\item 
\label{itemGebbAccInner}
$\psi_1(\inner(Z_L,i,0)) = S_i(\Omega)$ 
for $i \in N_k$. 
\setcounter{counterGuardEbb}{\value{enumi}}
\end{enumerate}
\end{lemma}

\begin{lemma}[Zone ebb recursive step]
\label{lemmaGebbStep}
Given the setup of Lemma \ref{lemmaGebbAcc} 
with $k=L$ 
or $k=L-1$, 
but with a $\mc{M}$-zoning $Z=Z_k=Z_L$ 
that accommodates both $A$ and $\Omega$ 
where (\ref{equationSAOmega}) holds for $i\in\supp(\Omega)$, 
then 
$\psi_t = \step(\lesssim,\Omega,Z,A,t)$ 
satisfies the \thecounterGuardEbb\ conditions stated in Lemma \ref{lemmaGebbAcc} along with the following.
\begin{enumerate}
\setcounter{enumi}{\value{counterGuardEbb}}
\item 
\label{itemGebbStepGreedy}
$\psi_t(S_i(A))\cap Z(\Sigma) = S_i(A)\cap Z(\Sigma)$ 
where $i$ is the greedy choice from $\bigcap\Sigma^0$. 
\item 
\label{itemGebbStepInner} 
$\psi_t(\inner(Z,i,0))\subset\inner(Z,i,0)^\circ$ for $t>0$ and $Z$ generic, 
\item 
\label{itemGebbStepWider}
$\psi_t(\wider(Z,i,0))\subseteq\wider(Z,i,0)$, 
\end{enumerate}
\end{lemma}

Note that 
for $Z$ to accommodate $\Omega$ in the hypotheses of Lemma \ref{lemmaGebbStep}, 
only the conditions of Definition \ref{defAccommodate} for $i\in N_k$ need to be satisfied   
since $N_k$ is the ground set of $\Omega$.


To deform an arrangement in $\upa(\mc{M})$, we also have to reduce the weights of loops to 0.

\begin{definition}[Zone ebb wrapper]
\label{defGebbWrap}
\[
\ebb_\mr{Z}(Z,A,t) = \acc(\emptyset,Z,A,t)*A_{\wt}(t)
\]
where $A_{\wt}(t)$ is the arrangement where loops of the oriented matroid $\mc{M}$ associated to $Z$ are scaled by $1-t$.  
That is, 
\[
\wt_i(A_{\wt}(t)) = \begin{cases} 
(1-t)\wt_i(A) & i \in N \\
\wt_i(A) & i \not\in N, 
\end{cases}
\] 
$N$ is the set of nonloops of $\mc{M}$, 
and $S_i(A_{\wt}(t)) = S_i(A)$ is unchanging 
except that $S_i(A_{\wt}(1))=0$ if $i$ is a loop.
\end{definition}

\begin{proof}[Proof of Lemma \ref{lemmaGebbWrap}]
Each part of the lemma follows immediately from the corresponding part of Lemma \ref{lemmaGebbAcc} as a special case where $\Omega=\emptyset$, except 
parts \ref{itemGebbWrapWeight} and \ref{itemGebbWrapOne} and 
part \ref{itemGebbWrapContinuous} in the case where $t\to 1$ from below.  

Applying $\acc(\emptyset,Z,A,t)$ to $A_{\wt}(t)$ does not change weights, 
and $A_{\wt}(t)$ only decreases weights of loops, so part \ref{itemGebbWrapWeight} holds. 

Consider $t\to 1$ from below. 
Then, 
\begin{align*}
\rest(N,\ebb_\mr{Z}(Z,A,t)) 
&= \rest(N,\acc(\emptyset,Z,A,t)*A_{\wt}(t)) \\    
&= \acc(\emptyset,Z,A,t)*\rest(N,A) \\    
&\to \acc(\emptyset,Z,A,1)*\rest(N,A) \in \rest(N,\psv(\mc{M})) 
\end{align*}
by Lemma \ref{lemmaGebbAcc} parts \ref{itemGebbAccContinuous} and \ref{itemGebbAccOne}, 
and loops $i\not\in N$ are scaled by $(1-t)$ in $\ebb_\mr{Z}(Z,A,t)$ 
by definition of $A_{\wt}(t)$,  
so loops vanish as $t\to 1$, 
so $\ebb_\mr{Z}(Z,A,t) \to \ebb_\mr{Z}(Z,A,1)$, 
which means parts \ref{itemGebbWrapOne} and \ref{itemGebbWrapContinuous} hold. 
\end{proof}

\begin{definition}[Zone ebb accumulator]
In the case where $k\in\{L,L-1\}$,  
let $\acc(\Omega,Z,A,t) = \step(\lesssim,\Omega,Z,A,t)$. 
Otherwise, let 
\begin{align*}
\acc(\Omega,Z,A,t) 
&= \psi_\mr{com}(t,s) \\
\intertext{where}
\psi_\mr{com}(t,x) 
&= \psi_\mr{rec}(t)\circ \psi_\mr{new}(x), \\
\psi_\mr{rec}(t) 
&= \acc(\Omega_\mr{rec},Z,A,t), \\
\psi_\mr{new}(x) 
&= \step(\lesssim,\Omega,Z_{k+1},A,x), \\
\Omega_\mr{rec} 
&= \psi_\mr{new}(1)*\rest(N_{k+1},A), \\ 
s 
&= \ext(b_t^{-1},b_0(0)) 
= \begin{cases}
b_t^{-1}(b_0(0)) & b_t(0) < b_0(0) \\
0 & \text{otherwise}, 
\end{cases} \\
b_t(x)  
&= \text{tailinf}(b_\mr{test},x)+c(x-1), \\
b_{\mr{test}}(x) 
&= b_{\mr{test},t}(x) 
= \min\{\dist_\mr{H}(Z(\sigma), C_{\sigma,t,x}) 
: \sigma\in\facet(\mc{N}_{k+1})\}, \\
C_{\sigma,t,x} 
&= \psi_\mr{com}(t,x,\maxcov^{-1}(\mc{N}_{k+1},A,\sigma)), \\
c 
&= (\nicefrac12)\inf\{b_{\mr{test},0}(x) : x \in [0,1]_\mb{R}\}. 
\end{align*}
\end{definition}

Next, we will prove Lemma \ref{lemmaGebbAcc}. 
We assume that 
Lemma \ref{lemmaGebbAcc} holds for $\psi_\mr{rec}(t)$ 
by induction on the size of $N_L\setminus N_k$ 
and we assume that Lemma \ref{lemmaGebbStep} holds. 
We start with some claims needed for the proof,  
the first of which are that 
the respective inputs to $\step$ and $\acc$ in the definitions of $\psi_\mr{new}$ and $\psi_\mr{rec}$ are valid. 


\begin{claim}
\label{claimGebbAccPsiNew}
$(\lesssim,\Omega,Z_{k+1},A)$ satisfies the hypotheses of Lemma \ref{lemmaGebbStep}. 
Hence,  
$\psi_\mr{new}$ is well-defined. 
\end{claim}

\begin{proof}
Let us first show that Equation \ref{equationSAOmega} holds. 
Consider a chain $\Sigma$ of $\csph(\mc{N}_{k+1})$ with highest covector $\sigma = \max\Sigma$ 
and let $i\in N_k$ be the greedy choice from $\bigcap \Sigma^0=\sigma^0$. 
Let $\tau = \rest(N_k,\sigma)$. 
Then, $\sigma^0 \cap N_{k+1} = \sigma^0 \cap N_{k}$, otherwise the greedy choice from $\sigma^0$ would be chosen from $N_{k+1}$, 
so $i$ is also the greedy choice from $\tau^0$.   
Additionally, 
\begin{align*}
Z_{k+1}(\Sigma) 
&\subset \bigcap\left\{ \wider(Z_{k+1},j,\sigma) : j\in \supp(\sigma)\right\} \\ 
&\subset \bigcap\left\{ \wider(Z_{k+1},j,\tau) : j\in \supp(\tau) = \supp(\sigma)\cap N_k \right\} \\ 
&\subseteq \bigcap\left\{ \inner(Z_{k},j,\tau) : j\in \supp(\tau)\right\} \\ 
&= \bigcup\left\{Z_k(\widetilde \Sigma) : \forall\, \widetilde\sigma \in \widetilde\Sigma, j\in \supp(\tau) : \widetilde\sigma(j) = \tau(j) \right\} \\
&\subseteq \bigcup \left\{ Z_k(\widetilde \Sigma) : \widetilde \Sigma \in \oc(\csph(\mc{N}_k)), \max \widetilde \Sigma \geq \tau \right\}
\end{align*} 
where the third containment is 
by Definition \ref{defPreference} since $Z$ is a $(\mc{M},\lesssim)$-zoning. 

Consider $\widetilde \Sigma \in \oc(\csph(\mc{N}_k))$ such that $\max \widetilde \Sigma \geq \tau$.  
Either, $i\in\max(\widetilde \Sigma)^0$, in which case  
$i$ is the greedy choice from $\bigcap\widetilde \Sigma^0$ 
since $\max(\widetilde \Sigma)^0 \subseteq \tau^0$, 
so $S_i(\Omega)\cap Z_{k}(\widetilde \Sigma) = S_i(A)\cap Z_{k}(\widetilde \Sigma)$ 
by Equation \ref{equationSAOmega} for $Z_k$ in the hypotheses of the lemma. 
Or, $i\not\in\max(\widetilde \Sigma)^0$, 
in which case $S_i(\Omega)\cap Z_{k}(\widetilde \Sigma) = \emptyset = S_i(A)\cap Z_{k}(\widetilde \Sigma)$ 
since $Z_k$ accommodates both $\Omega$ and $A$. 
Hence, $S_i(\Omega)\cap Z_{k+1}(\Sigma) = S_i(A)\cap Z_{k+1}(\Sigma)$ 
since this holds analogously in each zone $Z_k(\widetilde \Sigma)$, 
so (\ref{equationSAOmega}) holds for $Z_{k+1}$. 

Next, let us show that $Z_{k+1}$ accommodates $\Omega$. 
According to Definition \ref{defAccommodate}, we only need to consider $S_i(\Omega)$ for $i \in N_k$, which is the ground set of $\Omega$. 
In the case where $Z$ is fully degenerate we have $\rest(N_k,Z_{k+1}) = \rest(N_k,\rest(N_{k+1},A)) = \rest(N_k,A) = \Omega$, 
so $Z_{k+1}$ is a fully degenerate $\mc{N}_{k+1}$-zoning the accommodates $\Omega$. 
Let us consider the case where $Z$ is generic. 
Then,  
\[
S_i^+(\Omega) 
\supset \wider(Z_k,i,+) 
\supset \inner(Z_{k},i,+) 
\supseteq \wider(Z_{k+1},i,+)
\]
where the first containment holds since $Z_k$ accommodates $\Omega$ 
and the last containment holds since $Z$ is a $(\mc{M},\lesssim)$-zoning, 
and similarly for $S_i^-(\Omega)$, 
and so $S_i(\Omega)\subset \inner(Z_{k+1},i,0)^\circ$ by Remark \ref{remarkInnerWiderSubdivision}.

Consider the intersection of $S_i(\Omega)$ and an inner $(i,0)$-zone $Z_{k+1}(\Sigma)$. 
Let us first consider the case where $\Sigma = \{\upsilon,\sigma\}$ has 2 elements,  
and let $\sigma$ be an edge covector.  
Then, 
\[
S_i(\Omega)\cap Z_{k+1}(\Sigma) 
= S_{i_0}(\Omega)\cap Z_{k+1}(\Sigma) 
= S_{i_0}(A)\cap Z_{k+1}(\Sigma) 
\]
where $i_0$ is the greedy choice from the span of $i$ 
since $\Omega \in \psv(\mc{N}_k)$ 
and by (\ref{equationSAOmega}).
Therefore, $S_i(\Omega)\cap Z_{k+1}(\Sigma)$ is a path with endpoints on $\partial Z_{k+1}(\Sigma)$ 
since $Z_{k+1}$ accommodates $A$. 

Now consider the case where $\Sigma$ is a singleton. 
Then, $\Sigma = \{\upsilon\}$ is contained in exactly two other chains $\Sigma_j = \{\upsilon,\sigma_j\}$ of $\csph(\mc{N}_{k+1}/i)$ 
since the contraction $\mc{N}_{k+1}/i$ has rank 2 and $\Cov(\mc{N}_{k+1}/i)$ is a thin lattice,  
so $Z_{k+1}(\Sigma)$ is adjacent to exactly two other inner $(i,0)$-zones, namely $Z_{k+1}(\Sigma_j)$. 
From the previous case, we have  
$S_i(\Omega)\cap Z_{k+1}(\Sigma_j) = S_{i_0}(A)\cap Z_{k+1}(\Sigma_j)$ 
since $\Sigma_j$ has two elements, 
so $S_i(\Omega)\cap Z_{k+1}(\Sigma)$ intersects the boundary $\partial Z_{k+1}(\Omega)$ at exactly two points, 
and therefore $S_i(\Omega)\cap Z_{k+1}(\Sigma)$ is a path.  
Thus, $Z_{k+1}$ accommodates $\Omega$, 
and the rest of the hypotheses of Lemma \ref{lemmaGebbStep} follow immediately from our assumption that $(\Omega,Z,A)$ satisfy the hypotheses of Lemma \ref{lemmaGebbAcc}.  
\end{proof}

\begin{claim}
\label{claimGebbAccPsiRec}
$(\Omega_\mr{rec},Z,A)$ satisfies the hypotheses of Lemma \ref{lemmaGebbAcc}. 
Hence,  
$\psi_\mr{rec}$ is well-defined. 
\end{claim}

\begin{proof}
Consider a chain $\Sigma \subset \csph(\mc{N}_j)$, let $\sigma$ be the highest covector of $\Sigma$, 
and let $i$ be the greedy choice from $\sigma^0$. 
Then, 
\[
S_i(A)\cap Z_{k+1}(\Sigma) 
= \psi_\mr{new}(1,S_i(A))\cap Z_{k+1}(\Sigma) 
= S_i(\Omega_\mr{rec})\cap Z_{k+1}(\Sigma)
\]
by Lemma \ref{lemmaGebbStep} part \ref{itemGebbStepGreedy}, 
so Equation \ref{equationSAOmega} holds. 

Next, let us show that $Z_{k+1}$ accommodates $\Omega_\mr{rec}$. 
In the case where $Z$ is fully degenerate, we have 
$Z_{k+1} = \rest(N_{k+1},A)=\Omega_\mr{rec}$ 
since $Z$ accommodates $A$ and by part \ref{itemGebbAccStrong} of Lemma \ref{lemmaGebbStep}, 
so $Z_{k+1}$ accommodates $\Omega_\mr{rec}$. 
Let us consider the case where $Z$ is generic. 
Then, 
$S_i(A) \subset \inner(Z_{k+1},i,0)$ 
since $Z_{k+1}$ accommodates $A$, so 
\[
S_i(\Omega_\mr{rec}) = \psi_\mr{new}(1,S_i(A)) \subset \inner(Z_{k+1},i,0)^\circ  
\]
by Lemma \ref{lemmaGebbStep} part \ref{itemGebbStepInner}, 
and $S_i^+(\Omega_\mr{rec})  \supset \wider(Z_{k+1},i,+)$ 
and similarly for $S_i^-$ 
by Remark \ref{remarkInnerWiderSubdivision}.
Also, the intersection of $S_i(\Omega_\mr{rec})$ with each inner $(i,0)$-zone $Z_{k+1}(\Sigma)$ 
is a path with endpoints on $\partial Z_{k+1}(\Sigma)$ by the same argument as in Claim \ref{claimGebbAccPsiNew}.
The rest of the hypotheses hold by our assumption that $(\Omega,Z,A)$ satisfies the hypotheses of the Lemma \ref{lemmaGebbAcc}. 
\end{proof}





\begin{claim}
\label{claimGebbAccF}
$b_t$ is continuous and varies continuously in the sup-metric as $(\Omega,Z,A,t)$ vary.
\end{claim}

\begin{proof}
The maps $\psi_\mr{rec}(t)$ and $\psi_\mr{new}(x)$ vary continuously in the sup-metric as $(\Omega,Z,A,t)$ and $x$ vary by part \ref{itemGebbAccContinuous} and induction, 
so $\psi_\mr{com}(t,x)$ varies continuously. 
Also, $\maxcov^{-1}(\sigma)$ varies continuously in Hausdorff distance 
by Lemma \ref{lemmaMaxcovContinuous}, 
so $C_{\sigma,t,x}$ varies continuously,  
so $b_{\mr{test}}(x)$ varies continuously,
so $b_t(x)$ varies continuously by Lemma \ref{lemmaTailsupContinuous}.
Hence, $b_t$ varies continuously in the sup-metric by Lemma \ref{lemmaSupDistPartialApp} 
since $x$ is chosen from a compact set.  
\end{proof}

\begin{claim}
\label{claimGebbAccC}
$b_0(0) = c > 0$.
\end{claim}

\begin{proof}
$\psi_\mr{com}(0,x) = \psi_\mr{new}(x)$ by Lemma \ref{lemmaGebbAcc} part \ref{itemGebbAccZero} and induction, so 
$F_{k+1} = \facet(Z_{k+1})$ accommodates  
$\psi_\mr{com}(0,x)*A = \psi_\mr{new}(x)*A$
by Lemma \ref{lemmaGebbStep} part \ref{itemGebbAccFacet}, 
so $Z(\sigma)$ is in the interior of $C_{\sigma,0,x}$ for each facet covector $\sigma$, 
so $b_{\mr{test},0}(x) > 0$, and since $x$ is chosen from a compact domain, 
we have 
$\tailinf(b_{\mr{test},0},0) = 2c > 0$ by definition of $c$, 
so $b_0(0) = \tailinf(b_{\mr{test},0},0) -c = c > 0$.
\end{proof}

\begin{claim}
\label{claimGebbAccVert}
Vertices of $\Omega_\mr{rec}$ are fixed points of $\psi_\mr{rec}(t)$, 
and edges of $\Omega_\mr{rec}$ are preserved by $\psi_\mr{rec}(t)$. 
\end{claim}

\begin{proof}
Each vertex $v$ of $\Omega_\mr{rec}$ is on a pair of curves $[S_i\cap S_j](\Omega_\mr{rec}) = \{u,v\}$ for some independent pair $i,j$ of $\mc{N}_{k+1}$, which only meet at a pair of points, and $\psi_\mr{rec}(t;v)$ moves continuously by part \ref{itemGebbAccContinuous} and stays within the pair $\{u,v\}$ by part \ref{itemGebbAccOmega} for our inductive assumption. 
Hence, $u$ and $v$ must be fixed points, and edges are preserved by part \ref{itemGebbAccOmega}. 
\end{proof}

\begin{claim}
\label{claimGebbAccS}
$s$ is well-defined, varies continuously, and $s<1$. 
\end{claim}

\begin{proof}
Since $c>0$ by Claim \ref{claimGebbAccC}, 
and $b_t$ is strictly increasing since $\mr{tailinf}(b_\mr{test})$ is non-decreasing by definition and $c(x-1)$ is strictly increasing since $c>0$ by Claim \ref{claimGebbAccC}, 
and $b_t$ is continuous by Claim \ref{claimGebbAccF},
so $b_t^{-1}$ is well-defined on the interval $[b_t(0),b_t(1)]_\mb{R}$.

We have $\psi_\mr{new}(1)*A = \Omega_\mr{rec}$ by definition, 
so the boundary of 
$C_{\sigma,0,1} = \cell(\sigma,\Omega_\mr{rec})$ 
consists of edges of $\Omega_\mr{rec}$, 
so $\partial C_{\sigma,t,1} = \psi_\mr{rec}(t;\partial C_{\sigma,0,1}) = \partial C_{\sigma,0,1}$ by Claim \ref{claimGebbAccVert}, 
so $b_\mr{test}(1)$ is unchanging as $t$ varies, so 
$b_t(1) = b_\mr{test}(1) = b_{\mr{test},0}(1) = b_0(1)$,  
and $b_0$ is strictly increasing, 
so $b_0(0) < b_0(1) = b_t(1)$.
Hence, either $b_0(0)$ is in the range of $b_t$ or $b_0(0) < b_t(0)$.
In either case $s$ is well defined. 
The only way we can have $s=1$ would be $b_t(1) = b_0(0)$, 
which we have just seen is impossible, so $s<1$.  

Since $b_t$ varies continuously in the sup-metric, 
we can extend $b_t^{-1}$ to a function that varies continuously in the sup-metric as in the definition of $s$ by Lemma \ref{lemmaExtendedInverse}. 
Thus, $s$ varies continuously. 
\end{proof}

\begin{lemma}
\label{lemmaNearHomeoBoundary}
If $\psi$ is a near-homeomorphism of $\sphere^2$ 
and $A \subset \sphere^2$, 
then $\partial \psi(A) \subseteq \psi(\partial A)$. 
\end{lemma}

\begin{proof}
Consider $y \in \partial \psi(A)$.
Then, 
there are a sequences $a_k\in A$ such that 
$\psi(a_k) \to y$ and 
$z_k \to y$ such that 
$\psi^{-1}(z_k)$ is disjoint from $A$. 
Hence, $z_k\neq \psi(a_k)$. 
Let $b_k \in \psi^{-1}(z_k)$ 
since $\psi$ is homotopic to a homeomorphism 
and therefore surjective.  
By compactness, we may assume $a_k \to a$ and $b_k \to b$, 
so $\psi(a_k) \to \psi(a)$, so $\psi(a)=y$ and similarly $\psi(b)=y$.  
%
Since $\sphere^2$ is locally path connected, 
let $\gamma_k:[0,1] \to \sphere^2$ 
such that $\gamma_k(0) = \psi(a_k)$, $\gamma_k(1) = z_k$ 
and $\gamma_k([0,1]) \to y$ in Hausdorff distance. 
Since $\psi$ is a near-homeomorphism, 
there is $\psi_k \in\hom(\sphere^2)$ such that $\psi_k \to \psi$,  
and we may choose $\psi_k$ so that 
$\psi_k(a_k)=\psi(a_k)$ 
and 
$\psi_k(b_k) = \psi(b_k)$ since $\psi(b_k)\neq\psi(a_k)$. 
Otherwise, compose $\psi_k$ with  
a homeomorphism converging to the identity 
that sends 
$\psi_k(a_k)$ to $\psi(a_k)$ 
and 
$\psi_k(b_k)$ to $\psi(b_k)$. 
Then, $\psi_k^{-1}\gamma_k$ is a curve from $a_k\in A$ to $b_k \not\in A$, 
so there is $t_k$ such that $x_k = \psi_k^{-1}\gamma_k(t_k) \in \partial A$, 
and by compactness we may assume $x_k \to x \in \partial A$.    
Hence, 
$\psi_k(x_k) \to \psi(x)$ 
and 
$\psi_k(x_k) = \gamma_k(t_k) \to y$, 
so $\psi(x) = y$. 
Thus, $y\in \psi(\partial A)$.
\end{proof}

\begin{claim}
\label{claimGebbFacetCell}
$C_{\tau,1,x} = \cell(\Omega_\mr{rec},\tau)$. 
\end{claim}

\begin{proof} 
We have 
$\psi_\mr{com}(1,x)$ is the identity on $Z_L(\tau) \subset \cell(\Omega_\mr{rec},\tau)$, 
by part \ref{itemGebbAccFacet} of Lemma \ref{lemmaGebbStep} and the inductive hypothesis, 
and $C_{\tau,0,0} = \maxcov^{-1}(\mc{N}_{k+1},A,\tau) \supset Z_L(\tau)$ 
since $Z$ accommodates $A$, so 
$Z_L(\tau) \subset C_{\tau,1,x}$.  
Also, $C_{\tau,0,0}$ is disjoint from $\wider(Z_{k+1},i,-\tau)$, 
so $C_{\tau,0,x}$ is disjoint from $\wider(Z_{k+1},i,-\tau)$ 
by Lemma \ref{lemmaGebbStep} part \ref{itemGebbStepWider}, 
so $C_{\tau,0,x} \subset \inner(Z_{k+1},i,-\tau)\cup(Z_{k+1},i,\tau)$, 
so $C_{\tau,1,x} \subseteq S_i^\tau$ 
by part \ref{itemGebbAccInner} and induction. 
Hence, $C_{\tau,1,x} \subseteq \cell(\Omega_\mr{rec},\sigma)$. 
Also, $\partial C_{\tau,0,0}$ consists of arcs along pseudocircles $S_i$ with $i\in I = \supp(\tau)$, which are contained in $\inner(Z_{k+1},i,0)$, so 
\begin{align*}
\partial C_{\tau,1,x} 
&= \partial \psi_\mr{com}(1,x,C_{\tau,0,0}) \\
&\subseteq \psi_\mr{com}(1,x,\partial C_{\tau,0,0}) 
&& \text{by Lemma \ref{lemmaNearHomeoBoundary}} \\
&\subseteq [\psi_\mr{com}(1,x)]\bigcup_{i\in I} \inner(Z_{k+1},i,0) \\  
&\subseteq [\psi_\mr{com}(1,x)]\bigcup_{i\in I} \wider(Z_{k+1},i,0)  \\ 
&\subseteq [\psi_\mr{rec}(1)]\bigcup_{i\in I} \wider(Z_{k+1},i,0) 
&& \text{Lemma \ref{lemmaGebbStep} part \ref{itemGebbStepWider}} \\
&\subseteq [\psi_\mr{rec}(1)]\bigcup_{i\in I} \inner(Z_L,i,0) 
&& Z \text{ is a $(\mc{M},\lesssim)$-zoning}  \\
&\subseteq \bigcup_{i\in I} S_i(\Omega_\mr{rec}) 
&& \text{part \ref{itemGebbAccInner} and induction.} 
\end{align*}
Thus, $Z(\tau)\subset C_{\tau,1,x}\subseteq \cell(\Omega_\mr{rec},\tau)$, 
and $\partial C_{\tau,1,x}\subseteq \partial \cell(\Omega_\mr{rec},\tau)$, 
so $C_{\tau,1,x}= \cell(\Omega_\mr{rec},\tau)$.
\end{proof}

\begin{claim}
\label{claimGebbAccEnd}
If $t=1$, then $s=0$. 
\end{claim}

\begin{proof}
$C_{\tau,1,x} =\cell(\tau,\Omega_\mr{rec})$ is unchanging as $x$ varies 
by Claim \ref{claimGebbFacetCell}, 
so $b_\mr{test}(x)$ is unchanging, 
so 
\begin{gather*}
b_{\mr{test},1}(x) = b_{\mr{test},1}(1) = b_{\mr{test},0}(1) \geq \tailinf(b_{\mr{test},0},0) = 2c, \\ 
b_1(0) = \tailinf(b_{\mr{test},1},0) -c \geq 2c-c = c = b_0(0), 
\end{gather*}
so $s=0$ in the case $t=1$.
\end{proof}

\begin{proof}[Proof of Lemma \ref{lemmaGebbAcc}]
Here we assume that Lemma \ref{lemmaGebbStep} holds, which will also serve as the base case for induction. 
Parts \ref{itemGebbAccStrong}, \ref{itemGebbAccEquivariant}, and \ref{itemGebbAccOmega} hold by induction and the corresponding parts of Lemma \ref{lemmaGebbStep}.

We have $\psi_\mr{new}(s) \in \hom^+(\sphere^2)$ by Lemma \ref{lemmaGebbStep} part \ref{itemGebbAccHomeo} since $s<1$ by Claim \ref{claimGebbAccS}, 
and $\psi_\mr{rec}(t) \in \hom^+(\sphere^2)$ for $t<1$ by part \ref{itemGebbAccHomeo} and induction, 
so $\psi_t = \psi_\mr{rec}(t)\circ\psi_\mr{new}(s) \in \hom^+(\sphere^2)$, 
which means part \ref{itemGebbAccHomeo} holds.

Let us show that $\psi_t$ varies continuously in the sup-metric as $\Omega,Z,A,t$ vary.
We have $s$ varies continuously by Claim \ref{claimGebbAccS}, 
and $Z_{k+1}$ varies continuously since this is a restriction of $Z$, 
so $\psi_\mr{new}(s) = \step(\lesssim,\Omega,Z_{k+1},A,s)$ 
and $\Omega_\mr{rec} = \step(\lesssim,\Omega,Z_{k+1},A,1)*\rest(N_{k+1},A)$ 
vary continuously by Lemma \ref{lemmaGebbStep} part \ref{itemGebbAccContinuous}, 
so $\psi_\mr{rec}(t) = \acc(\Omega_\mr{rec},Z,A,t)$ 
varies continuously by part \ref{itemGebbAccContinuous} of the inductive assumption, 
so $\psi_t$ varies continuously since composition is continuous with respect to the sup-metric, 
which means that part \ref{itemGebbAccContinuous} holds.

Consider the case where $t=0$.  Then, $b_t(0) = b_0(0)$, so $s = 0$, 
so $\psi_\mr{new}(s) = \id$  by Lemma \ref{lemmaGebbStep} part \ref{itemGebbAccZero} 
and $\psi_\mr{rec}(t) = \id$  by part \ref{itemGebbAccZero} and induction, 
so part \ref{itemGebbAccZero} holds.

Let us show that $\psi_1*A \in \pstief(\mc{M})\cap\ext(\Omega)$. 
We have $\Omega_\mr{rec}$ is an extension of $\Omega$ by Lemma \ref{lemmaGebbStep} part \ref{itemGebbAccOne}, so $\ext(\Omega_\mr{rec}) \subset \ext(\Omega)$. 
Also, $\psi_\mr{new}(0) = \id$ by Lemma \ref{lemmaGebbStep} part \ref{itemGebbAccZero}, 
and $s=0$ at $t=1$ by Claim \ref{claimGebbAccEnd}, 
so $\psi_1 = \psi_\mr{rec}(1)$, 
and $\psi_\mr{rec}(1) *A \in \pstief(\mc{M})\cap\ext(\Omega_\mr{rec})$ 
by part \ref{itemGebbAccOne} of the inductive assumption, so 
\[\psi_1 *A = \psi_\mr{rec}(1) *A \in \pstief(\mc{M})\cap\ext(\Omega_\mr{rec}) \subset \pstief(\mc{M})\cap\ext(\Omega).\] 
Now consider the limit as $t\to 1$. 
Let $\eps>0$. 
Since $[0,1]_\mb{R}\times \sphere^2$ is compact, $\psi_\mr{rec}$ is uniformly continuous by part \ref{itemGebbAccContinuous} and induction, 
so there is some $\delta_1$ such that if $\|p-q\|<\delta_1$,  
then 
we have $\|\psi_\mr{rec}(t,p) -\psi_\mr{rec}(t,q)\| < \eps$ for all $t$. 
Also, $s \to 0$ as $t\to 1$ by Claims \ref{claimGebbAccS} and \ref{claimGebbAccEnd}, so $\psi_\mr{new}(s) \to \id$ in the sup-metric by parts \ref{itemGebbAccContinuous} and \ref{itemGebbAccZero}, 
so there is some $\delta_2$ such that if $|t-1|<\delta_2$, 
then $\|\psi_\mr{new}(s;p) -p\| < \delta_1$. 
Hence, $\psi_t(S_i) = \psi_\mr{rec}(t,\psi_\mr{new}(s;S_i))$ 
is less than $\eps$ from $\psi_\mr{rec}(t,S_i)$ in Fréchet distance. 
Also, $\psi_\mr{rec}(t,S_i) \to \psi_\mr{rec}(1,S_i) = \psi_1(S_i)$ by induction, 
so $\psi_\mr{rec}(t,S_i)$ is less than $\eps$ from $\psi_1(S_i)$ in Fréchet distance for $t$ sufficiently close to 1. 
Hence, 
\[
\dist_\mr{F}(\psi_t(S_i),\psi_1(S_i)) 
\leq \dist_\mr{F}(\psi_t(S_i),\psi_\mr{rec}(t,S_i)) +\dist_\mr{F}(\psi_\mr{rec}(t,S_i),\psi_1(S_i))
< 2\eps, 
\]
so $\psi_t*A \to \psi_1*A$, 
so part \ref{itemGebbAccOne} holds.

Let $\sigma$ be a facet covector of $\mc{M}$.
Then, $\tau = \rest(N_{k+1},\sigma)$ is a facet covector of $\mc{N}_{k+1}$, 
\[
Z(\tau) 
= \bigcap_{i\in N_{k+1}} \inner(Z_{k+1},i,\tau) 
\supset \bigcap_{i\in N_{L}} \wider(Z_L,i,\sigma) 
\supset Z(\sigma),
\]
and $\psi_\mr{new}$ is the identity on $Z(\tau)$ by Lemma \ref{lemmaGebbStep} part \ref{itemGebbAccFacet}, so part \ref{itemGebbAccFacet} holds by induction.

Finally, 
$\psi_1 = \psi_\mr{rec}(1)$ by Claim \ref{claimGebbAccEnd}, 
and $S_i(\Omega_\mr{rec}) = S_i(\Omega)$, 
so part \ref{itemGebbAccInner} holds by induction.
\end{proof}

%% file: step-v5-8.tex
\subsection{Zone ebb recursive step}

In this subsection we define 
$\psi_t = \step(\lesssim,\Omega,Z,A,t)$ 
and prove Lemma \ref{lemmaGebbStep}. 
Let $N_1$ be the support of $\Omega$ and $N_2$ be the support of $\mc{M}$. 
Note that these respectively correspond to $N_k$ and $N_L$ in Lemma \ref{lemmaGebbStep}. 
Let us assume that $Z$ is not degenerate; otherwise let $\psi_t = \id$. 
For each edge or facet covector $\sigma$, let 
\[
D_\sigma
= \bigcup_{\Sigma\ni\sigma} Z(\Sigma).
\]
We will define $\psi_t$ as the composition of maps $\psi_\sigma$, which are the identity on the complement of $D_\sigma$.

\subsubsection{Edge covectors}
\label{subsubsectionStepEdgeCovectors}

Let $\sigma$ be an edge covector of $\mc{M}$, 
and $i_0$ be the greedy choice from $\sigma^0$, 
and $\tau_1,\tau_2$ be the facet covectors incident to $\sigma$, 
and $\upsilon_1,\upsilon_2$ be the vertex covectors incident to $\sigma$.
We partition $D_\sigma$ into 3 regions, which we call columns, 
and we define $\psi_\sigma$ in each column.

\begin{definition}[Central column]
\label{defCentralEdgeMap}
Let us start with the central column. 
Let 
\begin{gather*}
C_\sigma = Z(\sigma) \cup Z(\{\sigma,\tau_1\}) \cup Z(\{\sigma,\tau_2\}), \\
P = \{S_i(A)\cap Z(\sigma) : i \in I \} \cup \{P_\mr{N2},P_\mr{N1},P_\mr{S1},P_\mr{S2}\}
\end{gather*}  
where $I$ is the set of nonloops of $\mc{M}$ in $\sigma^0$, 
and 
\begin{align*}
P_\mr{N1} &= Z(\sigma,\{\sigma,\tau_1\}), 
& P_\mr{S1} &= Z(\sigma,\{\sigma,\tau_2\}), \\ 
P_\mr{N2} &= B_\mr{N} = Z(\tau_1,\{\sigma,\tau_1\}), 
& P_\mr{S2} &= B_\mr{S} = Z(\tau_2,\{\sigma,\tau_2\}). 
\end{align*}
Let $B_\mr{E}$ and $B_\mr{W}$ be the closure of the components of $\partial C_\sigma \setminus (B_\mr{N}\cup B_\mr{S})$.
%
Let $h(\mr{N2},\mr{N1},i_0,\mr{S1},\mr{S2}) = (2h,h,0,-h,-2h)$ 
where $h=h(\mr{N1})$ is the distance between $Z$ and $A$ with $A$ regarded as a degenerate zoning as in Definition \ref{defZone}. 
Let 
\[\xi=\xi_\sigma = \mox(C_\sigma, B, P,h).\] 
Let $f_\sigma(t,y)$ scale $y$ on the interval $[-h,h]_\mb{R}$ by $(1-t)$, and keep $y$ at $2h,-2h$ fixed, and interpolate linearly on the rest of $[-2h,2h]_\mb{R}$,  
and let $\psi_t$ be $f(t)$ conjugated by $\xi$. 
That is, 
\begin{align*}
f=f_\sigma(t,y) 
&= \begin{cases} 
(1-t)y & y \in [-h,h]_\mb{R} \\ 
(y-h)2 +(2h-y)(1-t) & y \in [h,2h]_\mb{R} \\
(y+h)2 -(2h+y)(1-t) & y \in [-h,-2h]_\mb{R}
\end{cases} \\[6pt]
\psi_\sigma(t;p) 
&= \xi^{-1}(x,f(t,y))
\end{align*}
for $p \in C_\sigma$ 
where $(x,y) = \xi(p)$. 
\end{definition}

\begin{definition}[Side columns]
\label{defWEEdgeMap}
We define $\psi_\sigma$ in the western and eastern columns to act similarly, 
but to taper to the identity on boundary of $D_\sigma$. 
We choose one of the vertex covectors $\upsilon_1$ to be west of $\sigma$. 
We just define $\psi_\sigma$ formally in the western column $C_{\upsilon_1,\sigma}$; 
the definition in the eastern column is analogous. 
Let 
\begin{gather*}
C_{\upsilon_1,\sigma} = Z(\{\upsilon_1,\sigma\}) \cup Z(\{\upsilon_1,\sigma,\tau_1\}) \cup Z(\{\upsilon_1,\sigma,\tau_2\}), \\
P = \{S_i(A)\cap Z(\upsilon_1,\sigma) : i \in I \} \cup \{P_\mr{N2},P_\mr{N1},P_\mr{S1},P_\mr{S2}\}
\end{gather*}  
with $I$ as above.  
\begin{align*}
P_\mr{N1} &= Z(\{\upsilon_1,\sigma\},\{\upsilon_1,\sigma,\tau_1\}), 
& P_\mr{S1} &= Z(\{\upsilon_1,\sigma\},\{\upsilon_1,\sigma,\tau_2\}), \\ 
P_\mr{N2} &= B_\mr{N} = Z(\tau_1,\{\upsilon_1,\sigma,\tau_1\}), 
& P_\mr{S2} &= B_\mr{S} = Z(\tau_2,\{\upsilon_1,\sigma,\tau_2\}). 
\end{align*}
Let $B_\mr{E}$ and $B_\mr{W}$ be the closure of the components of $\partial C_\sigma \setminus (B_\mr{N}\cup B_\mr{S})$, 
and $h$ be the same as in $\xi_\sigma$. 
Let 
\begin{align*}
\xi = \xi_{\upsilon_1,\sigma} 
&= \mox(C_{\upsilon_1,\sigma},B,P,h), \\
f = f_{\upsilon_1,\sigma}(t,y) 
&= \xi\psi_\sigma(t;\xi^{-1}(1,y))), \\
g = g_{\upsilon_1,\sigma}(t;x,y) 
&= \begin{cases}
(x,f(t,y)) & x \geq \nicefrac12 \\ 
(x,f(2xt,y)) & x < \nicefrac12 
\end{cases} \\
\psi_\sigma(t;p) 
&= \xi^{-1} g(t;\xi(p))
\end{align*}
for $p \in C_{\upsilon_1,\sigma}$. 
In the case where $p \in C_{\upsilon_2,\sigma}$ replace 
$x$ with $1-x$ where appropriate 
in the definition of $g$.

Let 
\begin{align*}
H_{\upsilon_1,\sigma} 
&= \xi_{\upsilon_1,\sigma}^{-1}([0,\nicefrac12]\times[\tfrac{-3h}{2},\tfrac{3h}{2}]). \\[3pt]
H_{\upsilon_2,\sigma} 
&= \xi_{\upsilon_2,\sigma}^{-1}([\nicefrac12,1]\times[\tfrac{-3h}{2},\tfrac{3h}{2}]). \\[3pt]
H_\sigma 
&= (D_\sigma\cap \inner(Z,i_0,0)) \setminus (H_{\upsilon_1,\sigma}\cup H_{\upsilon_2,\sigma})
\end{align*}

\end{definition}

Let $\psi_\sigma(t)$ be defined as above on the columns $C_\sigma$ and $C_{\upsilon,\sigma}$ 
and be the identity on $\sphere^2 \setminus D_\sigma$. 


\begin{claim}
\label{claimEbbSEdgeZero}
$\psi_\sigma(0) = \id$.  
\end{claim}

\begin{proof}
Observe that $f_\sigma(0,y)=y$, 
so we have $\psi_\sigma(0) = \xi_\sigma^{-1}\xi_\sigma=\id$ on $C_\sigma$, 
so $f_{\upsilon,\sigma}(0) = \xi_{\upsilon,\sigma}^{-1}\xi_{\upsilon,\sigma}=\id$, 
so $g(0)=\id$, 
so $\psi_\sigma(0) = \id$ on $C_{\upsilon,\sigma}$. 
\end{proof}

\begin{claim}
\label{claimEbbSEdgeHom}
$\psi_\sigma(t) \in \hom^+(\sphere^2)$ for $t<1$.  
\end{claim}

\begin{proof}
Since $2h$ and $-2h$ are fixed points of $f_\sigma(t)$, 
and the northern and southern boarders of $C_\sigma$ are mapped to horizontal segments at height $2h$ and $-2h$ 
by Lemma \ref{lemmaMox} part \ref{itemMoxH}, 
so $\psi_\sigma$ acts trivially on the northern and southern boarders of $C_\sigma$. 
Also, $\psi_\sigma$ on the eastern and western columns agree with $\psi_\sigma$ on the central column and acts trivially on the northern and southern boarders by definition. 
On the western boarder, we have 
$g(t;0,y) = (0,f(0,y)) = (0,y)$ for $t<1$ 
since $\psi_\sigma(0)=\id$ on the central column $C_\sigma$, 
so $\psi_\sigma$ acts trivially on the western boarder of $D_\sigma$, and similarly on the eastern boarder. 
Hence, the restriction of $\psi_\sigma$ to $D_\sigma$ acts trivially on the boundary and $\psi_\sigma(t)$ keeps points of $D_\sigma$ in $D_\sigma$, so $\psi_\sigma(t)$ is a homeomorphism for $t<1$ by the gluing lemma. 
\end{proof}

\begin{claim}
\label{claimEbbSEdgeContinuous}
$\psi_\sigma$ varies continuously in the sup-metric as $(\Omega,Z,A,t)$ vary, 
as well as $\xi_\sigma$ and $\xi_{\upsilon,\sigma}$ in the partial map metric 
and their respective inputs. 
\end{claim}

\begin{proof}
The zones $Z(\Sigma)$ vary continuously in Hausdorff distance 
and the boarders $B$ vary continuously in Fréchet distance as $Z$ varies by definition of the metric on zonings, 
so $S_i\cap B_\mr{W}$ and $S_i\cap B_\mr{E}$ vary continuously by Lemma \ref{lemmaConvergentCrossingPoint}, 
so the paths of $P$ vary continuously in Fréchet distance, and so $h$ varies continuously. 
Hence, $\xi_\sigma$ and $\xi_{\upsilon_i,\sigma}$ vary continuously in the partial map metric by Lemma \ref{lemmaMox}.  
and so $\xi_\sigma^{-1}$ and $\xi_{\upsilon_i,\sigma}^{-1}$ vary continuously by Lemma \ref{lemmaCPDistContinuity}. 
Also, the maps $f$ and $g$ vary continuously, 
so $\psi_\sigma$ varies continuously by Lemma \ref{lemmaCPDistContinuity}. 
\end{proof}


\begin{claim}
\label{claimEbbSEdgeIsotopy}
Let $i\in\sigma^0$ be a nonloop. 
\begin{enumerate}
\item 
\label{itemEbbSEdgeIsotopy}
The restriction of $\psi_\sigma(t)$ to $S_i(A)\cap H_\sigma$ 
is an isotopy 
to $S_{i_0}(A)\cap H_\sigma$ 
that varies continuously in the sup metric as $(\Omega,Z,A)$ vary. 
\item 
\label{itemEbbSEdgeOne}
$\psi_\sigma(1,H_\sigma) = S_{i_0}(A)\cap H_\sigma$. 
\item 
\label{itemEbbSEdgeSubset}
If $Z$ is generic and $t>0$, then  
$\psi_\sigma(t,H_\sigma) \subset \inner(Z,i,0)^\circ$.  
\end{enumerate}
\end{claim}

\begin{proof}
Let us start with the restriction to the central column. 
There, $\xi_\sigma(P_i)$ varies continuously by Lemma \ref{lemmaPMFrechetDistance} 
since $\xi_\sigma$ and $P_i$ vary continuously by Claim \ref{claimEbbSEdgeContinuous}, 
so $\xi_\sigma(P_i)$ is the graph of a function $\gamma:[0,1]\to\mb{R}$ 
by Lemma \ref{lemmaMox} part \ref{itemMoxMonotone} 
that varies continuously in the sup metric by Lemma \ref{lemmaCPDistSup}, 
and $S_i$ only intersects the central column in $Z(\sigma)$, 
so $|\gamma|\leq h$,  
so $f_\sigma(t)\gamma = (1-t)\gamma$ is an isotopy to $[0,1]\times 0$, 
and $\xi^{-1}(1-t)\gamma$ varies continuously by Lemma \ref{lemmaCPDistContinuity}, 
so $\rest(S_i\cap C_\sigma,\psi_\sigma(t))$ 
is an isotopy to $S_{i_0}\cap C_\sigma$ that varies continuously, 
which means part \ref{itemEbbSEdgeIsotopy} holds on the restriction to $C_\sigma$. 
Likewise, parts \ref{itemEbbSEdgeOne} and \ref{itemEbbSEdgeSubset} hold on the restriction to $C_\sigma$ by a similar argument. 
In particular, the restriction of $\psi_\sigma$ to the border $Z(\sigma,\{\upsilon_1,\sigma\})$ between the central and western columns is a deformation retraction to $S_{i_0}\cap Z(\sigma,\{\upsilon_1,\sigma\})$, 
so $g = g_{\upsilon_1,\sigma}$ restricted to $[\nicefrac12,1]\times[-h,h]$ is a deformation retraction to $[\nicefrac12,1]\times 0$ that preserves the first coordinate, 
so the restriction of $g$ to $\xi_{\upsilon_1,\sigma}(P_i)$ is an isotopy to $[\nicefrac12,1]\times 0$, 
so the restriction of $\psi_\sigma$ to $S_i\cap H_\sigma \cap C_{\upsilon_1,\sigma}$ 
is an isotopy to $S_{i_0}\cap H_\sigma \cap C_{\upsilon_1,\sigma}$ that varies continuously 
like in the case of the central column. 
Hence, part \ref{itemEbbSEdgeIsotopy} holds on the restriction to the western column, 
and a similar argument applies in the eastern column and for parts \ref{itemEbbSEdgeOne} and \ref{itemEbbSEdgeSubset}. 
\end{proof}

\subsubsection{Vertex covectors}

\begin{definition}[Vertex covector deformation]
\label{defEbbStepVertex}
Consider a vertex covector $\upsilon$ of $\mc{M}$. 
To define $\psi_\upsilon$, we interpolate between two cell decompositions as in Figure \ref{figureVertexStepInterp}. 
let $\mc{C}_\upsilon$ be the subdivision of $D_\upsilon$ by 
$\partial Z(\upsilon)$, 
$\Omega$, 
$S_{i_1}(A)$ where $i_1$ is the greedy choice from $\upsilon^0$, 
$\partial Z(\{\upsilon,\sigma\})$ 
for each edge covector $\sigma > \upsilon$ 
and $S_i(A)\cap Z(\{\upsilon,\sigma\})$ 
where $i=i_\sigma$ is the greedy element of $\sigma$, 
and the boundary of a 2-cell $C=C_\upsilon$ as follows; 
$C\subset D_\upsilon$ is the region bounded by the paths 
$E_{\upsilon,\sigma} = \partial H_{\upsilon,\sigma} \cap D_\sigma^\circ$ 
as in Definition \ref{defWEEdgeMap} 
for each edge covector $\sigma > \upsilon$ 
and the hyperbolic geodesic $E_{\upsilon,\tau}$
between endpoints of $E_{\upsilon,\sigma_k}$ and $E_{\upsilon,\sigma_{k+1}}$ 
through 
$Z(\{\upsilon,\tau\})$
for each facet covector $\tau$ 
where $\tau>\{\sigma_{k},\sigma_{k+1}\}>\upsilon$ with addition mod $2K$. 

\begin{figure}
\centering
\includegraphics[scale=1]{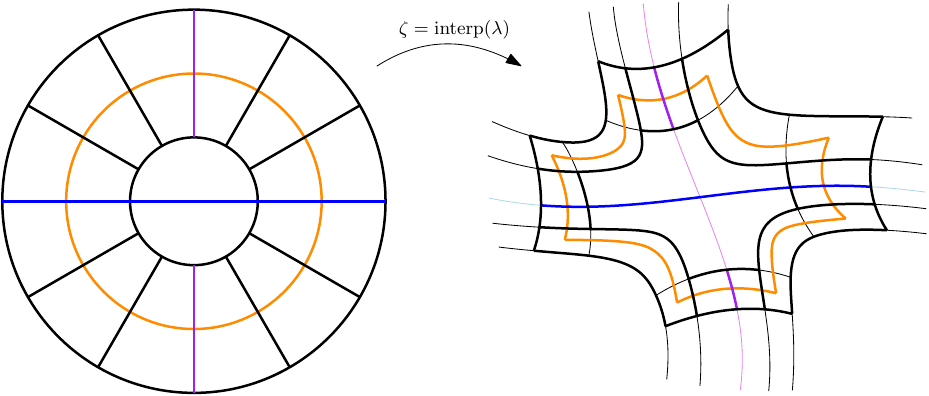}
\caption{$\zeta$ in the definition of $\psi_\upsilon$ where the blue pseudocircle is the greedy choice and the purple pseudocircle is not in $N_1$.}
\label{figureVertexStepInterp}
\end{figure}

Note that 
$S_i(A)\cap Z(\{\upsilon,\sigma\}) = S_i(\Omega)\cap Z(\{\upsilon,\sigma\})$ for $i$ in the ground set of $\Omega$  
and 
$S_{i_1}(A) = S_{i_1}(\Omega)$ 
by the hypotheses of the Lemma \ref{lemmaGebbStep} 
unless the ground set of $\Omega$ is disjoint from $\upsilon^0$. 
 
Let $K$ be the number of parallel classes in $\upsilon^0$. 
Let $\sigma_k,\sigma_{k+K} > \upsilon$ be the edge covector where 
$\sigma_k(i_k) = \sigma_{k+K}(i_k) = 0$ 
and $i_k$ is a greedy element of $(\mc{M},\lesssim)$ 
ordered cyclically according to the rank 2 ordered matroid $\mc{N} = \rest(\upsilon^0,\mc{M})$. 
Let $\mc{C}_{K}$ be the subdivision of the disk $3\disk$ by $\sphere^1$, $2\sphere^1$, 
and $3K$ evenly spaced diameters $L_{k,j}$ corresponding to $i_{k,+},i_{k,0},i_{k,-}$ ordered according to $\mc{N}$, 
but where only the diameters $L_{1,0}$ and $L_{k,0}$ for $i_k \in N_1$ subdivide the unit disk.  

Let $\lambda$ map each cell of $\mc{C}_K$ to the corresponding cell of $\mc{C}_\upsilon$. 
Here we order diameters $L_{k,+},L_{k,0},L_{k,-}$ in $\mc{C}_K$ so that the arc of $\partial Z(\{\upsilon,\sigma_k\})$ 
corresponding to a segment of $L_{k,+}$ is on the positive side of $S_{i_k}$, 
and so the arc of $\partial Z(\{\upsilon,\sigma_{k+K}\})$ 
corresponding to a segment of $L_{k,+}$ is on the negative side of $S_{i_k}$.


and let $\zeta : 3\disk \to D_{\upsilon}$ by $\zeta= \interp(\lambda)$. 
Let $f(t,z)$ scale $z$ in $2\disk$ by $(1-t)$ and keep the circle of radius 3 fixed and interpolate radially on the rest of $3\disk$,  
and let $\psi_\upsilon$ be $f$ conjugated by $\zeta$.
That is, 
\begin{align*}
f = f_\upsilon(t,z) 
&= \begin{cases}
(1-t)z & |z|\leq 2 \\ 
(1-t(3-|z|))z & |z|\geq 2, 
\end{cases} \\[6pt] 
\psi_{\upsilon}(t,p) 
&= \begin{cases}
[\zeta f(t)\zeta^{-1}](p) & p\in D_\upsilon \\ 
p & p \not\in D_\upsilon.  
\end{cases}
\end{align*}
\end{definition}

\begin{remark}
For a nonloop $i\not\in N_1$, we can have vertex covectors $\upsilon_1$ and $\upsilon_2$ where $i$ is the greedy choice from $\upsilon_1^0$ but is not the greedy choice from $\upsilon_2^0$. 
In this case, $\zeta_{\upsilon_1}^{-1}(S_i\cap D_{\upsilon_1})$ is a diameter of $3\disk$, 
but $\zeta_{\upsilon_2}^{-1}(S_i\cap D_{\upsilon_2})$ need not be, 
so $\psi_{\upsilon_1}$ preserves $S_i$, but $\psi_{\upsilon_2}$ can move $S_i$. 
\end{remark}



\begin{claim}
\label{claimEbbSVertexHom}
$\psi_\upsilon(t) \in \hom^+(\sphere^2)$ for $t<1$.  
\end{claim}

\begin{proof}
The restriction of $\psi_\upsilon(t)$ to $D_\upsilon$ is a homeomorphism of $D_\upsilon$ for $t<1$  
by Lemma \ref{lemmaInterp} 
since $f(t)$ is a homeomorphism of $3\disk$. 
Also, $f_\upsilon(t,z)=z$ in the case where $|z|=1$, 
so on the boundary of $D_\upsilon$ we have $\psi_\upsilon(t) = \zeta\zeta^{-1}=\id$ 
so $\psi_\upsilon(t)$ is a homeomorphism by the gluing lemma. 
\end{proof}

\begin{claim}
\label{claimEbbSVertexContinuous}
$\psi_\sigma$ varies continuously in the sup metric as $(\Omega,Z,A,t)$ vary 
\end{claim}

\begin{proof}
The boarders between zones of $Z$ vary continuously as $Z$ varies by the definition of the metric on zonings, 
and $S_i(A)$ and $S_i(\Omega)$ vary continuously in Fréchet distance by the definition of the metric on arrangements, 
and so the vertices where edge of $Z$ meet pseudocircles of $A$ or $\Omega$ 
vary continuously by Lemma \ref{lemmaConvergentCrossingPoint}.
Also, $\xi_{\upsilon,\sigma}^{-1}$ varies continuously in the partial map metric by Claim \ref{claimEbbSEdgeContinuous}, 
so $E_{\upsilon,\sigma}$ varies continuously in Fréchet distance by Lemma \ref{lemmaPMFrechetDistance}, 
and in particular the endpoints of $E_{\upsilon,\sigma}$ vary continuously, 
so the hyperbolic geodesic $E_{\upsilon,\tau}$ varies continuously by Radó's theorem (\ref{theoremRado}). 
Therefore, the edges of $\partial C_\upsilon$ vary continuously in Fréchet distance by Lemma \ref{lemmaConvergentPathSequence} 
since the vertices vary continuously and the edges consist of arcs along paths that vary continuously, 
so the 2-cells vary continuously by Radó's theorem, and $\mc{C}_K$ is fixed, 
so $\zeta = \interp(\lambda)$ varies continuously in the sup metric by Lemma \ref{lemmaInterp}. 
Also, $f$ varies continuously is the sup metric, so $\psi_\upsilon$ varies continuously in the sup metric. 
\end{proof}


\begin{claim}
\label{claimEbbSVertexNoEscape}
$\psi_\sigma$ does not move points out of $D_\upsilon$ or out of $C_\upsilon$. 
\end{claim}

\begin{proof}
If $\upsilon < \sigma$, then $D_\upsilon \cap D_\sigma$ is a side column of $D_\sigma$, and $\psi_\sigma$ does not move points out of the columns. 
Otherwise, $D_\upsilon \cap D_\sigma = \emptyset$, so $\psi_\sigma$ is the identity on $D_\upsilon$. 
Thus, $\psi_\sigma$ does not move points out of $D_\upsilon$. 

If $\upsilon < \sigma$, then $\xi_{\upsilon,\sigma}(C_\upsilon \cap D_\sigma) = H_{\upsilon,\sigma}$, 
which is a coordinate rectangle centered on the x-axis, and $g_{\upsilon,\sigma}(t)$ shrinks the vertical component, 
so $g_{\upsilon,\sigma}(t,H_{\upsilon,\sigma})$, so $\psi_\sigma$ does not move points out of $C_\upsilon \cap D_\sigma$ 
and is the identity on the rest of $C_\upsilon$. Otherwise, $\psi_\sigma$ is the identity on $C_\upsilon$.
Thus, $\psi_\sigma$ does not move points out of $C_\upsilon$. 
\end{proof}

\subsubsection{The composition}

\begin{definition}[Zone ebb recursive step]
\label{defEbbStep}
Let 
\[
\psi_t  
= \step(\lesssim,\Omega,Z,A,t) 
= \prod_{\upsilon \in \Sigma_0} \psi_{\upsilon}(t) \circ \prod_{\sigma \in \Sigma_1} \psi_{\sigma}(t)
\]
where 
$\Sigma_0,\Sigma_1$ are the respective sets of vertex and edge covectors of $\csph(\mc{M})$.  
\end{definition}

\begin{claim}
\label{claimGebbStepDegenerate}
$\step(\lesssim,\Omega,Z,A,t) \to \id$ in the sup metric as $Z \to A$ and $\Omega,A,t$ converge. 
\end{claim}

\begin{proof}
It suffices to show that the restriction of $\psi_t = \sep(\lesssim,\Omega,Z,A,t)$ to $Z(\Sigma)$ for each chain $\Sigma$ converges to the identity map on $Z(\Sigma)$ in the partial map topology by Lemma \ref{lemmaCPDistContinuity} (gluing). 
If $\Sigma = \{\tau\}$ consists of a facet covector, 
then $Z(\Sigma)$ only intersects regions $D_\sigma$ along the boundary for a vertex of edge covector $\sigma$, 
so the restriction of $\psi_t$ to $Z(\Sigma)$ is the identity. 
If $\Sigma \ni \upsilon$ contains a vertex covector, then $Z(\Sigma)$ converges to a point, and the map $\psi_t$ converges to the identity map on point.  
Consider an edge covector $\sigma$. 
Then, $h \to 0$ as $Z \to A$, 
so $f_\sigma$ converges to the identity on the interval $[0,1]$ in the partial map topology, 
and $\xi_\sigma$ converges to a parameterization of the path $P_i = S_i(A)\cap D_\sigma$ from $[0,1]$ 
where $i$ is the greedy element of $\sigma^0$ 
by Lemma \ref{lemmaMox}, 
so $\psi_\sigma$ converges to the identity map on the path $P_i$ 
by Lemma \ref{lemmaCPDistContinuity}, 
and likewise for $Z(\{\sigma,\tau\})$. 
\end{proof}

\begin{claim}
\label{claimGebbStepWider}
$\psi_t(\wider(Z,i,0)) \subseteq \wider(Z,i,0)$.
\end{claim}

\begin{proof}
Consider a point $p \in \wider(Z,i,0)$,  
and 
consider an edge covector $\sigma$.  
Then, $\psi_\sigma(t)$ 
only moves $p \in D_\sigma$. 
In the case where $\sigma(i) = 0$, we have $D_\sigma \subset \wider(Z,i,0)$, 
so $\psi_\sigma(t;p) \in \wider(Z,i,0)$. 
Otherwise, there must be some vertex covector $\upsilon<\sigma$ such that $\upsilon(i)=0$, 
so $p \in Z(\{\upsilon,\sigma\})$, 
so $\psi_\sigma(t;p) \in Z(\{\upsilon,\sigma\}) \subset \wider(Z,i,0)$. 
Hence, $\psi_\sigma(t)$ does not move points out of $\wider(Z,i,0)$. 
Also, $\psi_\upsilon(t)$ 
for a vertex covector $\upsilon$
only moves $p \in D_\upsilon$,
in which case $\sigma(i)=0$ for some chain $\Sigma\supset\{\upsilon,\sigma\}$ since $p \in \wider(Z,i,0)$,
and $\sigma \geq \upsilon$ since $\upsilon$ is a vertex covector, 
so $\upsilon(i)=0$, 
so $D_\upsilon \subset \wider(Z,i,0)$.
Hence, $\psi_\upsilon(t)$ does not move points out of $\wider(Z,i,0)$. 
Thus, $\psi_t$ is a composition of maps that do not move points out of $\wider(Z,i,0)$, 
so $\psi_t$ does not move points out of $\wider(Z,i,0)$.  
\end{proof}

\begin{claim}
\label{claimEbbSFrechet}
Let $i$ be a nonloop of $\upsilon^0$, 
and $i_k$ be the greedy element parallel to $i$. 
Then, 
$\psi_t(S_i\cap D_\upsilon)$
is a 1-cell that 
converges to $\zeta^{-1}(L_{k,0})$ in Fréchet distance as $t\to 1$. 
\end{claim}

\begin{proof}
$\psi_t(S_i\cap D_\upsilon)$ is a 1-cell for $t<1$  
since $\psi_t$ is a homeomorphism for $t<1$, 
so we only have to show this is a 1-cell for $t=1$ and convergence as $t\to 1$. 
Let $X_0 = S_i \cap  H_{\sigma_k} \cap D_\upsilon$ 
and $X_1 = S_{i_k} \cap  H_{\sigma_k} \cap D_\upsilon$, 
and consider the image of $X_0$. 
Then, $\psi_{\widetilde\sigma}(t)$ 
is the identity on $X_0$ unless $\widetilde\sigma\in\{\sigma_k,\upsilon\}$, so 
$\rest(X_0,\psi_t) = \rest(X_0,\psi_\upsilon(t)\circ\psi_{\sigma_k}(t))$ 
and we can disregard the other compositional factors of $\psi_t$. 
The restriction of $\psi_{\sigma_k}$ to $S_i\cap  H_{\sigma_k}$ 
is an isotopy to $S_{i_k}\cap  H_{\sigma_k}$ 
by Claim \ref{claimEbbSEdgeIsotopy}. 
Also, $H_{\sigma_k} \subset D_{\sigma_k}$ and $D_{\sigma_k}\cap D_{\upsilon} = C_{\upsilon,{\sigma_k}}$, 
which is a side column of $D_{\sigma_k}$,  
so $X_0 = S_i \cap  H_{\sigma_k} \cap C_{\upsilon,{\sigma_k}}$ 
and $X_1 = S_{i_k} \cap  H_{\sigma_k} \cap C_{\upsilon,{\sigma_k}}$, 
and so the restriction of $\psi_{\sigma_k}$ to $X_0$ 
is an isotopy to $X_1$ 
since $\psi_{\sigma_k}$ preserves columns of $D_{\sigma_k}$.  
Also, $\psi_\upsilon(t)$ is 2-Lipschitz for all $t$, 
so $\psi_t$ restricted to $X_0$ 
is a homotopy to $Y = \psi_\upsilon(1,X_1)$, 
which is the preimage by $\zeta$ of the radial segment along $L_{k,0}$ corresponding to $\sigma_k$. 
Also, $\psi_\upsilon(t)$ is a homeomorphism except on $C_\upsilon$ and $S_{i_k}\cap H_{\sigma_k}\cap C_\upsilon$ 
consists of a single point, 
so $\psi_t$ restricted to $X_0$ 
is an isotopy to $Y$. 
Therefore, $\psi_t(X_0)$ 
is a 1-cell that 
converges to $Y_1$ in Fréchet distance as $t\to 1$.   
Similarly, $\psi_t(S_i\cap D_\upsilon \cap  H_{\sigma_{k+K}})$ 
is a 1-cell that converges to the preimage by $\zeta$ of the other radial segment of $L_{k,0}$, 
and the rest of the path  
$S_i\cap D_\upsilon$  
is contained in $C_\upsilon$, and as such converges to the point $\zeta^{-1}(0)$. 
Thus, $\psi_t(S_i\cap D_\upsilon)$
converges in Fréchet distance to $\zeta^{-1}(L_{k,0})$, 
and $\psi_1(S_i\cap D_\upsilon) = \zeta^{-1}(L_{k,0})$ a 1-cell. 
\end{proof}

\begin{proof}[Proof of Lemma \ref{lemmaGebbStep}]
Let $\sigma$ be an edge covector and $\upsilon$ be a vertex covector of $\mc{M}$.

For $t<1$, the maps $\psi_\sigma(t)$ and $\psi_\upsilon(t)$ 
are orientation preserving homeomorphisms by Claims \ref{claimEbbSEdgeHom} and \ref{claimEbbSVertexHom}, 
so their composition $\psi_t$ is an orientation preserving homeomorphisms, 
which means part \ref{itemGebbAccHomeo} holds.

Each $\psi_\sigma$ 
varies continuously in the sup metric by Claim \ref{claimEbbSEdgeContinuous}, 
and each $\psi_\upsilon$ 
varies continuously in the sup metric by Claim \ref{claimEbbSVertexContinuous}, 
so the composition $\psi_t$ varies continuously, which means part \ref{itemGebbAccContinuous} holds. 

By Claim \ref{claimEbbSEdgeZero}, $\psi_\sigma(0) =\id$ and $f_\upsilon(0)=\id$ in Definition \ref{defEbbStepVertex}, 
so $\psi_\upsilon(0)=\id$, so $\psi_0=\id$, which means part \ref{itemGebbAccZero} holds. 


We will prove part \ref{itemGebbAccOne} after part \ref{itemGebbAccOmega}.

Suppose $Z$ is degenerate. 
Then, $D_\sigma$ is a path, which means that $D_\sigma^\circ = \empty$, 
and $\psi_\sigma$ is the identity everywhere except possibly in $D_\sigma$, 
so $\psi_\sigma = \id$. 
Also, $D_\upsilon$ is a point, so $\psi_\upsilon=\id$, 
so $\psi_t=\id$, which means part \ref{itemGebbAccStrong} holds. 

The maps $\xi_\sigma$ and $\xi_{\upsilon,\sigma}$ in Definitions \ref{defCentralEdgeMap} and \ref{defWEEdgeMap} are equivariant by Lemma \ref{lemmaMox} part \ref{itemMoxEquivariant}, 
and $\zeta_\upsilon$ is equivariant by Lemma \ref{lemmaInterp} part \ref{itemInterpLEquivariant}, 
so $\psi_t$ is equivariant, which means that part \ref{itemGebbAccEquivariant} holds. 

If $\tau$ is a facet covector of $\mc{M}$, then $Z(\tau)$ is disjoint from each region $D_\sigma$ and $D_\upsilon$, so each map $\psi_\sigma(t)$ and $\psi_\upsilon(t)$ is trivial on $Z(\tau)$, so $\psi_t$ is trivial on $Z(\tau)$, which means that part \ref{itemGebbAccFacet} holds.

Consider a chain $\Sigma \subset \csph(\mc{M})$ and $i\in\bigcap\Sigma^0$ that is either a greedy element in $N_1$ or is the greedy choice from $\bigcap \Sigma^0$. 
Let $S_i = S_i(A)$ when $i$ is the greedy choice and $S_i = S_i(\Omega)$ when $i\in N_1$. 
Note that this is consistent since the hypotheses of the lemma require $S_i(A)=S_i(\Omega)$ if the greedy choice is in $N_1$. 
If $\sigma \in \Sigma$, then $\xi_\sigma$ and $\xi_{\upsilon,\sigma}$ send $S_i \cap D_\sigma$ to the x-axis, which is preserved by $f_\sigma(t)$ and $g_{\upsilon,\sigma}(t)$, so $\psi_\sigma$ does not deform $S_i$. 
If $\upsilon \in \Sigma$, then $\zeta_\upsilon^{-1}$ sends $S_i \cap D_\upsilon$ to a diameter of $3\disk$, 
which is preserved by $f_\upsilon(t)$, 
so $\psi_\upsilon$ does not deform $S_i$. 
If $i \in N_1$, then $S_j(\Omega)=S_i$ for $j \in N_1$ parallel to $i$ since $\Omega \in \psv(\mc{N}_1)$, and $S_i$ only intersects $D_{\widetilde \sigma}$ for $i \in \widetilde \sigma^0$
since $Z$ accommodates $\Omega$, and $\psi_{\widetilde \sigma}$ is trivial on the complement of $D_{\widetilde \sigma}$. 
Hence, $\psi_t$ does not deform $\Omega$, which means part \ref{itemGebbAccOmega} holds.
Also, $Z(\Sigma) =\bigcap_{\widetilde\sigma \in\Sigma} D_{\widetilde \sigma}$, 
so $\psi_t$ does not deform $S_i\cap Z(\Sigma)$, which means part \ref{itemGebbStepGreedy} holds.

For each parallel class $I$ of nonloops, 
each path $\psi_t(S_i\cap D_\upsilon)$ for $i \in I\cap\upsilon^0$ 
converges as $t\to 1$ to the same path in Fréchet distance by Claim \ref{claimEbbSFrechet}. 
Also, each path $\psi_t(S_i\cap D_\sigma \setminus \bigcup \{D_\upsilon : \upsilon < \sigma\})$ 
for $i \in I\cap\sigma^0$
converges to the same path by Claim \ref{claimEbbSEdgeIsotopy}, 
so the pseudocircles $\psi_t(S_i)$ for $i\in I$ converge to a common pseudocircle as $t\to 1$. 
Furthermore, the pseudocircles $\psi_1(S_i)$ for $i \in \upsilon^0$ all intersect at the common point $\zeta_\upsilon^{-1}(0)$
by Claim \ref{claimEbbSFrechet}.  
Also, the top elements of $\csph(\mc{M})$ appear among $\cov(\psi_1*A)$
by part \ref{itemGebbAccFacet}, 
so $\psi_1*A\geq_\mr{w} \mc{M}$,  
and we have just seen that $\psi_1*A$ has the same 
dependencies as $\mc{M}$, 
so $\om(\psi_1*A) = \mc{M}$. 
Also, $\rest(N_1,\psi_1*A) = \Omega$ by part \ref{itemGebbAccOmega}.
Hence, $\psi_t*A \to \psi_1*A \in \psv(\mc{M})\cap\ext(\Omega)$, 
which means part \ref{itemGebbAccOne} holds.

Next we show part \ref{itemGebbAccInner}.
Consider $D_\sigma$ that intersects $\inner(Z,i,0)$ for $i\in N_1$, 
and let $X_\sigma = \inner(Z,i,0) \cap C_\sigma$ 
where $C_\sigma$ is the central column in Definition \ref{defCentralEdgeMap}.  
Then, 
the support $N_1$ of $\Omega$ intersects $\sigma^0$, 
so $i_0 \in N_1$ 
since $i_0$ in Subsubsection \ref{subsubsectionStepEdgeCovectors} is the greedy choice from $\sigma^0$, 
so 
\begin{align*}
\psi_\sigma(1,H_\sigma) 
&= S_{i_0}(A) \cap H_\sigma
&& \text{by Claim \ref{claimEbbSEdgeIsotopy} part \ref{itemEbbSEdgeOne}} \\
&= \psi_\sigma(1,S_{i}(A)) \cap H_\sigma
&& \text{by Claim \ref{claimEbbSEdgeIsotopy} part \ref{itemEbbSEdgeIsotopy}} \\
&= S_i(\Omega) \cap H_\sigma
&& \text{by part \ref{itemGebbAccOne} of the lemma.}
\end{align*}
Hence, 
$\psi_1(X_\sigma) = \psi_\sigma(1,X_\sigma) = S_i(\Omega) \cap X_\sigma$ 
since all other compositional factors in the definition of $\psi_t$ are the identity on $X_\sigma \subseteq H_\sigma$. 

Let $X_\upsilon = \inner(Z,i,0)\cap D_\upsilon$, 
and suppose $X_\upsilon$ is nonempty; otherwise $\psi_\upsilon$ is trivial on $X_\upsilon$.  
Let $i_k$ be the greedy choice from that parallel class of $i$ as in Definition \ref{defEbbStepVertex}. 
Then, $i_k \in N_1$ 
and $X_\upsilon \subset C_\upsilon\cup H_{\sigma_k}\cup H_{\sigma_{k+K}}$. 
Also, $\psi_{\sigma_k}(1,H_{\sigma_k}) \subseteq S_i(\Omega) \cap H_{\sigma_k}$
as above 
and $\psi_{\sigma_k}(1,C_\upsilon) \subseteq C_\upsilon$ by Claim \ref{claimEbbSVertexNoEscape}, 
and $\psi_{\sigma_k}(1)$ is the identity on $H_{\sigma_{k+K}}$,  
and analogously for $\sigma_{k+K}$, 
so 
$\psi_1(X_\upsilon) \subseteq \psi_\upsilon(1,C_\upsilon \cup S_i(\Omega))$.  
Also, $\psi_\upsilon(1,S_i(\Omega)) = S_i(\Omega)$ as shown in the proof of part \ref{itemGebbAccOmega}, 
and $\psi_\upsilon(1,C_\upsilon) = \zeta_\upsilon^{-1}(0) \in S_i(\Omega)$, 
so $\psi_1(X_\upsilon) \subseteq S_i(\Omega)$. 
Hence, $\psi_1(\inner(Z,i,0)) \subseteq  S_i(\Omega)$ since the sets of the form $X_\sigma$ and $X_\upsilon$ cover $\inner(Z,i,0)$. 
Also, $\psi_1(\inner(Z,i,0)) \supseteq \psi_1(S_i(\Omega)) = S_i(\Omega)$ by part \ref{itemGebbAccOmega} of the lemma, 
so $\psi_1(\inner(Z,i,0)) = S_i(\Omega)$, which means part \ref{itemGebbAccInner} holds.

We have already shown part \ref{itemGebbStepGreedy} along with part \ref{itemGebbAccOmega}, so let us show part \ref{itemGebbStepInner}. 
Consider $t>0$ and suppose $Z$ is generic. 
If $D_\sigma$ intersects $\inner(Z,i,0)$, 
then $\psi_\sigma(t,H_\sigma) \subset \inner(Z,i,0)^\circ$ 
by Claim \ref{claimEbbSEdgeIsotopy} part \ref{itemEbbSEdgeSubset}, 
so $\psi_t(X_\sigma) \subset \inner(Z,i,0)^\circ$.  
Now consider $D_\upsilon$ where $X_\upsilon$ is nonempty.  
Then, $W = \zeta^{-1}(X_\upsilon)$ consists of the disk $2\disk$ and an opposite pair of cones 
$C_{\upsilon,k}, C_{\upsilon,k+K}$, which are bounded by the diameters $L_{k,+}$ and $L_{k,-}$, 
where $C_{\upsilon,k}$ is the non-negative span of $\zeta^{-1}(H_{\sigma_k})$ 
and $C_{\upsilon,k+K}$ likewise. 
The only compositional factors of $\psi_t$ that are nontrivial on $H_{\sigma_k}\cap D_\upsilon$ 
are $\psi_\upsilon$ and $\psi_{\sigma_k}$, so 
\begin{align*}
\psi_t(H_{\sigma_k}\cap D_\upsilon) 
&= [\psi_\upsilon(t)\psi_{\sigma_k}(t)](H_{\sigma_k}\cap D_\upsilon) \\
&\subseteq \psi_\upsilon(t,\inner(Z,i,0)^\circ\cap D_\upsilon) 
&& \text{by Claim \ref{claimEbbSVertexNoEscape} and the case for $\psi_\sigma$ above} \\
&= [\zeta f(t)\zeta^{-1}](X_\upsilon^\circ) 
&& \text{in the induced topology on $D_\upsilon$} \\
&= [\zeta f(t)](W_\upsilon^\circ) 
&& \text{in the induced topology on $3\disk$} \\
&\subseteq \zeta(W_\upsilon^\circ) 
\subset \inner(Z,i,0)^\circ
\end{align*}
since 
$W_\upsilon$ is star shaped and  
$f(t)$ moves points radially toward 0.
Also, 
$\psi_t(C_\upsilon) \subseteq \psi_\upsilon(t,C_\upsilon)$ by Claim \ref{claimEbbSVertexNoEscape}, 
so $\psi_t(C_\upsilon) \subset C_\upsilon^\circ$
since $\zeta(2\disk) = C_\upsilon$ and $f(t)$ scales points in $2\disk$ by $1-t$. 
Hence, 
$\psi_t(X_\upsilon) \subset \inner(Z,i,0)^\circ$, 
and since $\inner(Z,i,0)$ is covered by sets of the form $X_\sigma$ and $X_\upsilon$, 
we have $\psi_t(\inner(Z,i,0)) \subset \inner(Z,i,0)^\circ$, 
which means part \ref{itemGebbStepInner} holds.

Consider $p\in \wider(Z,i,0)$ 
and a map $\psi_\sigma$. 
In the case where $p \not\in D_\sigma$, we have $\psi_\sigma(t,p)=p$ is unchanging. 
In the case where $i \in \sigma^0$,  
we have $D_\sigma \subset \wider(Z,i,0)$ so $\psi_\sigma$ does not move points out of $\wider(Z,i,0)$ 
since $\psi_\sigma(t)$ preserves $D_\sigma$ and is the identity outside $D_\sigma$. 
In the case where $p \in D_\sigma$ and $i\not\in\sigma^0$, then $Z(\sigma)$ is not a wider $(i,0)$-zone, 
so $p \in Z(\{\upsilon,\sigma\})$ for a vertex covector $\upsilon$ with $i\in\upsilon^0$, 
so $p$ is in a side column $C_{\upsilon,\sigma} \subset \wider(Z,i,0)$, 
so $\psi_\sigma(t,p) \in C_{\upsilon,\sigma} \subset \wider(Z,i,0)$. 
Hence, the maps $\psi_\sigma$ do not move points out of $\wider(Z,i,0)$. 
If $p \in D_\upsilon$, then $i \in \upsilon^0$, so $D_\upsilon \subset \wider(Z,i,0)$, 
so the maps $\psi_\upsilon$ also do not move points out of $\wider(Z,i,0)$.
Hence $\psi_t(\wider(Z,i,0))\subseteq\wider(Z,i,0)$, which means 
part \ref{itemGebbStepWider} holds. 
\end{proof}